\documentclass[reqno,12pt]{amsart}
\usepackage[colorlinks=true, linkcolor=blue, citecolor=blue]{hyperref}

\usepackage{amssymb}
\usepackage{amsmath, graphicx, rotating}
\usepackage{color}
\usepackage{soul}
\usepackage[dvipsnames]{xcolor}

\allowdisplaybreaks
\usepackage{ifthen}
\usepackage{xkeyval}
\usepackage{todonotes}
\usepackage[T1]{fontenc}
\usepackage{lmodern}
\usepackage[english]{babel}

\usepackage{ upgreek }
\usepackage{stmaryrd}
\SetSymbolFont{stmry}{bold}{U}{stmry}{m}{n}
\usepackage{amsthm}
\usepackage{float}

\usepackage{ bbm }
\usepackage{ stmaryrd }
\usepackage{ mathrsfs }
\usepackage{ frcursive }
\usepackage{ comment }

\usepackage{pgf, tikz}
\usetikzlibrary{shapes}
\usepackage{varioref}
\usepackage{enumitem}

\definecolor{rouge}{rgb}{0.7,0.00,0.00}
\definecolor{vert}{rgb}{0.00,0.5,0.00}
\definecolor{bleu}{rgb}{0.00,0.00,0.8}
\usepackage[margin=1in]{geometry}
\newtheorem{theorem}{Theorem}[section]
\newtheorem*{theorem*}{Theorem}
\newtheorem{lemma}[theorem]{Lemma}

\newtheorem{proposition}[theorem]{Proposition}

\labelformat{hypothesis}{\textbf{M\kern-0.1mm#1}}

\newtheorem{condition}{Condition}

\newtheorem{conditionA}{A\kern-0.1mm}
\labelformat{conditionA}{\textbf{A\kern-0.1mm#1}}

\theoremstyle{definition}

\newtheorem{remark}[theorem]{Remark}

\def \eref#1{\hbox{(\ref{#1})}}

\numberwithin{equation}{section}

\def\geq{\geqslant}
\def\leq{\leqslant}

\def\RR{\mathbb{R}}
\def\PP{\mathbb{P}}
\def\EE{\mathbb{E}}

\def\vare{{\varepsilon}}
\def \eref#1{\hbox{(\ref{#1})}}

\def\EE{\mathbb{ E}}

\begin{document}

\title[Optimal convergence rates for slow-fast SPDEs ]
{Optimal convergence rates in the averaging principle for slow-fast SPDEs driven by multiplicative noise}

\author{Yi Ge}
\curraddr[Ge,Y.]{ School of Mathematics and Statistics and Research Institute of Mathematical Science, Jiangsu Normal University, Xuzhou, 221116, China}
\email{yge@jsnu.edu.cn}

\author{Xiaobin Sun}
\curraddr[Sun, X.]{ School of Mathematics and Statistics and Research Institute of Mathematical Science, Jiangsu Normal University, Xuzhou, 221116, China}
\email{xbsun@jsnu.edu.cn}

\author{Yingchao Xie}
\curraddr[Xie, Y.]{ School of Mathematics and Statistics and Research Institute of Mathematical Science, Jiangsu Normal University, Xuzhou, 221116, China}
\email{ycxie@jsnu.edu.cn}

\begin{abstract}
In this paper, we study a class of slow-fast stochastic partial differential equations
with multiplicative Wiener noise. Under some appropriate conditions, we prove the slow component converges to the solution of the corresponding averaged equation with optimal orders 1/2 and 1 in the strong and weak sense respectively. The main technique is based on the Poisson equation.
\end{abstract}

\date{\today}

\maketitle

\section{Introduction}

\vspace{0.1cm}
In this paper, we study
the following slow-fast stochastic system on a Hilbert space $H$:
\begin{equation}\left\{\begin{array}{l}\label{main equation}
\displaystyle
dX^{\vare}_t=\left[AX^{\vare}_t+F_1(X^{\vare}_t, Y^{\vare}_t)\right]dt+G_1(X^{\vare}_t)dW_t^1,\quad X^{\vare}_0=x\in H,\\
dY^{\vare}_t=\frac{1}{\vare}[AY^{\vare}_t+F_2(X^{\vare}_t, Y^{\vare}_t)]dt+\frac{1}{\sqrt{\varepsilon}}G_2(X^{\vare}_t, Y^{\vare}_t)dW_t^2,\quad Y^{\vare}_0=y\in H,\end{array}\right.
\end{equation}
where $\varepsilon >0$ is a small parameter describing the ratio of time scales between the slow component $X^{\varepsilon}$ and the fast component $Y^{\varepsilon}$, $A$ is a selfadjoint operator, $F_1,F_2, G_1, G_2$ are measurable functions satisfying some appropriate conditions, and $\{W_t^1\}_{t\geq 0}$ and $\{W_t^2\}_{t\geq 0}$ are independent  cylindrical Wiener processes, defined on a complete filtered probability space
$(\Omega,\mathscr{F},\{\mathscr{F}_{t}\}_{t\geq0},\mathbb{P})$.

\vspace{1mm}
The averaging principle for the stochastic system \eref{main equation}
describes the asymptotic behavior of the slow component as the scale parameter $\vare\to 0$, i.e.,
\begin{eqnarray}
X^{\vare} \rightarrow \bar{X} ,\quad \text{as} \quad \vare\rightarrow 0,\label{Convergence}
\end{eqnarray}
in various convergence ways, where $\bar{X}$ is the solution of the corresponding averaged equation. Many people have studied the mode of convergence and the convergence rates of
the stochastic system \eref{main equation}.
The classical method is based on Khasminskii's time discretization, inspired by \cite{K1}.
For instance, Cerrai  \cite{C1} proved \eref{Convergence} holds with convergence in probability;
Fu and Liu \cite{FL} proved \eref{Convergence} holds in the strong sense (see \eqref{ST2} below);
In the case that $G_1(x)\equiv 0$, Br\'{e}hier \cite{B1} proved \eref{Convergence} holds in the strong sense (see \eqref{ST} below) with order $1/2-r$ for any $r\in (0,1/2)$. For further developments, see e.g., \cite{C2,CL,GP3,GP4,PXY,SXX1,WR,XML1}.

\vspace{1mm}
The method based on asymptotic expansion of solutions of Kolmogorov equations in the parameter $\vare$ is  also often used to study the weak averaging principle rate (see \eqref{OWC} below).
For instance, Br\'{e}hier \cite{B1} and Fu et al. \cite{FWLL2018} proved \eref{Convergence} holds in the weak sense with order $1-r$ for any $r\in (0,1)$. For further developments, see e.g., \cite{DSXZ, FWLL,ZFWL} .

\vspace{1mm}
Studying the convergence rate is an interesting and important topic in multiscale systems. For example, the convergence rate can be used to construct the efficient numerical schemes (see e.g., \cite{B2,B3}). It is well known that the optimal strong and weak convergence orders of stochastic system \eref{main equation} should be $1/2$ and $1$ respectively. However, these optimal convergence rates have not
been established in the  papers mentioned  above.

\vspace{0.1cm}
Recently the Poisson equation technique has been successfully used to study the optimal strong and weak convergence rates for the stochastic systems \eref{main equation}.
For instance in the case that the fast component does not depend on the slow component, Br\'{e}hier \cite{B3} proved the strong and weak averaging principle with order $1/2$ and $1$ respectively.
In the fully dependent case with additive noise, R\"{o}ckner et al. \cite{RXY} proved the strong averaging principle holds with order $1/2$. In the fully dependent case  with multiplicative noise, Xie and Yang \cite{XY} proved the weak averaging principle holds with order $1$. However, the references \cite{B3} and \cite{XY} considered the stochastic reaction-diffusion equation and the coefficients are Nemytskii operators with $A$ is the Laplace operator and $H:= L^2(0,1)^d$, these settings are crucial in the $L^p$ theory for the Nemytskii operators (see \cite{BD}).
%However, one key assumption in both \cite{B3,RXY} is that the initial value $x$
%belongs to $H^{\alpha}$ for some $\alpha>0$ (see the detailed definition in Section 2 below), instead of the natural assumption $x\in H$.

\vspace{0.1cm}
The purpose of this paper is to prove the optimal strong and weak convergence orders of stochastic system \eref{main equation} in a general model. More precisely, under some proper assumptions, for any initial value $(x,y)\in H^{\eta}\times H$ with $\eta\in (0,1)$, $T>0$, $p>1$ and small enough $\vare>0$
\begin{eqnarray*}
\sup_{t\in [0,T]}\mathbb{E} |X_{t}^{\vare}-\bar{X}_{t}|^{p}\leq C\vare^{p/2}.
\end{eqnarray*}
and for some nice test function $\phi$,
\begin{eqnarray*}
|\mathbb{E}\phi(X^{\vare}_t)-\EE\phi(\bar{X}_t)|\leq C\vare,
\end{eqnarray*}
where $C$ is a constant depending on $T, \|x\|_{\eta}, |y|, p$ and $\bar{X}$ is the solution of the corresponding averaged equation  (see equation \eref{1.3} below).

Meanwhile under some additional assumptions, we also prove another a stronger result in the strong convergence, i.e.,  for any initial value $(x,y)\in H^{\eta}\times H$ with $\eta\in (0,1)$, $T>0$, $p>1$ and small enough $\vare>0$,
\begin{eqnarray*}
\mathbb{E} \left(\sup_{t\in [0,T]} |X_{t}^{\vare}-\bar{X}_{t}|^{p} \right)\leq C\vare^{p/2}.
\end{eqnarray*}
The above result implies that the supremum in the strong convergence can be taken inside of the expectation, which is new and also give a positive answer to a open question in \cite[Remark 4.9]{B3}.
Our method is based on a careful estimation of the solution of the Poisson equation. We believe that the technique used in this paper is also useful for tackling many other SPDEs.

\vspace{1mm}
The structure of the paper is as follows. In Section 2, we give some notation and our main results. In Section 3, we give some
a priori estimates of the finite dimensional frozen equation and study the corresponding Poisson equation. We prove the strong and weak averaging principle in Sections 4 and  5 respectively.
In the appendix, we give some a priori estimates  of the solutions and the Galerkin approximation.

\vspace{0.1cm}
In this paper, $C$, $C_p$, $C_{T}$ and $C_{p,T}$ stand for positive constants which may change from line to line. $C_p$ means that the constant depends on $p$.
The meanings of $C_{T}$ and $C_{p,T}$ are similar.

\section{Notations and main results} \label{Sec Main Result}

\subsection{Notation and assumptions}

We introduce some notation used throughout this paper.  $H$ is Hilbert space with inner product $\langle\cdot,\cdot\rangle$ and norm $|\cdot|$.
$\mathbb{N}_{+}$ stands for the collection of all the positive integers.

$\mathcal{B}(H)$ denotes the collection of all measurable functions $\varphi(x): H\rightarrow \RR$.
For any $k\in \mathbb{N}_{+}$,
\begin{align*}
C^k(H):=&\{\varphi\in \mathcal{B}(H): \varphi \mbox{ and all its Fr\'{e}chet derivatives up to order } k \mbox{ are continuous}\},\\
C^k_b(H):=&\{\varphi\in C^k(H): \mbox{ for } 1\le i\le k, \mbox{ all $i$-th Fr\'{e}chet derivatives of } \varphi \mbox{ are bounded} \}.
\end{align*}
For any $\varphi\in C^3(H)$, by the Riesz representation theorem, we often identify the first Fr\'{e}chet derivative $D\varphi(x)\in \mathcal{L}(H, \RR)\cong H$, the second derivative $D^2\varphi(x)$ as a linear operator in $\mathcal{L}(H,H)$ and the third derivative $D^3\varphi(x)$ as a linear operator in $\mathcal{L}(H,\mathcal{L}(H,H))$, i.e.,
\begin{eqnarray*}
&&D\varphi(x) \cdot h = \langle D\varphi(x), h \rangle,\quad h\in H,\\
&&D^2\varphi(x) \cdot (h,k) = \langle D^2\varphi(x)\cdot h, k \rangle,  \quad  h,k \in H,\\
&&D^3\varphi(x) \cdot (h,k,l)= [D^3\varphi(x)\cdot h]\cdot (k,l),\quad h,k,l\in H.
\end{eqnarray*}

%For $k_1,k_2\in\mathbb{N}_{+}$ and  measurable function $\varphi(x,y):H\times H\rightarrow \RR$, the notation $\varphi(x,y)\in C^{k_1,k_2}_b(H\times H)$
%means that for all the Fr\'{e}chet derivatives $D_{\underbrace{x\ldots x}_{i~\text{times}} \underbrace{y\ldots y}_{j~\text{times}}}\varphi(x,y)$ ($i$-order and $j$-order Fr\'{e}chet derivatives with respect to $x$ and $y$ respectively) are uniformly bounded, where $0\leq i\leq k_1, 0\leq j\leq k_2, 1\leq i+j\leq \max\{k_1,k_2\}$.
%
%For any $k\in \mathbb{N}_{+}$ and  Hilbert space $V$, $C^k(H,V)$ stands for the collection of all measurable functions
%$\varphi(x): H\rightarrow V$ such that all its Fr\'{e}chet derivatives up to order $k$ are continuous.
%The notation $C^k_b(H,V)$ and $C^{k_1,k_2}_b(H\times H,V)$ can be interpreted similarly.

Let $\mathcal{L}_2(H,H)$ be the space of all linear operators $S:H\to H$ such that
$S$ is a Hilbert-Schmidt operator on $H$.
The norm on $\mathcal{L}_{2}(H,H)$ is defined by
$$|S|^2_{HS}=\text{Tr}(SS^*):=\sum_{k\geq1}|S e_k|^2,$$
where $S^*$ is the adjoint operator of $S$.

$A$ is a selfadjoint operator which satisfies $Ae_n=-\lambda_n e_n$ with $\lambda_n>0$ and $\lambda_n\uparrow \infty$, as $n\uparrow \infty$, where $\{e_n\}_{n\geq1}\subset \mathscr{D}(A)$ is a complete orthonormal basis of $H$. Then for any $s\in\RR$, we define
 $$H^s:=\mathscr{D}((-A)^{s/2}):=\left\{u=\sum_{k\in\mathbb{N}_{+}}u_ke_k: u_k\in \mathbb{R},~\sum_{k\in\mathbb{N}_{+}}\lambda_k^{s}u_k^2<\infty\right\}$$
and
 $$(-A)^{s/2}u:=\sum_{k\in \mathbb{N}_{+}}\lambda_k^{s/2} u_ke_k,~~u\in\mathscr{D}((-A)^{s/2}),$$
with the associated norm $\|u\|_{s}:=|(-A)^{s/2}u|=\left(\sum_{k\in\mathbb{N}_{+}}\lambda_k^{s} u^2_k\right)^{1/2}$. It is easy to see $\|\cdot\|_0=|\cdot|$.

The following smoothing properties of the semigroup $e^{tA}$ (see \cite[Proposition 2.4]{B1}) will be used quite often later in this paper:
\begin{eqnarray}
&&|e^{tA}x|\leq e^{-\lambda_1 t}|x|, \quad x\in H, t\geq 0;\label{SP1}\\
&&\|e^{tA}x\|_{\sigma_2}\leq C_{\sigma_1,\sigma_2}t^{-\frac{\sigma_2-\sigma_1}{2}}e^{-\frac{\lambda_1 t}{2}}\|x\|_{\sigma_1},\quad x\in H^{\sigma_2},\sigma_1\leq\sigma_2, t>0;\label{P3}\\
&&|e^{tA}x-x|\leq C_{\sigma}t^{\frac{\sigma}{2}}\|x\|_{\sigma},\quad x\in H^{\sigma},\sigma>0, t\geq 0.\label{P4}
\end{eqnarray}

Let $\{W_t^1\}_{t\geq 0}$ and $\{W_t^2\}_{t\geq 0}$ be independent cylindrical-Wiener process, i.e,
$$
W_t^1=\sum_{k\in \mathbb{N}_{+}}W^{1,k}_{t}e_k,\quad W_t^2=\sum_{k\in \mathbb{N}_{+}} W^{2,k}_{t}e_k,\quad t\geq 0,
$$
where $\{W^{i,k}\}_{k\in \mathbb{N}_{+}}$ is a sequence of independent one dimensional standard Brownian motions
on the probability space $(\Omega,\mathscr{F},\{\mathscr{F}_{t}\}_{t\geq0},\mathbb{P})$, for $i=1,2$. We also assume that $W^{1,k}$ and $W^{2,k}$ are independent.
\vspace{0.2cm}

Now, we assume the following conditions  on the coefficients $F_1, F_2: H\times H \rightarrow H$, $G_1: H \rightarrow \mathcal L_{2}(H; H)$, $G_2: H\times H \rightarrow \mathcal L_{2}(H; H)$ throughout the paper:
\begin{conditionA}\label{A1}
$F_1,F_2,G_1,G_2$ are Lipschitz continuous, i.e., there exist positive constants $L_{F_2}, L_{G_2}$ and $C$ such that for any $x_1,x_2,y_1,y_2\in H$,
\begin{eqnarray*}
&&\left|F_1(x_1,y_1)-F_1(x_2,y_2)\right|+\left|G_1(x_1)-G_1(x_2)\right|_{HS}\leq C(|x_1-x_2|+|y_1-y_2|),\\
&&\left|F_2(x_1,y_2)-F_2(x_2,y_2)\right|\leq C|x_1-x_2|+L_{F_2}|y_1-y_2|,\\
&&\left|G_2(x_1, y_1)-G_2(x_2, y_2)\right|_{HS}  \leq C|x_1-x_2| + L_{G_2}|y_1-y_2|.
\end{eqnarray*}
\end{conditionA}

\begin{conditionA}\label{A2}
Assume that
 \begin{align}
\theta:=2\lambda_{1}-2L_{F_2}-L^2_{G_2}>0.\label{SDC}
\end{align}
Moreover, there exist two constant $\gamma, \zeta\in(0,1)$, such that for any $x, y\in H$,
\begin{align}
&|(-A)^{1-\frac{\gamma}{2}}G_1(x)|_{HS}\leq C(1+|x|),\label{BOG1}\\
&|G_2(x, y)|_{HS}\leq C(1+|x| + |y|^{\zeta}). \label{BOG2}
\end{align}\end{conditionA}

\begin{conditionA}\label{A3}
Assume that there exists $\delta\in (0,1)$ such that the following directional derivatives are well-defined and satisfy:
\begin{eqnarray}
&&\!\!\!\!\!\!\!\!\!\!\!\!\!\!\!\!\!|D_{x}F(x,y)\cdot h|\leq C |h| \quad \text{and} \quad |D_{y}F(x,y)\cdot h|\leq C |h|,\quad \forall x,y, h\in H ;  \label{FirstDer F}  \\
&&\!\!\!\!\!\!\!\!\!\!\!\!\!\!\!\!\!|D_{xx}F(x,y)\cdot(h,k)|\leq C |h|\|k\|_{\delta},\quad \forall x,y, h\in H, k\in H^{\delta} ;  \label{SecondDer F1}  \\
&&\!\!\!\!\!\!\!\!\!\!\!\!\!\!\!\!\!|D_{yy}F(x,y)\cdot(h,k)|\leq C |h|\|k\|_{\delta},\quad \forall x,y, h\in H, k\in H^{\delta} ;  \label{SecondDer F2}  \\
&&\!\!\!\!\!\!\!\!\!\!\!\!\!\!\!\!\! |D_{xy}F(x,y)\cdot (h,k)|\leq C |h|\|k\|_{\delta}, \quad \forall x,y, h\in H, k\in H^{\delta} ;\label{SecondDer F3}  \\
&&\!\!\!\!\!\!\!\!\!\!\!\!\!\!\!\!\! |D_{xy}F(x,y)\cdot (h,k)|\leq C \|h\|_{\delta}|k|, \quad \forall x,y, k\in H, h\in H^{\delta} ;\label{SecondDer F4}  \\
&&\!\!\!\!\!\!\!\!\!\!\!\!\!\!\!\!\!|D_{xyy}F(x,y)\cdot(h,k,l)|\leq C|h|\|k\|_{\delta}\|l\|_{\delta}, \quad \forall x,y, h\in H, k,l\in H^{\delta} ;  \label{ThirdDer F1}\\
&&\!\!\!\!\!\!\!\!\!\!\!\!\!\!\!\!\!|D_{yyy}F(x,y)\cdot(h,k,l)|\leq C|h|\|k\|_{\delta}\|l\|_{\delta}, \quad \forall x,y, h\in H, k,l\in H^{\delta} ;  \label{ThirdDer F2}\\
&&\!\!\!\!\!\!\!\!\!\!\!\!\!\!\!\!\!|D_{xxy}F(x,y)\cdot(h,k,l)|\leq C|h|\|k\|_{\delta}\|l\|_{\delta}, \quad \forall x,y, h\in H, k,l\in H^{\delta} . \label{ThirdDer F3}
\end{eqnarray}
Here $F$ takes $F_1,F_2,G_1$ and $G_2$. $D_{x}F(x,y)\cdot h$ is the directional derivative of $F(x,y)$ in the direction $h$ with respective to $x$, other notations can be interpreted similarly. Note that if $F$ takes $G_1$, $G_1(x)$ only depends on $x$ and the norm takes $|\cdot|_{HS}$. If $F$ takes $G_2$, the norm takes $|\cdot|_{HS}$.
\end{conditionA}

\begin{remark} For any given $\varepsilon>0$ and initial value $(x,y)\in H\times H$,  the assumption \ref{A1} ensures the system \eref{main equation} admits a unique strong solution $(X^{\vare}_t, Y^{\vare}_t)\in H\times H$, which has the following expression, $\PP$-a.s.,
\begin{equation}\left\{\begin{array}{l}\label{A mild solution}
\displaystyle
X^{\varepsilon}_t=e^{tA}x+\int^t_0e^{(t-s)A}F_1(X^{\varepsilon}_s, Y^{\varepsilon}_s)ds+\int^t_0 e^{(t-s)A}G_1(X^{\varepsilon}_s)dW_s^1,\\
Y^{\varepsilon}_t=e^{tA/\varepsilon}y+\frac{1}{\varepsilon}\int^t_0e^{(t-s)A/\varepsilon}F_2(X^{\varepsilon}_s,Y^{\varepsilon}_s)ds
+\frac{1}{\sqrt{\varepsilon}}\int^t_0 e^{(t-s)A/\varepsilon}G_2(X^{\varepsilon}_s, Y^{\varepsilon}_s)dW_s^2.
\end{array}\right.
\end{equation}
\end{remark}

\begin{remark}
We give some comments on the assumption \ref{A2}:\\
 \begin{itemize}
 \item{ Condition \eref{SDC} is the classical strong dissipative condition, which is used to prove the existence and uniqueness of the invariant measure
and the exponential ergodicity of the transition semigroup of the frozen equation.}
\item { Condition \eref{BOG1}
is used to prove stochastic convolutions $\int^t_0 e^{(t-s)A}G_1(X^{\varepsilon}_s)dW_s^1$ has finite $p$-order moment in $\|\cdot\|_2$ for any $p\geq 2$, i.e.,
$$
\EE\left\|\int^t_0 e^{(t-s)A}G_1(X^{\varepsilon}_s)dW_s^1\right\|^p_2<\infty.
$$
Especially, if $G_1(x)\equiv G_1$, where $G_1$  satisfies $G_1 e_k=\sqrt{\alpha_{1k}}e_k$ with $\alpha_{1k}\geq 0$, then condition \eref{BOG1} can be replaced by condition $\sum^{\infty}_{k=1}\lambda_k\alpha_{1k}<\infty$.}
\item{ Condition \eref{BOG2}
is used to prove the solution has finite $p$-order moment for any $p\geq 2$. Thus the strong averaging principle (see \eref{ST} below) holds for any $p\geq 2$. However if we only want to prove the strong averaging principle holds for some fixed $p\geq 2$, then it is sufficient to show the solution has finite $3p$-th order moment.
Similarly, the weak averaging principle holds  (see \eref{OWC} below)  only needs the solution has finite $3$-th order moment.
Consequently, to prove the strong convergence \eref{ST} for fixed $p\geq 2$, one can remove condition  \eref{BOG2} and replace condition \eref{SDC} by $2\lambda_{1}-2L_{F_2}-(3p-1)L^2_{G_2}>0$; and to prove the weak convergence \eref{OWC}, one can  remove condition \eref{BOG2} and replace condition \eref{SDC} by $\lambda_{1}-L_{F_2}-L^2_{G_2}>0$. Especially, if $G_2(x,y)=\sqrt{Q_2}$ with $Q_2$ being a trace operator, i.e., $\text{Tr}Q_2<\infty$, \eref{BOG2} holds.}
 \end{itemize}
\end{remark}

\begin{remark} \label{EA3}
Here we give a example that the conditions in assumption \ref{A3} hold firstly.
Let $H:=\{g\in L^2(0,1): g(0)=g(1)=0\}$ and let $A:=\partial^2_{\xi}$ be the Laplacian operator.
$F$ is defined to be the Nemytskii operator associated with a function $f:\RR\times \RR\rightarrow \RR$, i.e., $F(x,y)(\xi):=f(x(\xi),y(\xi))$. Then the following directional derivatives are well-defined and belong to $H$,
\begin{eqnarray*}
&&\!\!\!\!\!\!\!\!\!\!\!\!\!\!\!\!\!D_{x}F(x,y)\cdot h=\partial_x f(x,y)h \quad \text{and} \quad D_{x}F(x,y)\cdot h=\partial_y f(x,y)h ,\quad \forall x,y, h\in H ;  \label{EFirstDer F}  \\
&&\!\!\!\!\!\!\!\!\!\!\!\!\!\!\!\!\!D_{xx}F(x,y)\cdot(h,k)=\partial_{xx}f(x,y)hk\in H,\quad \forall x,y, h, k\in L^{\infty}(0,1) ;  \label{ESecondDer F1}  \\
&&\!\!\!\!\!\!\!\!\!\!\!\!\!\!\!\!\!D_{yy}F(x,y)\cdot(h,k)=\partial_{yy}f(x,y)hk,\quad \forall x,y, h\in H, k\in L^{\infty}(0,1) ;  \label{ESecondDer F2}  \\
&&\!\!\!\!\!\!\!\!\!\!\!\!\!\!\!\!\! D_{xy}F(x,y)\cdot (h,k)=\partial_{xy}f(x,y)hk, \quad \forall x,y, h\in H, k\in L^{\infty}(0,1) ;\label{ESecondDer F3}  \\
&&\!\!\!\!\!\!\!\!\!\!\!\!\!\!\!\!\!D_{xyy}F(x,y)\cdot(h,k,l)=\partial_{xyy}f(x,y)hkl, \quad \forall x,y, h\in H, k,l\in L^{\infty}(0,1) ;  \label{EThirdDer F1}\\
&&\!\!\!\!\!\!\!\!\!\!\!\!\!\!\!\!\!D_{yyy}F(x,y)\cdot(h,k,l)=\partial_{yyy}f(x,y)hkl, \quad \forall x,y, h\in H, k,l\in L^{\infty}(0,1) ;  \label{EThirdDer F2}\\
&&\!\!\!\!\!\!\!\!\!\!\!\!\!\!\!\!\!D_{xxy}F(x,y)\cdot(h,k,l)=\partial_{xxy}f(x,y)hkl, \quad \forall x,y, h\in H, k,l\in L^{\infty}(0,1), \label{EThirdDer F3}
\end{eqnarray*}
where we assume that all the partial derivatives of $f$ appear above are all uniformly bounded. Note that $H^{\delta}\subset L^{\infty}(0,1)$ for any $\delta>1/2$, thus it is easy to see that $F$ satisfy assumption \ref{A3} with any $\delta\in (1/2,1)$. Secondly, if the noises in slow and fast components are  additive, then the existing constant $\delta\in (0,1)$ can be extend to $\delta\in (0,2)$, which covers more examples.
\end{remark}

\subsection{Main results}

Let  $\mu^{x}$ be the unique invariant measure of the transition semigroup of the frozen equation
\begin{eqnarray}\label{FZE}
\left\{ \begin{aligned}
&dY_{t}=\left[AY_{t}+F_2(x,Y_{t})\right]dt+G_2(x,Y_{t})d W_{t}^2,\\
&Y_{0}=y\in H,
\end{aligned} \right.
\end{eqnarray}
and define $\bar{F}_1(x):=\int_{H}F_1(x,y)\mu^{x}(dy)$. Let $\bar{X}$ is the solution of the corresponding averaged equation:
\begin{equation}\left\{\begin{array}{l}
\displaystyle d\bar{X}_{t}=\left[A\bar{X}_{t}+\bar{F_1}(\bar{X}_{t})\right]dt+G_1(\bar{X}_{t})d W_{t}^1,\\
\bar{X}_{0}=x\in H.\end{array}\right. \label{1.3}
\end{equation}
Here is our first main result.

\begin{theorem} (\textbf{Strong convergence 1})\label{main result 1}
Suppose  that \ref{A1}-\ref{A3} hold. Then for any initial value $(x,y)\in H^{\eta}\times H$ with $\eta\in (0,1)$, $T>0$, $p>1$ and small enough $\vare>0$, there exists a positive constant $C_{p,T}$ depending only on $p$ and $T$ such that
\begin{equation}\label{ST}
\sup_{t\in [0,T]}\mathbb{E} |X_{t}^{\vare}-\bar{X}_{t}|^{p}\leq C_{p,T}(1+\|x\|^{3p}_{\eta}+|y|^{3p})\vare^{p/2}.
\end{equation}
\end{theorem}

\begin{remark}
The theorem above implies that the convergence order is $1/2$, which is the optimal strong convergence order (see \cite[Example 1]{LD}). Compared with \cite{B1}, we allow
 multiplicative noise in the slow component and we obtain the optimal strong convergence order 1/2.  In contrast to \cite{B3,RXY},  we are dealing with the case of multiplicative noise, hence some methods developed there do not work in this situation.
\end{remark}

It is easy to find that the supremum in strong convergence \eref{ST} is outsider, next we can prove a stronger result, i.e., the supremum can be taken insider. In order to do this, we need the following assumption:
\begin{conditionA}\label{A4}
Suppose that
\begin{eqnarray}
\sup_{x,y\in H}|G_2(x,y)|_{HS}<\infty,\quad \sup_{x\in H}|F_2(x,0)|<\infty.\label{CFSUP}
\end{eqnarray}
Moreover, there exist two constants $\alpha\in (0,2]$ and $\beta\in (0,1)$ such that
$F_1: H^{\beta}\times H^{\beta}\rightarrow H^{\alpha}$, and for any $x,y,y_1,y_2\in H^{\beta}$,
\begin{eqnarray}
\|F_1(x, y_1)-F_1(x, y_2)\|_{\alpha}\leq C(1+\|x\|_{\beta}+\|y_1\|_{\beta}+\|y_2\|_{\beta})\|y_1-y_2\|_{\beta}.\label{BH}
\end{eqnarray}
\end{conditionA}
\begin{theorem} (\textbf{Strong convergence 2})\label{main result 12}
Suppose  that \ref{A1}-\ref{A4} hold. Then for any initial value $(x,y)\in H^{\eta}\times H$ with $\eta\in (0,1)$, $T>0$, $p>1$ and small enough $\vare>0$, there exists a positive constant $C_{p,T}$ depending only on $p$ and $T$ such that
\begin{align}
\mathbb{E} \left(\sup_{t\in [0,T]} |X_{t}^{\vare}-\bar{X}_{t}|^{p} \right)\leq C_{p,T}(1+\|x\|^{3p}_{\eta}+|y|^{3p})\vare^{p/2}. \label{ST2}
\end{align}
\end{theorem}
\begin{remark}
The result \eref{ST2} for the stochastic system \eref{main equation} is new and we have not seen results in the literature, which also give a positive answer to a open question in \cite[Remark 4.9]{B2}.
\end{remark}
\vspace{0.3cm}

To obtain the optimal weak convergence order, we need the following assumption:

\begin{conditionA}\label{A5}
Suppose that the noise in slow component is additive, i.e., $G_1(x)\equiv G_1$, where $G_1\in L_{2}(H; H)$ and satisfies $G_1 e_k=\sqrt{\alpha_{1k}}e_k$ with $\alpha_{1k}\geq 0$ and $\sum^{\infty}_{k=1}\lambda_k\alpha_{1k}<\infty$. Moreover, condition \eref{BH} holds and there exists $C>0$ such that for any $x,y,y_1,y_2\in H^{\beta}$,
\begin{eqnarray}
&&\|F_1(x,y)\|_{\alpha}\leq C(1+\|x\|_{\beta}+\|y\|_{\beta}).\label{BF}
\end{eqnarray}
Meanwhile, there exists $\kappa\in (0,1)$ such that the following directional derivatives are well-defined and satisfy:
\begin{eqnarray}
&&\!\!\!\!\!\!\!\!\!\!\!\!\!\!\!\!\!|D_{xxx}F_1(x,y)\cdot(h,k,l)|\leq C|h|\|k\|_{\kappa}\|l\|_{\kappa}, \quad \forall x,y, h\in H, k,l\in H^{\kappa}, \label{ThirdDer F4}\\
&&\!\!\!\!\!\!\!\!\!\!\!\!\!\!\!\!\!|D_{xxx}F_2(x,y)\cdot(h,k,l)|\leq C|h|\|k\|_{\kappa}\|l\|_{\kappa}, \quad \forall x,y, h\in H, k,l\in H^{\kappa} ;  \label{ThirdDer F5}\\
%&&\!\!\!\!\!\!\!\!\!\!\!\!\!\!\!\!\!|D_{xxx}G_1(x)\cdot(h,k,l)|_{HS}\leq C|h|\|k\|_{\kappa}\|l\|_{\kappa}, \quad \forall x, h\in H, k,l\in H^{\kappa} ;  \label{ThirdDer G1}\\
&&\!\!\!\!\!\!\!\!\!\!\!\!\!\!\!\!\!|D_{xxx}G_2(x,y)\cdot(h,k,l)|_{HS}\leq C|h|\|k\|_{\kappa}\|l\|_{\kappa}, \quad \forall x,y,h\in H, k,l\in H^{\kappa} .  \label{ThirdDer G2}
\end{eqnarray}
\end{conditionA}
\begin{remark}
Recall the notations in Remark \ref{EA3}, i. e., $H=\{g\in L^2(0,1): g(0)=g(1)=0\}$ and let $A:=\partial^2_{\xi}$ be the Laplacian operator.
$F(x,y)(\xi):=f(x(\xi),y(\xi))$.  Assume that $\partial_y f(\cdot,\cdot):\RR\times \RR\rightarrow \RR$ is a global Lipschitz function, then it follows from \cite[Proposition 2.1]{B3} or \cite[Section 3.2]{BD} that for any $\alpha\in (0,1)$, $\tau\in (0, 1-\alpha)$,
\begin{eqnarray*}
\|(-A)^{\alpha/2}\left[f(x,y_1)-f(x,y_2)\right]\|_{L^2}\leq\!\!\!\!\!\!\!\!&&C\Big(1+\|(-A)^{\frac{\alpha+\tau}{2}}x\|_{L^4}+\|(-A)^{\frac{\alpha+\tau}{2}}y_1\|_{L^4}\nonumber\\
&&\quad+\|(-A)^{\frac{\alpha+\tau}{2}}y_2\|_{L^4}\Big)\|(-A)^{\frac{\alpha+\tau}{2}}(x-y)\|_{L^4},
\end{eqnarray*}
where $\|g\|_{L^p}:=\left[\int^1_0 |g(\xi)|^p d\xi \right]^{1/p}$. Then by Sobolev inequality $\|x\|_{L^4}\leq C\|x\|_{\frac{1}{4}}$, we have
\begin{eqnarray*}
\|F(x,y_1)-F(x,y_2)\|_{\alpha}\leq\!\!\!\!\!\!\!\!&& C\Big(1+\|x\|_{\alpha+\tau+1/4}\nonumber\\
&&\quad+\|y_1\|_{\alpha+\tau+1/4}+\|y_2\|_{\alpha+\tau+\frac{1}{4}}\Big)\|x-y\|_{\alpha+\tau+\frac{1}{4}}.
\end{eqnarray*}
Furthermore, if $f(\cdot,\cdot):\RR\times \RR\rightarrow \RR$ is a global Lipschitz function, then it follows
\begin{eqnarray*}
\|F(x,y)\|_{\alpha}\leq\!\!\!\!\!\!\!\!&&C\Big(1+\|(-A)^{\frac{\alpha+\tau}{2}}x\|_{L^2}+\|(-A)^{\frac{\alpha+\tau}{2}}y\|_{L^2}\Big)\nonumber\\
\leq\!\!\!\!\!\!\!\!&&C\Big(1+\|x\|_{\alpha+\tau}+\|y\|_{\alpha+\tau}\Big).
\end{eqnarray*}
Consequently, $F$ satisfies the conditions \eref{BH} and \eref{BF} for any $\alpha\in (0,3/4)$ and $\beta=\alpha+\tau+1/4$ with $\tau<3/4-\alpha$.

Assume that $\sup_{x,y\in \RR}|\partial_{xxx}f(x,y)|<\infty$. Note that $H^{\kappa}\subset L^{\infty}(0,1)$ for any $\kappa>1/2$, then for any $x,y, h\in H$ and $k,l \in H^{\kappa} $,  we have
$$
|D_{xxx}F(x,y)\cdot(h,k,l)|=|\partial_{xxx}f(x,y)hkl|\leq C|h|\|k\|_{\kappa}\|l\|_{\kappa}.
$$
\end{remark}

The following is our third result.

\begin{theorem}(\textbf{Weak convergence})\label{main result 2}
Suppose that \ref{A1}-\ref{A3}, \ref{A5} and  \eref{BH} hold. Then for any test function $\phi\in C^3_b(H)$, initial value $(x,y)\in H^{\eta}\times H$ with $\eta\in (0,1)$, $T>0$ and $\vare>0$, there exists a positive constant $C$ depending only on $T$ such that
\begin{align}
\sup_{t\in [0,T]}\left|\mathbb{E}\phi(X_{t}^{\vare})-\EE\phi(\bar{X}_{t})\right|\leq C_{T}(1+\|x\|^3_{\eta}+|y|^3)\vare. \label{OWC}
\end{align}
\end{theorem}

\begin{remark}
The theorem above implies that the convergence order is $1$, which is the optimal weak convergence order. In contrast to \cite{B1}, we allow a multiplicative noise in the slow component and we achieve  the optimal weak convergence order 1. Compared with \cite{B3, XY}, we here consider a general model for stochastic system \eref{main equation}.
\end{remark}

\subsection{Idea of proof}
For the convenience of our reader, we sketch the main idea of the proof of  the strong convergence 1. The  strong convergence 2 and weak convergence can be handled similarly.
\vspace{0.05cm}

$\textbf{Step 1 (Galerkin approximation):}$ Since the operator $A$ is an unbounded operator, we use the Galerkin approximation to reduce the infinite dimensional problem to a finite dimensional one, i.e.,  we consider
\begin{equation}\left\{\begin{array}{l}\label{Ga mainE}
\displaystyle
dX^{m,\vare}_t=[AX^{m,\vare}_t+F^m_1(X^{m,\vare}_t, Y^{m,\vare}_t)]dt+G^m_1(X^{m,\vare}_t)d\bar W^{1,m}_t,\  X^{m,\vare}_0=x^{m}\in H_m,\\
dY^{m,\vare}_t=\frac{1}{\vare}[AY^{m,\vare}_t+F^m_2(X^{m,\vare}_t, Y^{m,\vare}_t)]dt+\frac{1}{\sqrt{\varepsilon}}G^m_2(X^{m,\vare}_t, Y^{m,\vare}_t)d\bar W^{2,m}_t,\quad Y^{m,\vare}_0=y^m\in H_m,\end{array}\right.
\end{equation}
where $m\in \mathbb{N}_{+}$, $H_{m}:=\text{span}\{e_{k}:1\leq k \leq m\}$, $\pi_{m}$ is the orthogonal projection of $H$ onto $H_{m}$, $x^{m}:=\pi_m x, y^m:=\pi_m y$ and
\begin{eqnarray*}
&& F^{m}_1(x,y):=\pi_m F_1(x,y),\quad F^{m}_2(x,y):=\pi_m F_2(x, y),\\
&& G^{m}_1(x):=\pi_m G_1(x),\quad G^{m}_2(x,y):=\pi_m G_2(x, y),\\
&&\bar W^{1,m}_t:=\sum^m_{k=1}W^{1,k}_{t}e_k,\quad \bar W^{2,m}_t:=\sum^m_{k=1}W^{2,k}_{t}e_k.
\end{eqnarray*}
Similarly, we consider the following approximation to the averaged equation \eref{1.3}:
\begin{equation}\left\{\begin{array}{l}
\displaystyle d\bar{X}^m_{t}=\left[A\bar{X}^m_{t}+\bar{F}^m_1(\bar{X}^m_{t})\right]dt+G^m_1(\bar{X}^m_{t})d \bar{W}^{1,m}_{t},\\
\bar{X}^m_{0}=x^m,\end{array}\right. \label{Ga 1.3}
\end{equation}
where $\bar{F}^m_1(x):=\int_{H_m}F^m_1(x,y)\mu^{x,m}(dy)$, and $\mu^{x,m}$ is the unique invariant measure of the transition semigroup of the following frozen equation:
\begin{equation}\left\{\begin{array}{l}
\displaystyle dY^{x,y,m}_t=[AY^{x,y,m}_t+F^m_2(x,Y^{x,y,m}_t)]dt+G^m_2(x,Y^{x,y,m}_t)d\bar{W}^{2,m}_t,\\
Y^{x,y,m}_0=y\in H_m.\end{array}\right. \label{Ga FZE}
\end{equation}

 It is easy to see that for any $T>0$, $p\geq 1$ and $m\in \mathbb{N}_{+}$,
\begin{eqnarray*}
\sup_{t\in [0, T]}\EE|X_{t}^{\vare}-\bar{X}_{t}|^p\leq\!\!\!\!\!\!\!\!&&C_p\sup_{t\in [0, T]}\EE|X_{t}^{m,\vare}-X_{t}^{\vare}|^p+C_p\sup_{t\in [0, T]}\EE|\bar{X}^m_{t}-\bar{X}_{t}|^p\\
&&+C_p\sup_{t\in [0, T]}\EE|X_{t}^{m,\vare}-\bar{X}^m_{t}|^p.
\end{eqnarray*}
Using the following two approximations (see Lemmas \ref{GA1} and \ref{GA2} in the Appendix), we get
\begin{eqnarray*}
\lim_{m\rightarrow\infty}\sup_{t\in [0, T]}\EE|X_{t}^{m,\vare}-X_{t}^{\vare}|^p=0,\quad \lim_{m\rightarrow\infty}\sup_{t\in [0, T]}\EE|\bar{X}^m_{t}-\bar{X}_{t}|^p=0.
\end{eqnarray*}
Then the proof will be complete if we show that there exists a positive constant $C$ independent of $m$ such that
$$
\sup_{t\in [0, T]}\EE|X_{t}^{m,\vare}-\bar{X}^m_{t}|^p\leq C\vare^{p/2},
$$
which will be proved by using the  Poisson equation in Step 2.

\vspace{0.1cm}
$\textbf{Step 2 (Technique of Poisson equation):}$
Using the formulation of the mild solutions $X_{t}^{m,\vare}$ and $\bar{X}^{m}_{t}$, we have for any $t>0$,
\begin{eqnarray*}
\!X_{t}^{m,\vare}\!-\!\bar{X}^{m}_{t}\!\!\!=\!\!\!\!\!\!\!\!&&\int_{0}^{t}\!\!e^{(t-s)A}\!\left[F^m_1(X_{s}^{m,\vare},Y_{s}^{m,\vare})\!-\!\bar{F}^m_1(\bar{X}^{m}_{s})\right]\!ds\!+\!\int_{0}^{t}\!\!e^{(t-s)A}\!\left[G^m_1(X_{s}^{m,\vare})\!-\!G^m_1(\bar{X}^{m}_{s})\right]\!d\bar{W}_s^{1,m}\\
=\!\!\!\!\!\!\!\!&&\int_{0}^{t}\!\!\!e^{(t-s)A}\!\!\left[F^m_1( X^{m,\vare}_{s},Y_{s}^{m,\vare})\!-\!\bar{F}^m_1(X^{m,\vare}_{s})\right]ds+\!\int_{0}^{t}\!\!\!e^{(t-s)A}\!\!\left[\bar{F}^m_1(X^{m,\vare}_{s})\!-\!\bar{F}^m_1(\bar{X}^m_{s})\right]ds\\
\!\!\!\!\!\!\!\!&&+\int_{0}^{t}e^{(t-s)A}\left[G^m_1(X_{s}^{m,\vare})-G^m_1(\bar{X}^{m}_{s})\right]d\bar{W}_s^{1,m}.
\end{eqnarray*}
%The averaged coefficient $\bar{F}^m_1$ will be proved that it is Lipschitz  (see \eref{BL} below), i.e.,
In \eref{BL} below, we will show that the averaged coefficient $\bar{F}^m_1$ is Lipschitz,  i.e.,
$$
|\bar{F}^m_1(x)-\bar{F}^m_1(y)|\leq C|x-y|,
$$
where $C$ is a positive constant independent of $m$.  Then we get for any $T>0$,
\begin{eqnarray*}
\sup_{t\in[0,T]}\EE|X_{t}^{m,\vare}-\bar{X}^m_{t}|^p\leq\!\!\!\!\!\!\!\!&&C_p\sup_{t\in[0,T]}\EE\left|\int_{0}^{t}e^{(t-s)A}\left[F^m_1(X^{m,\vare}_{s},Y_{s}^{m,\vare})-\bar{F}^m_1(X^{m,\vare}_{s})\right]ds\right|^p\\
&&+C_{p,T}\EE\int_{0}^{T}|X_{t}^{m,\vare}-\bar{X}^m_{t}|^p dt .
\end{eqnarray*}
Then it follows from Gronwall's inequality that
\begin{eqnarray*}
\sup_{t\in[0,T]}\EE|X_{t}^{m,\vare}-\bar{X}^m_{t}|^p\leq\!\!\!\!\!\!\!\!&&C_{p,T}\sup_{t\in[0,T]}\EE\left|\int_{0}^{t}e^{(t-s)A}\left[F^m_1(X^{m,\vare}_{s},Y_{s}^{m,\vare})-\bar{F}^m_1(X^{m,\vare}_{s})\right]ds\right|^p, \label{ID1}
\end{eqnarray*}

To estimate the right-hand side above,
we consider the following Poisson equation:
\begin{equation}
-\mathscr{L}^m_{2}(x)\Phi_m(x,y)=F^m_1(x,y)-\bar{F}^m_1(x),\quad x, y\in H_m.\label{ID2}
\end{equation}
where $\mathscr{L}^m_{2}(x)$ is the infinitesimal generator of the transition semigroup of the frozen equation \eref{Ga FZE} (see \eref{L_2} below).

If $\Phi_m$ is a sufficiently regular solution of equation \eref{ID2} (see Proposition \ref{P3.6} below), applying It\^o's formula we get
\begin{eqnarray}
\Phi_m(X^{m,\vare}_{t},Y^{m,\vare}_{t})=\!\!\!\!\!\!\!\!&&e^{tA}\Phi_m(x^m,y^m)+\int^t_0 (-A)e^{(t-s)A}\Phi_m(X^{m,\vare}_{s},Y^{m,\vare}_{s})ds\nonumber\\
&&+\int^t_0 e^{(t-s)A}\mathscr{L}^m_{1}(Y^{m,\vare}_{s})\Phi_m(X^{m,\vare}_{s},Y^{m,\vare}_{s})ds\nonumber\\
&&+\frac{1}{\vare}\int^t_0 e^{(t-s)A}\mathscr{L}^m_{2}(X_{s}^{m,\vare})\Phi_m(X^{m,\vare}_{s},Y^{m,\vare}_{s})ds+M^{m,\vare,1}_{t}+M^{m,\vare,2}_{t},\label{ID3}
\end{eqnarray}
where $\mathscr{L}^m_{1}(y)$ is the infinitesimal generator of the transition semigroup for the slow component $X^{\vare}$ with  fixed fast component $y$
(see \eref{L^m_1} below), $M^{m,\vare,1}_{t}$ and $M^{m,\vare,2}_{t}$ are defined in \eref{M_1} and \eref{M_2} below.

As a result, it follows
\begin{eqnarray}
&&\int^t_0 -e^{(t-s)A}\mathscr{L}^m_{2}(X_{s}^{m,\vare})\Phi_m(X^{m,\vare}_{s},Y^{m,\vare}_{s}) ds\nonumber\\
=\!\!\!\!\!\!\!\!&&\vare\Big[e^{tA}\Phi_m(x^m,y^m)-\Phi_m(X_{t}^{m,\vare},Y^{m,\vare}_{t})+\int^t_0 (-A)e^{(t-s)A}\Phi_m(X^{m,\vare}_{s},Y^{m,\vare}_{s})ds\nonumber\\
&&+\int^t_0 \!\!e^{(t-s)A}\mathscr{L}^m_{1}(Y^{m,\vare}_{s})\Phi_m(X^{m,\vare}_{s},Y^{m,\vare}_{s})ds+M^{m,\vare,1}_{t}+M^{m,\vare,2}_{t}\Big].\label{ID4}
\end{eqnarray}
By \eref{ID2}-\eref{ID4},  we get
\begin{eqnarray}
\sup_{t\in [0, T]}\EE|X_{t}^{m,\vare}-\bar{X}^m_{t}|^p\leq\!\!\!\!\!\!\!\!&&C_{p,T}\sup_{t\in[0,T]}\EE\left|\int_{0}^{t}-e^{(t-s)A}\mathscr{L}^m_{2}(X_{s}^{m,\vare})\Phi_m(X_{s}^{m,\vare},Y^{m,\vare}_{s})ds\right|^p\nonumber\\
\leq\!\!\!\!\!\!\!\!&&C_{p,T}\vare^{p}\Bigg\{\sup_{t\in[0,T]}\EE|e^{tA}\Phi_m(x^m,y^m)-\Phi_m(X_{t}^{m,\vare},Y^{m,\vare}_{t})|^p\nonumber\\
&&+\sup_{t\in[0,T]}\EE\left|\int^t_0 (-A)e^{(t-s)A}\Phi_m(X^{m,\vare}_{s},Y^{m,\vare}_{s})ds\right|^p\nonumber\\
&&+\sup_{t\in[0,T]}\EE\left|\int^t_0 e^{(t-s)A}\mathscr{L}^m_{1}(Y^{m,\vare}_{s})\Phi_m(X_{s}^{m,\vare},Y^{m,\vare}_{s})ds\right|^p\nonumber\\
&&+\sup_{t\in[0,T]}\EE|M^{m,\vare,1}_{t}|^p+\sup_{t\in[0,T]}\EE|M^{m,\vare,2}_{t}|^p\Bigg\}.\nonumber
\end{eqnarray}
Hence, the remaining work is to estimate all the terms above, which crucially depends on the regularity of the solution $\Phi_m(x,y)$ of the Poisson
equation \eref{ID2}. So we are going to study the regularity of $\Phi_m(x,y)$ with respect to $x$ and $y$ in the next section.

\section{The Poisson equation}

In Subsection 3.1, we give some a priori estimates of the solution of the finite dimensional frozen equation and the asymptotic behavior of its transition semigroup. In subsection 3.2, we study the regularity of solution $\Phi_m(x,y)$ of the Poisson equation. Note that we always assume \ref{A1}-\ref{A3} hold in this section.

\subsection{The frozen equation}
Recall the finite dimensional frozen equation \eref{Ga FZE}. Note that $F^m_2(x,\cdot)$ is Lipschitz continuous, then it is easy to show that for any fixed $x\in H$ and initial value $y\in H_m$,
equation $(\ref{Ga FZE})$ has a unique solution $\{Y_{t}^{x,y,m}\}_{t\geq 0}$ in $H_m$, i.e., $\PP$-a.s.,
\begin{eqnarray*}
Y_{t}^{x,y,m}=\!\!\!\!\!\!\!\!&&e^{tA}y+\int^t_0 e^{(t-s)A}F^m_2(x,Y_{s}^{x,y,m})ds+\int^t_0e^{(t-s)A}G^m_2(x,Y^{x,y,m}_s)d\bar{W}^{2,m}_s.
\end{eqnarray*}
Moreover, the solution $\{Y_{t}^{x,y,m}\}_{t\geq 0}$ is a time homogeneous Markov process. Let $\{P^{x,m}_t\}_{t\geq 0}$ be its transition semigroup, i.e., for any bounded measurable function $\varphi: H_m\rightarrow \mathbb{R}$,
$$
P^{x,m}_t\varphi(y):=\EE\varphi(Y_{t}^{x,y,m}), \quad y\in H_m, t\geq 0.
$$

Before studying the asymptotic behavior of $\{P^{x,m}_t\}_{t\geq 0}$, we prove the following lemma.

\begin{lemma}\label{L3.2}
For any $p\geq1$ and $\eta\in (0,1)$, we have
\begin{eqnarray}
&&(i)\quad \sup_{t\geq 0,x\in H,m\geq 1}\EE|Y_{t}^{x,y,m}|^p\leq e^{-\frac{ p(\lambda_1-L_{F_2})t}{4}}|y|^p+C_p(1+|x|^p).\label{FEq2}\\
&&(ii) \sup_{m\geq 1}\left(\EE\|Y_{t}^{x,y,m}\|^p_{\eta}\right)^{1/p}\leq Ct^{-\frac{\eta}{2}}e^{-\frac{\lambda_1 t}{2}}|y|+C(1+|x|+|y|);\label{FEqH}
\end{eqnarray}
Furthermore, if condition \eref{CFSUP} holds, we have
\begin{eqnarray}
(iii) \sup_{x\in H,m\geq 1}\EE\left(\sup_{t\in [0,T]}|Y_{t}^{x,y,m}|^p\right)\leq C_{p}T^{p/2}+C|y|^p,\quad \forall T>0.\label{FEq1}
\end{eqnarray}
\end{lemma}
\begin{proof}
(i) For any fixed $T>0$ and $p\geq 4$, by It\^{o}'s formula (see e.g. \cite[Theorem 6.1.1]{LR}), we have
\begin{eqnarray}
|Y^{x,y,m}_t|^p=\!\!\!\!\!\!\!\!&&|y|^p+p\int^{t}_{0}|Y^{x,y,m}_s|^{p-2}\langle Y^{x,y,m}_s,AY^{x,y,m}_s+F^m_2(x,Y^{x,y,m}_s\rangle ds\nonumber\\
\!\!\!\!\!\!\!\!&&+p\int^{t}_{0}|Y^{x,y,m}_s|^{p-2}\langle Y^{x,y,m}_s,G^m_2(x,Y^{x,y,m}_s)d\bar{W}_s^{2,m}\rangle\nonumber\\
&&+\frac{p}{2}\int_{0} ^{t}|Y^{x,y,m}_s|^{p-2}|G^m_2(x,Y^{x,y,m}_s)|_{HS}^2ds\nonumber\\
&&+\frac{p(p-2)}{2}\int_{0}^{t}|Y^{x,y,m}_s|^{p-4}|G^m_2(x,Y^{x,y,m}_s)^{\ast} Y^{x,y,m}_s|^2ds.\label{Ito formula}
\end{eqnarray}

Note that
$$
\langle Ay,y\rangle=-\|y\|^2_{1}\leq -\lambda_1|y|^2.
$$
Assumption \ref{A1} implies that
$$
|F^m_2(x,y)|\leq |F^m_2(x,y)-F^m_2(0,0)|+|F_2(0,0)|\leq L_{F_2}|y|+C(1+|x|).
$$
It follows from Young's inequality that
\begin{eqnarray*}
\langle y, Ay+F^m_2(x,y)\rangle\leq -\frac{\lambda_1-L_{F_2}}{2}|y|^2+C(1+|x|^2).
\end{eqnarray*}
By condition \eref{BOG2}, it is easy to see that
\begin{eqnarray*}
\frac{d}{dt}\mathbb{E}|Y^{x,y,m}_t|^p=\!\!\!\!\!\!\!\!&&p\mathbb{E}\left[|Y^{x,y,m}_t|^{p-2}\langle Y^{x,y,m}_t,AY^{x,y,m}_t+F^m_2(x,Y^{x,y,m}_t\rangle\right]\nonumber\\
\!\!\!\!\!\!\!\!&&+\frac{p}{2}\EE\left[|Y^{x,y,m}_t|^{p-2}|G^m_2(x,Y^{x,y,m}_t)|_{HS}^2\right]\nonumber\\
&&+\frac{p(p-2)}{2}\EE\left[|Y^{x,y,m}_t|^{p-4}|G^m_2(x,Y^{x,y,m}_t)^{\ast} Y^{x,y,m}_t|^2\right]\\
\leq\!\!\!\!\!\!\!\!&& -\frac{p (\lambda_1-L_{F_2})}{2}\mathbb{E}|Y_t^{x,y,m}|^p+C_p\mathbb{E}|Y_t^{x,y,m}|^{p-2+2\zeta}+C_p(1+|x|^2)\mathbb{E}|Y_t^{x,y,m}|^{p-2}\\
\leq\!\!\!\!\!\!\!\!&& -\frac{ p(\lambda_1-L_{F_2})}{4}\mathbb{E}|Y_t^{x,y,m}|^p+C_p(1+|x|^p).
\end{eqnarray*}

Note that  $\lambda_1-L_{F_2}>0$ by condition \eref{SDC}, thus the comparison theorem yields that for any $t>0$,
\begin{eqnarray*}
\EE|Y^{x,y,m}_t|^p\leq\!\!\!\!\!\!\!\!&&e^{-\frac{p(\lambda_1-L_{F_2})t}{4}}|y|^p+C_p(1+|x|^p)\int_0^te^{-\frac{p(\lambda_1-L_{F_2})(t-s)}{4}}ds\\
\leq\!\!\!\!\!\!\!\!&& e^{-\frac{ p(\lambda_1-L_{F_2})t}{4}}|y|^p+C_p(1+|x|^p).
\end{eqnarray*}

(ii) Note that for any $t>0$  and $\eta\in (0,1)$,
\begin{eqnarray*}
\|Y^{x,y,m}_{t}\|_{\eta}\leq \!\!\!\!\!\!\!\!&&\|e^{tA}y\|_{\eta}+\int^t_0 \|e^{(t-s)A}F_2(x,Y_{s}^{x,y,m})\|_{\eta}ds\\
&&+\left|\int^t_0 (-A)^{\eta/2}e^{(t-s)A}G_2(x,Y_{s}^{x,y,m})d\bar{W}^{2}_s\right|.
\end{eqnarray*}
Then by Minkowski's inequality, Burkholder-Davis-Gundy inequality, \eref{P3} and \eref{FEq2}, we have
\begin{eqnarray*}
\left[\EE\|Y^{x,y,m}_{t}\|^p_{\eta}\right]^{1/p}
\leq \!\!\!\!\!\!\!\!&&Ct^{-\frac{\eta}{2}}e^{-\frac{\lambda_1 t}{2}}|y|+C\int^t_0(t-s)^{-\frac{\eta}{2}}e^{-\frac{\lambda_1 (t-s)}{2}}\left[ \EE(1+|x|^p+|Y_{s}^{x,y,m}|^p)\right]^{1/p}ds\\
&&+C\left[\int^t_0 (t-s)^{-\eta}e^{-\frac{\lambda_1 (t-s)}{2}}\left[\EE(1+|x|^p+|Y_{s}^{x,y,m}|^p)\right]^{2/p}ds\right]^{1/2}\\
\leq \!\!\!\!\!\!\!\!&&Ct^{-\frac{\eta}{2}}e^{-\frac{\lambda_1 t}{2}}|y|+C\Big[\sup_{s\geq 0}\left(\EE|Y_{s}^{x,y,m}|^p\right)^{1/p}+1+|x|\Big]\\
&&+\Big[\sup_{s\geq 0}\left(\EE|Y_{s}^{x,y,m}|^p\right)^{1/p}+1+|x|\Big]\\
\leq \!\!\!\!\!\!\!\!&&Ct^{-\frac{\eta}{2}}e^{-\frac{\lambda_1 t}{2}}|y|+C(1+|x|+|y|).
\end{eqnarray*}

(iii)  Note that the condition \eref{CFSUP} in the assumption \ref{A4} implies
$$
\sup_{x\in H}|F^m_2(x,y)|\leq \sup_{x\in H}|F^m_2(x,y)-F^m_2(x,0)|+\sup_{x\in H}|F^m_2(x,0)|\leq L_{F_2}|y|+C.
$$
Then it is easy to see
\begin{eqnarray*}
\langle y, Ay+F^m_2(x,y)\rangle\leq -\lambda_1|y|^2+|y|(L_{F_2}|y|+C)\leq -\frac{\lambda_1-L_{F_2}}{2}|y|^2+C.
\end{eqnarray*}
Recall the formula \eref{Ito formula}, then by Burkholder-Davis-Gundy's inequality and Young's inequality, there exists $C_p>0$ such that
\begin{eqnarray}
\mathbb{E}\left[\sup_{t\in [0,T]}|Y^{x,y,m}_t|^p\right]\nonumber
\leq\!\!\!\!\!\!\!\!&&|y|^p+C_p\mathbb{E}\int^{T}_{0}|Y^{x,y,m}_s|^{p-2}ds+C_p\mathbb{E}\Big[\int^{T}_{0}|Y^{x,y,m}_s|^{2p-2}ds\Big]^{1/2}\nonumber\\
\leq\!\!\!\!\!\!\!\!&&|y|^p+\frac{1}{2}\mathbb{E}\left[\sup_{t\in [0,T]}|Y_t^{x,y,m}|^{p}\right]+C_p T^{p/2},\label{I1}
\end{eqnarray}
which proves the estimate \eref{FEq1}.
The proof is complete.
\end{proof}

\begin{lemma}\label{L3.6}
For any $t>0$, $\eta\in (0,1)$ and $x_1,x_2\in H, y_1,y_2\in H_m$, there exists $C>0$ such that
\begin{eqnarray}
 \!\!\!\!\!\!\!\!&&(i) \sup_{m\geq 1}\EE|Y^{x_1,y_1,m}_t-Y^{x_2,y_2,m}_t|^2\!\!\leq e^{-\frac{\theta t}{2}}|y_1-y_2|^2+C|x_1-x_2|^2,\label{increase Y}\\
 \!\!\!\!\!\!\!\!&&(ii) \sup_{x\in H,m\geq 1}\left(\EE\|Y^{x,y_1,m}_t-Y^{x,y_2,m}_t\|^2_{\eta}\right)^{1/2}\!\!\leq C\left(t^{-\frac{\eta}{2}}e^{-\frac{\lambda_1 t}{2}}+e^{-\frac{\theta t}{4}}\right)|y_1-y_2|,\label{increase YH}
\end{eqnarray}
where $\theta=2\lambda_{1}-2L_{F_2}-L^2_{G_2}$ is defined in assumption \ref{A2}.
\end{lemma}
\begin{proof}
(i) Note that, for any $x_1,x_2\in H, y_1,y_2\in H_m$,
\begin{eqnarray*}
d(Y^{x_1,y_1,m}_t-Y^{x_2,y_2,m}_t)=\!\!\!\!\!\!\!\!&&A(Y^{x_1,y_1,m}_t-Y^{x_2,y_2,m}_t)dt+\left[F^m_2(x_1, Y^{x_1,,y_1,m}_t)-F^m_2(x_2, Y^{x_2,y_2,m}_t)\right]dt\\
\!\!\!\!\!\!\!\!&&+\left[G^m_2(x_1, Y^{x_1,y_1,m}_t)-G^m_2(x_2, Y^{x_2,y_2,m}_t)\right]d\bar{W}^{2,m}_t.
\end{eqnarray*}
Applying  It\^{o}'s formula and taking expectation, we obtain
\begin{eqnarray*}
\frac{d}{dt}\EE|Y^{x_1,y_1,m}_t-Y^{x_2,y_2,m}_t|^2=\!\!\!\!\!\!\!\!&&-2\EE\|Y^{x_1,y_1,m}_t-Y^{x_2,y_2,m}_t\|^2_1\\
&&+\EE|G^m_2(x_1,Y^{x_1,y_1,m}_t)-G^m_2(x_2,Y^{x_2,y_2,m}_t)|^2_{HS}\\
&&+2\EE\langle F^m_2( x_1,Y^{x_1,y_1,m}_t)-F^m_2(x_2, Y^{x_2,y_2,m}_t), Y^{x_1,y_1,m}_t-Y^{x_2,y_2,m}_t\rangle.
\end{eqnarray*}

Note that $\theta>0$ by condition \eref{SDC}, then by Young's inequality we get
\begin{eqnarray*}
&&\frac{d}{dt}\EE|Y^{x_1,y_1,m}_t-Y^{x_2,y_2,m}_t|^2\\
\leq\!\!\!\!\!\!\!\!&&-2\lambda_1\EE|Y^{x_1,y_1,m}_t-Y^{x_2,y_2,m}_t|^2+\EE(C|x_1-x_2|+L_{G_2}|Y^{x_1,y_1,m}_t-Y^{x_2,y_2,m}_t|)^2\\
&&+2L_{F_2}\EE\left|Y^{x_1,y_1,m}_t-Y^{x_2,y_2,m}_t\right|^2+C|x_1-x_2|\EE\left|Y^{x_1,y_1,m}_t-Y^{x_2,y_2,m}_t\right|\nonumber\\
\leq\!\!\!\!\!\!\!\!&& -\frac{\theta}{2}\EE\left|Y^{x_1,y_1,m}_t-Y^{x_2,y_2,m}_t\right|^2+C|x_1-x_2|^2.
\end{eqnarray*}
Then the comparison theorem implies that for any $t> 0$,
\begin{eqnarray}
\EE|Y^{x_1,y_1,m}_t-Y^{x_2,y_2,m}_t|^2\leq e^{-\frac{\theta t}{2}}|y_1-y_2|^2+C|x_1-x_2|^2.\label{increase Y^2}
\end{eqnarray}

(ii) For any $x\in H, y_1,y_2\in H_m$, it is easy to see that
\begin{eqnarray*}
Y^{x,y_1,m}_t-Y^{x,y_2,m}_t=\!\!\!\!\!\!\!\!&&e^{tA}(y_1-y_2)+\int^t_0 e^{(t-s)A}\left[F^m_2(x, Y^{x,y_1,m}_s)-F^m_2(x,Y^{x,y_2,m}_s)\right]ds\\
&&+\int^t_0 e^{(t-s)A}\left[G^m_2(x, Y^{x,y_1,m}_s)-G^m_2(x,Y^{x,y_2,m}_s)\right]d\bar{W}^{2,m}_s.
\end{eqnarray*}
Then by \eref{increase Y^2} and \eref{P3}, we have for any $t>0, \eta\in (0,1)$,
\begin{eqnarray}
&&\left(\EE\|Y^{x,y_1,m}_t-Y^{x,y_2,m}_t\|^2_{\eta}\right)^{1/2}\nonumber\\
\leq\!\!\!\!\!\!\!\!&&Ct^{-\frac{\eta}{2}}e^{-\frac{\lambda_1 t}{2}}|y_1-y_2|\nonumber\\
&&\!\!\!\!\!\!\!\!+\int^t_0 (t-s)^{-\eta/2}e^{-\lambda_{1}(t-s)/2}\left[\EE\left|F^m(x, Y^{x,y_1,m}_s)-F^m(x,Y^{x,y_2,m}_s)\right|^2\right]^{1/2} ds\nonumber\\
&&\!\!\!\!\!\!\!\!+\left[\int^t_0 (t-s)^{-\eta}e^{-\lambda_{1}(t-s)}\EE\left|G^m_2(x, Y^{x,y_1,m}_s)-G^m_2(x,Y^{x,y_2,m}_s)\right|^2_{HS} ds\right]^{1/2}\nonumber\\
\leq\!\!\!\!\!\!\!\!&&Ct^{-\frac{\eta}{2}}e^{-\frac{\lambda_1 t}{2}}|y_1-y_2|+C\int^t_0 (t-s)^{-\eta/2}e^{-\lambda_{1}(t-s)/2}\left[\EE|Y^{x,y_1,m}_s-Y^{x,y_2,m}_s|^2\right]^{1/2}ds\nonumber\\
&&\!\!\!\!\!\!\!\!+C\left[\int^t_0 (t-s)^{-\eta}e^{-\lambda_{1}(t-s)}\EE|Y^{x,y_1,m}_s-Y^{x,y_2,m}_s|^2ds\right]^{1/2}\nonumber\\
\leq\!\!\!\!\!\!\!\!&&Ct^{-\frac{\eta}{2}}e^{-\frac{\lambda_1 t}{2}}|y_1-y_2|+\int^t_0(t-s)^{-\eta/2}e^{-\lambda_{1}(t-s)/2}e^{-\frac{\theta s}{4}}ds|y_1-y_2|\nonumber\\
&&\!\!\!\!\!\!\!\!+C\left[\int^t_0 (t-s)^{-\eta}e^{-\lambda_{1}(t-s)} e^{-\frac{\theta s}{2}}ds\right]^{1/2}|y_1-y_2|\nonumber\\
\leq\!\!\!\!\!\!\!\!&&C\left(t^{-\frac{\eta}{2}}e^{-\frac{\lambda_1 t}{2}}+e^{-\frac{\theta t}{4}}\right)|y_1-y_2|.\nonumber
\end{eqnarray}
The proof is complete.
\end{proof}

\vspace{0.3cm}
Under the condition \eref{SDC}, it is well known that the transition semigroup $\{P^{x,m}_t\}_{t\geq 0}$ admits a unique invariant measure
$\mu^{x,m}$ (see e.g., \cite[Theorem 4.3.9]{LR}).
%Using \eref{FEq2} and \eref{increase Y}, we obtain for any $m\in \mathbb{N}_{+}$, $t>0$ and $x\in H$,
%\begin{eqnarray}
%\EE|Y_{t}^{x,y,m}|^2\leq\!\!\!\!\!\!\!\!&&  2\EE|Y_{t}^{x,y,m}-Y_{t}^{x,0,m}|^2+2\EE|Y_{t}^{x,0,m}|^2\nonumber\\
%\leq\!\!\!\!\!\!\!\!&& e^{-\frac{\theta t}{2} }|y|^2+C(1+|x|^2).\label{E3.19}
%\end{eqnarray}
By \eref{FEq2} and a standard argument, it is easy to see
\begin{eqnarray}
\sup_{m\geq 1}\int_{H_m}|z|^2\mu^{x,m}(dz)\leq C(1+|x|^2).\label{E3.20}
\end{eqnarray}

Next, we shall prove the following exponential ergodicity for the transition semigroup $\{P^{x,m}_t\}_{t\geq 0}$,  which plays an important role in studying the regularity of the solution of the Poisson equation.
\begin{proposition} \label{ergodicity in finite}
(i) For any measurable function $\varphi: H\rightarrow H$ satisfying
$$
|\varphi(x)-\varphi(y)|\leq C|x-y|.
$$
Then it holds that for any $t>0$,
\begin{eqnarray}
\sup_{m\geq 1}\left| P^{x,m}_t\varphi(y)-\mu^{x,m}(\varphi)\right|\leq\!\!\!\!\!\!\!\!&& C e^{-\frac{\theta t}{4}}(1+|x|+|y|). \label{ergodicity1}
\end{eqnarray}
%where $\|\varphi\|_{Lip}:=\sup_{x\neq y\in H}\frac{|\varphi(x)-\varphi(y)|}{|x-y|}$.

%For any measurable function $\varphi: H^{\beta}\rightarrow H^{\alpha}$ with $\alpha\in (0,2], \beta\in (0,1)$ satisfying
(ii) For any measurable function $\varphi: H^{\beta}\rightarrow H^{\alpha}$ with $\alpha\in (0,2]$ and $\beta\in (0,1)$, and satisfying
$$
\|\varphi(x)-\varphi(y)\|_{\alpha}\leq C(1+\|x\|_{\beta}+\|y\|_{\beta})\|x-y\|_{\beta}.
$$
Then it holds that for any $t>0$,
\begin{eqnarray}
\!\!\!\!\sup_{m\geq 1}\left\| P^{x,m}_t\varphi(y)-\mu^{x,m}(\varphi)\right\|_{\alpha}\!\leq\!\!\!\!\!\!\!\!&& C\!\!\left[(t^{-\beta}+t^{-\frac{\beta}{2}})e^{-\frac{\lambda_{1}t}{2}}+(t^{-\frac{\beta}{2}}+1)e^{-\frac{\theta t}{4}}\right]\!\!(1+|x|^2+|y|^2). \label{ergodicity2}
\end{eqnarray}
\end{proposition}
\begin{proof}
(i) By the definition of the invariant measure $\mu^{x,m}$ and \eref{increase Y},  we have for any $t> 0$,
\begin{eqnarray*}
\left| P^{x,m}_t\varphi(y)-\mu^{x,m}(\varphi)\right|=\!\!\!\!\!\!\!\!&&\left| \EE \varphi( Y^{x,y,m}_t)-\int_{H_m}\varphi(z)\mu^{x,m}(dz)\right|\\
\leq\!\!\!\!\!\!\!\!&& \left|\int_{H_m}\left[\EE \varphi(Y^{x,y,m}_t)-\EE \varphi(Y^{x,z,m}_t)\right]\mu^{x,m}(dz)\right|\\
\leq\!\!\!\!\!\!\!\!&& \|\varphi\|_{Lip}\int_{H_m} \EE\left| Y^{x,y,m}_t-Y^{x,z,m}_t\right|\mu^{x,m}(dz)\\
\leq\!\!\!\!\!\!\!\!&& \|\varphi\|_{Lip} e^{-\frac{\theta t}{4}}\int_{H_m}|y-z|\mu^{x,m}(dz)\\
\leq\!\!\!\!\!\!\!\!&& \|\varphi\|_{Lip}e^{-\frac{\theta t}{4}}\left[|y|+\int_{H_m}|z|\mu^{x,m}(dz)\right]\\
\leq\!\!\!\!\!\!\!\!&& C \|\varphi\|_{Lip}e^{-\frac{\theta t}{4}}(1+|x|+|y|),
\end{eqnarray*}
where the last inequality is a consequence of \eref{E3.20}.

(ii) By \eref{FEqH}, \eref{increase YH} and following an  argument similar to that in the proof of \eref{ergodicity1}, we have for any $t> 0$ and $m\geq 1$,
\begin{eqnarray*}
&&\left\| P^{x,m}_t\varphi(y)-\mu^{x,m}(\varphi)\right\|_{\alpha}=\left\| \EE \varphi( Y^{x,y,m}_t)-\int_{H_m}\varphi(z)\mu^{x,m}(dz)\right\|_{\alpha}\\
\leq\!\!\!\!\!\!\!\!&&\int_{H_m}\left\|\EE \varphi(Y^{x,y,m}_t)-\EE \varphi(Y^{x,z,m}_t)\right\|_{\alpha}\mu^{x,m}(dz)\\
\leq\!\!\!\!\!\!\!\!&& C\int_{H_m} \EE\left[(1+\| Y^{x,y,m}_t\|_{\beta}+\| Y^{x,z,m}_t\|_{\beta})\left\| Y^{x,y,m}_t-Y^{x,z,m}_t\right\|_{\beta}\right]\mu^{x,m}(dz)\\
\leq\!\!\!\!\!\!\!\!&& C\int_{H_m} \left[\EE(1+\| Y^{x,y,m}_t\|^2_{\beta}+\| Y^{x,z,m}_t\|^2_{\beta})\right]^{1/2}\left(\EE\| Y^{x,y,m}_t-Y^{x,z,m}_t\|^2_{\beta}\right)^{1/2}\mu^{x,m}(dz)\\
\leq\!\!\!\!\!\!\!\!&& C\left(t^{-\beta}e^{-\frac{\lambda_{1}t}{2}}+t^{-\frac{\beta}{2}}e^{-\frac{\theta t}{4}}\right)\int_{H_m}|y-z|(|y|+|z|)\mu^{x,m}(dz)\\
&&+ C\left[t^{-\frac{\beta}{2}}e^{-\frac{\lambda_{1}t}{2}}+e^{-\frac{\theta t}{4}}\right]\int_{H_m}|y-z|(1+|x|+|y|+|z|)\mu^{x,m}(dz).
\end{eqnarray*}
Thus it is easy to see that
$$
\left\| P^{x,m}_t\varphi(y)-\mu^{x,m}(\varphi)\right\|_{\alpha}\leq C\left[(t^{-\beta}+t^{-\frac{\beta}{2}})e^{-\frac{\lambda_{1}t}{2}}+(t^{-\frac{\beta}{2}}+1)e^{-\frac{\theta t}{4}}\right](1+|x|^2+|y|^2).
$$
The proof is complete.
\end{proof}

\subsection{The regularity of solution of the Poisson equation}
In this subsection devotes to study the following Poisson equation:
\begin{equation}
-\mathscr{L}^m_{2}(x)\Phi_m(x,y)=F^m_1(x,y)-\bar{F}^m_1(x),\quad x,y\in H_m,\label{PE2}
\end{equation}
where $\mathscr{L}^m_{2}(x)$ is the infinitesimal generator of the transition semigroup of the finite dimensional frozen equation \eref{Ga FZE}, i.e.,
\begin{eqnarray}
\mathscr{L}^m_{2}(x)\Phi_m(x,y)=\!\!\!\!\!\!\!\!&& D_y\Phi_m(x,y)\cdot \left[Ay+F^m_2(x,y)\right] \nonumber\\
&&+\frac{1}{2}\sum^m_{k=1}\left[D_{yy}\Phi_m(x,y)\cdot (G^m_2(x,y)e_k,G^m_2(x,y)e_k)\right].\label{L_2}
\end{eqnarray}
The regularity of the solution of the Poisson equation has been studied by many people, see e.g. \cite{PV1,PV2,RSX,RSX2,RXY,SXX}.
\begin{proposition}\label{P3.6}
Under the assumption \ref{A1}-\ref{A3}. Define
\begin{eqnarray}
\Phi_m(x,y):=\int^{\infty}_{0}\left[\EE F^m_1(x,Y^{x,y,m}_t)-\bar{F}^m_1(x)\right]dt.\label{SPE}
\end{eqnarray}
Then $\Phi_m(x,y)$ is a solution of equation \eref{PE2}, and there exists $C>0$ such that for any $x,y,h,k\in H_m$,
\begin{eqnarray}
&&\sup_{m\geq1}|\Phi_m(x,y)|\leq C(1+|x|+|y|),\quad \sup_{x,y\in H_m,m\geq 1}|D_y \Phi_m(x,y)\cdot h|\leq C|h|; \label{E1}\\
&&\sup_{m\geq 1}|D_x \Phi_m(x,y)\cdot h|\leq C(1+|x|+|y|)|h|;\label{E2}\\
&& \sup_{m\geq 1}| D_{xx}\Phi_m(x, y)\cdot(h,k)|\leq C(1+|x|+|y|)|h|\|k\|_{\delta},\label{E3}
\end{eqnarray}
where $\delta\in (0,1)$ is the constant in assumption \ref{A3}. Furthermore,  if condition \eref{BH} also holds, then we have
\begin{eqnarray}
\sup_{,m\geq 1}\|\Phi_m(x,y)\|_{\alpha}\leq C(1+|x|^2+|y|^2).\label{E4}
\end{eqnarray}
\end{proposition}
\begin{proof}
The proof is divided into three steps:

\textbf{Step 1.}
Recall that $\mathscr{L}^m_{2}(x)$ is the infinitesimal generator of the transition semigroup of the frozen process $\{Y^{x,y,m}_t\}_{t\geq 0}$. It is easy to check \eref{SPE} is a solution of the Poisson equation \eref{PE2} by a standard argument. Hence, we only prove \eref{E1}-\eref{E4}.

By (i) of Proposition \ref{ergodicity in finite},
we get for any $x\in H_m$,
\begin{eqnarray*}
|\Phi_m(x,y)|\leq\!\!\!\!\!\!\!\!&&\int^{\infty}_{0}|\EE F^m_1(x,Y^{x,y,m}_t)-\bar{F}^m_1(x)|dt\\
\leq\!\!\!\!\!\!\!\!&& C(1+|x|+|y|)\int^{\infty}_{0}e^{-\frac{\theta t}{4}}dt\leq C(1+|x|+|y|).
\end{eqnarray*}
So the first estimate in \eref{E1} holds.

By (ii) of Proposition \ref{ergodicity in finite} and \eref{BH}, we get for any $x\in H$,
\begin{eqnarray*}
\|\Phi_m(x,y)\|_{\alpha}\leq\!\!\!\!\!\!\!\!&&\int^{\infty}_{0}\| \EE[F^m_1(x,Y^{x,y}_t)]-\bar{F}^m_1(x)\|_{\alpha}dt\\
\leq\!\!\!\!\!\!\!\!&& C(1+|x|^2+|y|^2)\int^{\infty}_{0}\left[(t^{-\beta}+t^{-\frac{\beta}{2}})e^{-\frac{\lambda_{1}t}{2}}+(t^{-\frac{\beta}{2}}+1)e^{-\frac{\theta t}{4}}\right]dt\\
\leq\!\!\!\!\!\!\!\!&&C(1+|x|^2+|y|^2).
\end{eqnarray*}
So \eref{E4} holds.

For any $h\in H_m$, we have
$$D_y \Phi_m(x,y)\cdot h=\int^{\infty}_0 \EE[D_y F^m_1(x,Y^{x,y,m}_t)\cdot D^h_y Y^{x,y,m}_t]dt,$$
where $D^h_{y} Y^{x,y,m}_t$ is the derivative of $Y^{x,y,m}_t$ with respect to $y$ in the direction $h$, which satisfies
 \begin{equation}\left\{\begin{array}{l}\label{partial y}
\displaystyle
dD^h_{y} Y^{x,y,m}_t=AD^h_{y} Y^{x,y,m}_tdt+D_y F^m_2(x,Y^{x,y,m}_{t})\cdot D^h_{y} Y^{x,y,m}_tdt\nonumber\\
\quad\quad\quad\quad\quad\quad+D_yG^m_2(x,Y^{x,y,m}_t)\cdot D^h_{y} Y^{x,y,m}_tdW^{2}_t,\\
D^h_{y} Y^{x,y,m}_0=h.\\
\end{array}\right.
\end{equation}
It is easy to check that
\begin{eqnarray}
\sup_{x,y\in H_m}\EE|D^h_y Y^{x,y,m}_t|^2\leq e^{-\frac{\theta t}{2}}|h|^2\label{partial yY}
\end{eqnarray}
and for any $\eta\in (0,1)$ and $t>0$,
\begin{eqnarray}
\sup_{x,y\in H_m} \EE\|D^h_{y} Y^{x,y,m}_t\|^2_{\eta}\leq Ce^{-\frac{\theta t}{2}} (1+t^{-\eta})|h|^2.\label{Hpartial yY}
\end{eqnarray}
Thus it follows
$$\sup_{x,y\in H_m}|D_y \Phi_m(x,y)\cdot h|\leq C|h|.$$
So the second estimate in \eref{E1} holds.

Now, we define
\begin{eqnarray*}
\tilde{F}^m_{t_0}(x, y, t):=\hat F^m_1(x,y, t)-\hat{F}^m_1(x, y, t+t_0),
\end{eqnarray*}
where $\hat{F}^m_1(x, y, t):=\EE F^m_1(x, Y^{x,y,m}_t)$. Note that \eref{ergodicity1} implies
\begin{eqnarray*}
\lim_{t_0\rightarrow \infty} \tilde{F}^m_{t_0}(x, y, t)=\EE[F^m_1(x,Y^{x,y,m}_t)]-\bar{F}^m_1(x).
\end{eqnarray*}
So in order to prove the estimates  \eref{E2} and \eref{E3},  it suffices to show that there exists $C>0$ such that for any $t_0>0$, $ t>0$, $x,y\in H_m$,
\begin{eqnarray}
&&|D_{x}\tilde{F}^m_{t_0}(x,y,t)\cdot h|\leq C e^{-\frac{\theta t}{8}}(1+t^{-\delta/2})(1+|x|+|y|)|h|,\label{E21}\\
&&\left|D_{xx} \tilde{F}^m_{t_0}(x,y,t)\cdot(h,k)\right|\leq Ce^{-\frac{\theta t}{8}}(1+t^{-\delta/2})(1+|x|+|y|)|h||k|\label{E22},
\end{eqnarray}
%which will be proved in step 2 and step 3 respectively.
which will be proved in Steps 2 and 3 respectively.
\vspace{0.3cm}

\textbf{Step 2.}
%In this step, we intend to prove \eref{E21}. It follows from the Markov property,
In this step, we prove \eref{E21}. It follows from the Markov property that
\begin{eqnarray*}
\tilde{F}^m_{t_0}(x,y, t)=\!\!\!\!\!\!\!\!&& \hat F^m_1(x, y, t)-\EE F^m_1(x, Y^{x,y,m}_{t+t_0})\nonumber\\
=\!\!\!\!\!\!\!\!&& \hat F^m_1(x,y, t)-\EE \{\EE[F^m_1(x,Y^{x,y,m}_{t+t_0})|\mathscr{F}_{t_0}]\}\nonumber\\
=\!\!\!\!\!\!\!\!&& \hat F^m_1(x, y, t)-\EE \hat F^m_1(x, Y^{x,y,m}_{t_0},t).
\end{eqnarray*}
Thus
\begin{eqnarray}
D_{x}\tilde{F}^m_{t_0}(x,y, t)\cdot h=\!\!\!\!\!\!\!\!&& D_{x} \hat F^m_1(x, y, t)\cdot h-\EE D_{x}\hat F^m_1(x, Y^{x,y,m}_{t_0},t)\cdot h\nonumber\\
&&- \EE \left[D_y\hat F^m_1(x,Y^{x,y,m}_{t_0},t)\cdot D^h_{x} Y^{x,y,m}_{t_0} \right],\label{5.8}
\end{eqnarray}
where $D^h_{x} Y^{x,y,m}_t$ is the derivative of $Y^{x,y,m}_t$ with respect to $x$ in the direction $h$, which satisfies
 \begin{equation}\left\{\begin{array}{l}
\displaystyle
dD^h_{x} Y^{x,y,m}_t=AD^h_{x} Y^{x,y,m}_tdt+\left[D_{x} F^m_2(x,Y^{x,y,m}_{t})\cdot h+D_y F^m_2(x,Y^{x,y,m}_{t})\cdot D^h_{x} Y^{x,y,m}_t\right]dt\nonumber\\
\quad\quad\quad\quad\quad +\left[D_xG^m_2(x,Y^{x,y,m}_t)\cdot h+D_yG^m_2(x,Y^{x,y,m}_t)\cdot D^h_{x} Y^{x,y,m}_t\right]dW^{2}_t\nonumber\\
D^h_{x} Y^{x,y,m}_0=0.\nonumber
\end{array}\right.
\end{equation}
It is easy to check that
\begin{eqnarray}
&&\sup_{t\geq 0, x,y\in H_m} \EE|D^h_{x} Y^{x,y,m}_t|^2\leq C|h|^2\label{S0}
\end{eqnarray}
and  for any $\eta\in (0,1)$,
\begin{eqnarray}
\sup_{t\geq 0, x,y\in H_m} \EE\|D^h_{x} Y^{x,y,m}_t\|^2_{\eta}\leq C|h|^2.\label{HS0}
\end{eqnarray}

Note that $D_y\hat F^m_1(x, y, t)\cdot h=\EE\left[D_y F^m_1(x, Y^{x,y,m}_t)\cdot D^h_y Y^{x,y,m}_t\right]$. By \eref{partial yY} we have
\begin{eqnarray}
\sup_{x, y\in H_m}|D_y\hat F^m_1(x, y, t)\cdot h|\leq\!\!\!\!\!\!\!\!&&Ce^{-\frac{\theta t}{4}}|h|.\label{S1}
\end{eqnarray}

Next, if we can prove that there exists $C>0$ such that for any $t>0,x,y,h,k\in H_m$,
%\begin{eqnarray*}
\begin{equation}\label{rs}
|D_{xy}\hat F^m_1(x,y, t)\cdot (h,k)|\leq C e^{-\frac{\theta t}{8}}(1+t^{-\delta/2})|h||k|,
%\end{eqnarray*}
\end{equation}
%which implies that
then we will have
\begin{eqnarray}
&&|D_{x} \hat F^m_1(x, y, t)\cdot h-\EE D_{x}\hat F^m_1(x, Y^{x,y,m}_{t_0},t)\cdot h|\nonumber\\
\leq\!\!\!\!\!\!\!\!&& \EE\left |\int^1_0 D_{xy}\hat F^m_1( x, \xi y+(1-\xi)Y^{x,y,m}_{t_0}, t)\cdot (h, y-Y^{x,y,m}_{t_0})d\xi\right|\nonumber\\
\leq\!\!\!\!\!\!\!\!&& C e^{-\frac{\theta t}{8}}(1+t^{-\delta/2})\EE|y-Y^{x,y,m}_{t_0}||h|.\label{S2}
\end{eqnarray}
Then by estimates \eref{S0}-\eref{S2} and \eref{FEq2}, we have
\begin{eqnarray*}
|D_{x}\tilde{F}^m_{t_0}( x, y, t)\cdot h|
\leq\!\!\!\!\!\!\!\!&& C e^{-\frac{\theta t}{8}}(1+t^{-\delta/2})\EE|y-Y^{x,y,m}_{t_0}||h|\\
\leq\!\!\!\!\!\!\!\!&&Ce^{-\frac{\theta t}{8}}(1+t^{-\delta/2})(1+|x|+|y|)|h|.
\end{eqnarray*}
Hence, \eref{E21} is proved.

%Now, we give a position to prove \eref{S2}.
Now we prove \eqref{rs}.
Note that
\begin{eqnarray*}
&&|D_{xy}\hat F^m_1(x,y, t)\cdot (h,k)|=|D_{xy} \left[\EE F^m_1(x, Y^{x,y,m}_t)\right]\cdot (h,k)|\nonumber\\
=\!\!\!\!\!\!\!\!&& |D_y\left[\EE D_{x} F^m_1(x, Y^{x,y,m}_t)+\EE D_{y} F^m_1(x, Y^{x,y,m}_t)\cdot D_{x} Y^{x,y,m}_t\right]\cdot (h,k)|\\
\leq\!\!\!\!\!\!\!\!&& \left|\EE \left[D_{xy} F^m_1(x, Y^{x,y,m}_t)\cdot(D^h_{y}Y^{x,y,m}_t,k)\right]\right|+\left|\EE \left[D_{y} F^m_1(x, Y^{x,y,m}_t)\cdot D^{(h,k)}_{xy} Y^{x,y,m}_t\right]\right|\\
&&+\left|\EE \left[D_{yy} F^m_1(x, Y^{x,y,m}_t)\cdot(D^h_{x} Y^{x,y,m}_t, D^k_{y} Y^{x,y,m}_t)\right]\right|,
\end{eqnarray*}
where $D^{(h,k)}_{xy}Y^{x,y,m}_t$ is the second
%derivative of $Y^{x,y,m}_{t}$
%with respect to $x$ and $y$ each once towards the direction $(h, k)\in H_m\times H_m$,
order derivative of $Y^{x,y,m}_{t}$ (once with respect to $x$ in the direction $h\in H_m$ and once with respect to $y$ in the direction $k\in H_m$),
which satisfies
 \begin{equation}\left\{\begin{array}{l}\label{Sxy DY}
\displaystyle
d D^{(h,k)}_{xy}Y^{x,y,m}_{t}=\left[AD^{(h,k)}_{xy}Y^{x,y,m}_{t}+D_{xy}F^m_2(x,Y^{x,y,m}_{t})\cdot (h,D^k_{y} Y^{x,y,m}_{t})\right.\nonumber\\
\quad\quad\quad+\left.D_{yy}F^m_2(x,Y^{x,y,m}_{t})\cdot (D^{h}_{x}Y^{x,y,m}_{t},D^{k}_{y}Y^{x,y,m}_{t})+D_{y}F^m_2(x,Y^{x,y,m}_{t})\cdot D^{(h,k)}_{xy}Y^{x,y,m}_{t}\right]dt\nonumber\\
\quad\quad\quad+\left[D_{xy}G^m_2(x,Y^{x,y,m}_{t})\cdot (h,D^k_{y} Y^{x,y,m}_{t})+D_{yy}G^m_2(x,Y^{x,y,m}_{t})\cdot (D^{h}_{x}Y^{x,y,m}_{t},D^{k}_{y}Y^{x,y,m}_{t})\right.\nonumber\\
\quad\quad\quad+\left.D_{y}G^m_2(x,Y^{x,y,m}_{t})\cdot D^{(h,k)}_{xy}Y^{x,y,m}_{t}\right]dW^2_t,\nonumber\\
D^{(h,k)}_{xy}Y^{x,y,m}_{0}=0.\\
\end{array}\right.
\end{equation}
By assumption \ref{A3}, \eref{Hpartial yY}, \eref{HS0} and a straightforward computation, it is easy to prove for any $t>0,h,k\in H_m$,
\begin{eqnarray}
\sup_{x,y\in H_m}\EE|D^{(h,k)}_{xy}Y^{x,y,m}_t|^2\leq Ce^{-\frac{\theta t}{4}}\left(t^{-\delta}+1\right)|h|^2|k|^2,\label{partial xyY}
\end{eqnarray}
which together with \eref{partial yY} and \eref{S0},
%we can easily obtain \eref{S2}.
we can easily obtain \eref{rs}.

\textbf{Step 3.}
% In this step, we intend to prove \eref{E22}. Recall that
In this step, we prove \eref{E22}. Recall that
\begin{eqnarray*}
D_{x}\tilde{F}^m_{t_0}(x,y, t)\cdot h=\!\!\!\!\!\!\!\!&& D_{x} \hat F^m_1(x, y, t)\cdot h-\EE D_{x}\hat F^m_1(x, Y^{x,y,m}_{t_0},t)\cdot h\\
&&- \EE \left[D_y\hat F^m_1(x,Y^{x,y,m}_{t_0},t)\cdot D^h_{x} Y^{x,y,m}_{t_0} \right].
\end{eqnarray*}
Then it is easy to see
\begin{eqnarray*}
&&D_{xx}\tilde{F}^m_{t_0}(x,y, t)\cdot (h,k)\\
=\!\!\!\!\!\!\!\!&& \left[D_{xx} \hat F^m_1(x, y, t)\cdot(h,k)-\EE D_{xx}\hat F^m_1(x, Y^{x,y,m}_{t_0},t)\cdot(h,k)\right]\nonumber\\
&&-\EE \left[D_{xy}\hat F^m_1(x,Y^{x,y,m}_{t_0},t)\cdot (h, D^k_{x} Y^{x,y,m}_{t_0})+D_{yx}\hat F^m_1(x, Y^{x,y,m}_{t_0},t)\cdot(D^h_{x} Y^{x,y,m}_{t_0},k)\right]\\
&&-\EE\left[D_{yy}\hat F^m_1(x, Y^{x,y,m}_{t_0},t) \cdot (D^h_{x} Y^{x,y,m}_{t_0},D^k_{x} Y^{x,y,m}_{t_0})\right]\\
&&-\EE \left[ D_{y}\hat F^m_1(x,Y^{x,y,m}_{t_0},t) \cdot D^{(h,k)}_{xx} Y^{x,y,m}_{t_0} \right]\\
:=\!\!\!\!\!\!\!\!&&\sum^4_{i=1} J_i,
\end{eqnarray*}
where $D^{(h,k)}_{xx}Y^{x,y,m}_{t}$ is the
%derivative of $Y^{x,y,m}_{t}$ with respect to $x$ twice towards the direction $(h, k) \in H_m\times H_m$,
second order derivative of $Y^{x,y,m}_{t}$ with respect to $x$ (once in the direction
$h\in H_m$ and once in the direction
$k\in H_m$),
which satisfies
\begin{equation}\left\{\begin{array}{l}
\displaystyle
d D^{(h,k)}_{xx}Y^{x,y,m}_{t}\!\!\!=\!\!\left[AD^{(h,k)}_{xx}Y^{x,y,m}_{t}\!\!\!+\!D_{xx}F^m_2(x,Y^{x,y,m}_{t})\cdot(h,k)\!\!\!+\!D_{xy}F^m_2(x,Y^{x,y,m}_{t})\cdot (h,D^k_{x} Y^{x,y,m}_{t})\right.\nonumber\\
\quad\quad+D_{yx}F^m_2(x,Y^{x,y,m}_{t})\cdot(D^h_{x} Y^{x,y,m}_{t},k)+D_{yy}F^m_2(x,Y^{x,y,m}_{t})\cdot (D^{h}_{x}Y^{x,y,m}_{t},D^{k}_{x}Y^{x,y,m}_{t})\nonumber\\
\quad\quad+\left.D_{y}F^m_2(x,Y^{x,y,m}_{t})\cdot D^{(h,k)}_{xx}Y^{x,y,m}_{t}\right]dt+\left[D_{xx}G^m_2(x,Y^{x,y,m}_{t})\cdot(h,k)\right.\nonumber\\
\quad\quad+D_{xy}G^m_2(x,Y^{x,y,m}_{t})\cdot (h,D^k_{x} Y^{x,y,m}_{t})+D_{yx}G^m_2(x,Y^{x,y,m}_{t})\cdot(D^h_{x} Y^{x,y,m}_{t},k)\nonumber\\
\quad\quad\left.+D_{yy}G^m_2(x,Y^{x,y,m}_{t})\cdot (D^{h}_{x}Y^{x,y,m}_{t},D^{k}_{x}Y^{x,y,m}_{t})+D_{y}G^m_2(x,Y^{x,y,m}_{t})\cdot D^{(h,k)}_{xx}Y^{x,y,m}_{t}\right]dW^2_t,\nonumber\\
D^{(h,k)}_{xx}Y^{x,y,m}_{0}=0.\label{SDY}\nonumber\\
\end{array}\right.
\end{equation}

For the term $J_1$, note that
\begin{eqnarray*}
D_{x}\hat F^m_1(x,y, t)\cdot h=\!\!\!\!\!\!\!\!&&\EE \left[D_x F^m_1(x, Y^{x,y,m}_{t})\cdot h+D_{y}F^m_1(x,Y^{x,y,m}_{t})\cdot D^h_{x}Y^{x,y,m}_t \right],
\end{eqnarray*}
which implies
\begin{eqnarray*}
D_{xx}\hat F^m_1(x,y, t)\cdot(h,k)=\!\!\!\!\!\!\!\!&&\EE \left[D_{xx}F^m_1(x,Y^{x,y,m}_t)\cdot(h,k)+D_{xy} F^m_1(x,Y^{x,y,m}_t) \cdot
(h,D^k_x Y^{x,y,m}_t)\right.\nonumber\\
&&+D_{yx} F^m_1(x, Y^{x,y,m}_t)\cdot (D^h_{x} Y^{x,y,m}_{t},k)\\
&&+D_{yy} F^m_1(x, Y^{x,y,m}_t)\cdot(D^h_{x} Y^{x,y,m}_{t}, D^k_{x} Y^{x,y,m}_{t})\\
&&+ \left.D_{y} F^m_1(x, Y^{x,y,m}_t)\cdot D^{(h,k)}_{xx}Y^{x,y,m}_{t} \right].
\end{eqnarray*}
Then it follows
\begin{eqnarray*}
&&D_{xxy}\hat F^m_1(x,y, t)\cdot(h,k,l)\\
=\!\!\!\!\!\!\!\!&& \EE \Big[D_{xxy}F^m_1(x,Y^{x,y,m}_t)\cdot(h,k,D^l_{y} Y^{x,y,m}_{t})+D_{xyy} F^m_1(x,Y^{x,y,m}_t) \cdot
(h,D^k_x Y^{x,y,m}_t,D^l_{y} Y^{x,y,m}_{t})\nonumber\\
&&\!\!\!\!\!\!\!\!+D_{xy} F^m_1(x,Y^{x,y,m}_t) \cdot
(h,D^{(k,l)}_{xy} Y^{x,y,m}_t)\!+\!D_{yxy} F^m_1(x, Y^{x,y,m}_t)\cdot (D^h_{x} Y^{x,y,m}_{t},k, D^l_{y} Y^{x,y,m}_{t})\\
&&\!\!\!\!\!\!\!\!+D_{yx} F^m_1(x,Y^{x,y,m}_t) \cdot
(D^{(h,l)}_{xy} Y^{x,y,m}_t, k)\!+\!D_{yyy} F^m_1(x, Y^{x,y,m}_t)\cdot(D^h_{x} Y^{x,y,m}_{t}, D^k_{x} Y^{x,y,m}_{t},D^l_{y} Y^{x,y,m}_{t})\\
&&\!\!\!\!\!\!\!\!+D_{yy} F^m_1(x,Y^{x,y,m}_t) \cdot
(D^{(h,l)}_{xy} Y^{x,y,m}_t, D^k_{x} Y^{x,y,m}_{t})\!+\!D_{yy} F^m_1(x,Y^{x,y,m}_t) \cdot (D^{h}_{x} Y^{x,y,m}_t, D^{(k,l)}_{xy} Y^{x,y,m}_{t})\\
&&\!\!\!\!\!\!\!\!+D_{yy} F^m_1(x, Y^{x,y,m}_t)\cdot (D^{(h,k)}_{xx}Y^{x,y,m}_{t}, D^l_{y} Y^{x,y,m}_{t})\!+\!D_{y} F^m_1(x, Y^{x,y,m}_t)\cdot D^{(h,k,l)}_{xxy}Y^{x,y,m}_{t} \Big],
\end{eqnarray*}
where $D^{(h,k,l)}_{xxy}Y^{x,y,m}_{t}$ is the
%derivative of $Y^{x,y,m}_{t}$ with respect to $x$ twice and $y$ once towards the direction $(h, k,l) \in H_m\times H_m\times H_m$.
third order derivative of $Y^{x,y,m}_{t}$ (once with respect to $x$ in the direction $h\in H_m$, once with respect to $x$ in the direction $k\in H_m$ and once with respect to $y$ in the direction $l\in H_m$).

By a straightforward computation, it is easy to prove for any $h,k\in H_m$,
\begin{eqnarray}
&&\sup_{t\geq 0,x,y\in H_m}\EE|D^{(h,k)}_{xx} Y^{x,y,m}_t|^2\leq C|h|^2|k|^2,\label{EDxx}\\
&&\sup_{ x,y\in H_m}\EE|D^{(h,k,l)}_{xxy}Y^{x,y,m}_{t}|^2\leq C e^{-\frac{\theta t}{4}}(1+t^{-\delta})|h|^2|k|^2|l|^2.\label{EDxxy}
\end{eqnarray}
%which combine with  $F_1\in C^{2,3}_b(H\times H,H)$, we have
Combining these with  assumption \ref{A3}, we get for any $t>0,h,k,l\in H_m$,
\begin{eqnarray}
\sup_{x,y\in H_m}|D_{xxy}\hat F^m_1(x,y, t)\cdot(h,k,l)|\leq Ce^{-\frac{\theta t}{8}}(1+t^{-\delta/2})|h|\|k\|_{\delta}|l|.\label{DxxyF}
\end{eqnarray}
%Thus it follows from \eref{DxxyF}, we get
Hence,
\begin{eqnarray}
J_1\leq\!\!\!\!\!\!\!\!&& Ce^{-\frac{\theta t}{8}}(1+t^{-\delta/2})|h|\|k\|_{\delta}\EE|y-Y^{x,y,m}_{t_0}|\nonumber\\
\leq\!\!\!\!\!\!\!\!&& Ce^{-\frac{\theta t}{8}}(1+t^{-\delta/2})|h|\|k\|_{\delta}(1+|x|+|y|).\label{J1}
\end{eqnarray}

For the term $J_2$, note that
\begin{eqnarray*}
D_{xy}\hat F^m_1(x,y,t)\cdot(h,k)
=\!\!\!\!\!\!\!\!&& \EE \Big[ D_{xy}F^m_1(x, Y^{x,y,m}_t)\cdot(h,D^k_{y}Y^{x,y,m}_t)\!\!+\!\! D_{y}F^m_1(x,Y^{x,y,m}_t)\cdot D^{(h,k)}_{xy}Y^{x,y,m}_t\\
&&+D_{yy}F^m_1(x,Y^{x,y,m}_{t})\cdot(D^h_{x}Y^{x,y,m}_t, D^k_{y}Y^{x,y,m}_t)\Big].
\end{eqnarray*}
%which together with \eref{partial yY} and \eref{S0}, we have
Combining this with assumption \ref{A3}, \eref{Hpartial yY}, \eref{S0} and \eref{partial xyY}, we get for any $t>0,h,k\in H_m$,
$$
\sup_{x,y\in H_m}|D_{xy}\hat F^m_1(x,y,t)\cdot(h,k)|\leq C e^{-\frac{\theta t}{4}}(1+t^{-\delta/2})|h||k|,
$$
which implies that
\begin{eqnarray}
J_2\leq\!\!\!\!\!\!\!\!&& C e^{-\frac{\theta t}{4}}(1+t^{-\delta/2})|h||k|.\label{J2}
\end{eqnarray}

For the term $J_3$, by a similar argument as in estimating $J_2$, we have
\begin{eqnarray*}
\sup_{x,y\in H_m}|D_{yy}\hat F^m_1(x,y,t)\cdot(h,k)|\leq C e^{-\frac{\theta t}{4}}(1+t^{-\delta/2})|h||k|,
\end{eqnarray*}
which implies that
\begin{eqnarray}
J_3\leq C e^{-\frac{\theta t}{4}}(1+t^{-\delta/2})|h||k|.\label{J3}
\end{eqnarray}

For the term $J_4$, by  \eref{EDxx} and \eref{S1},
we easily get
\begin{eqnarray}
J_4\leq C e^{-\frac{\theta t}{4}}|h||k|.\label{J4}
\end{eqnarray}

Finally, \eref{E22} can be easily obtained by combining \eref{J1}-\eref{J4}. The proof is complete.
\end{proof}

\section{Strong averaging convergence rate}

In this section, we
%are devoted to proving Theorem \ref{main result 1}. We first study the well-posedness of the averaged equation. Then by Galerkin approximation
give the proofs of Theorem \ref{main result 1} and Theorem \ref{main result 12}. We first study the well-posedness of the averaged equation. Then by the Galerkin approximation
and the technique of Poisson equation, we prove Theorem \ref{main result 1} and Theorem \ref{main result 12} respectively.
\begin{lemma} \label{barX}
%Equation \eref{Ga 1.3} exists a unique solution $\bar{X}^m_{t}$ satisfies
The process
\begin{eqnarray}
\bar X^m_t=e^{tA}x^m+\int^t_0e^{(t-s)A} \bar F^m_1(\bar X^m_s)ds+\int^t_0 e^{(t-s)A}G^m_1(\bar X^m_s)d\bar{W}^{1,m}_s,\label{3.22}
\end{eqnarray}
is the unique solution of equation \eref{Ga 1.3}.
Moreover, for any $x\in H$, $T>0$ and $p \geq1$, there exists a constant $C_{p,T}>0$ such that
\begin{align} \label{Control X}
\sup_{m\geq 1, t\in [0,T]}\mathbb{E}|\bar{X}^m_{t} |^p\leq  C_{p,T}(1+ |x|^p).
\end{align}
\end{lemma}
\begin{proof}
It suffices to check that the $\bar{F}^m_1$ is Lipschitz continuous,
%then \eref{Ga 1.3} exists a unique mild solution $\bar{X}^m_{t}$. Meanwhile the estimate \eref{Control X} can be proved by the same argument as in the proof of Lemma \ref{GA1} in the appendix.
from which one immediately get that \eref{Ga 1.3} has unique mild solution $\bar{X}^m_{t}$. The estimate \eref{Control X} can be proved by the same argument as in the proof of Lemma \ref{GA1} in the appendix.

For any $x_1,x_2\in H_m$ and $t>0$, by \eref{increase Y} and \eref{ergodicity1}, we have
\begin{eqnarray*}
|\bar{F}^m_1(x_1)-\bar{F}^m_1(x_2)|
=\!\!\!\!\!\!\!\!&&\left|\int_{H_m} F^m_1(x_1,y)\mu^{x_1,m}(dy)-\int_{H_m} F^m_1(x_2,y)\mu^{x_2,m}(dy)\right|\nonumber\\
\leq\!\!\!\!\!\!\!\!&&\left|\EE F^m_1(x_1, Y^{x_1,0,m}_t)-\int_{H_m} F^m_1(x_1,z)\mu^{x_1,m}(dz)\right|\nonumber\\
\!\!\!\!\!\!\!\!&&+\left|\EE F^m_1(x_2, Y^{x_2,0,m}_t)-\int_{H_m} F^m_1(x_2,z)\mu^{x_2,m}(dz)\right|\nonumber\\
&&+\left|\EE F^m_1(x_1, Y^{x_1,0,m}_t)-\EE F^m_1(x_2, Y^{x_2,0,m}_t)\right|\nonumber\\
\leq\!\!\!\!\!\!\!\!&&Ce^{-\frac{\theta t}{4}}(1+|x_1|+|x_2|)+C\left(|x_1-x_2|+\EE|Y^{x_1,0,m}_t-Y^{x_2,0,m}_t|\right)\nonumber\\
\leq\!\!\!\!\!\!\!\!&&Ce^{-\frac{\theta t}{4}}+C|x_1-x_2|,
\end{eqnarray*}
where $C>0$ which is independent of $t$. Then by letting $t\rightarrow \infty$, we get
 \begin{eqnarray}
|\bar{F}^m_1(x_1)-\bar{F}^m_1(x_2)|\leq C|x_1-x_2|.\label{BL}
\end{eqnarray}
The proof is complete.
\end{proof}
\begin{remark}\label{R4.2}
By following an argument similar as above, it is easy to prove that the averaged coefficient $\bar{F}_1$ is also Lipschitz continuous, thus equation \eref{1.3} admits a unique soultion.
\end{remark}

Next we are in a position to prove Theorems \ref{main result 1} firstly.

\vspace{0.1cm}
\noindent
\textbf{Proof of Theorem \ref{main result 1}}. We will divide the proof into three steps.

\textbf{Step 1.}  By Proposition \ref{P3.6}, the following Poisson equation
\begin{eqnarray*}
-\mathscr{L}^m_{2}(x)\Phi_m(x,y)=F^m_1(x,y)-\bar{F}^m_1(x)
\end{eqnarray*}
%exists a solution $\Phi_m(x,y)$ satisfies $\Phi_m(\cdot,y)\in C^{2}(H_m,H_m)$, $\Phi_m(x,\cdot)\in C^{1}(H_m, H_m)$ and estimates \eref{E1}-\eref{E3} hold.
has a solution $\Phi_m(x,y)$ and estimates \eref{E1}-\eref{E3} hold.
Thus by the argument in Step 2 of Subsection 2.3, we have for any $T>0$ and $m\in \mathbb{N}_{+}$,
\begin{eqnarray}
\sup_{t\in [0, T]}\EE|X_{t}^{m,\vare}-\bar{X}^m_{t}|^p
\leq\!\!\!\!\!\!\!\!&&C_{p,T}\vare^{p}\Bigg\{\sup_{t\in[0,T]}\EE|e^{tA}\Phi_m(x^m,y^m)-\Phi_m(X_{t}^{m,\vare},Y^{m,\vare}_{t})|^p\nonumber\\
&&+\sup_{t\in[0,T]}\EE\left|\int^t_0 (-A)e^{(t-s)A}\Phi_m(X^{m,\vare}_{s},Y^{m,\vare}_{s})ds\right|^p\nonumber\\
&&+\sup_{t\in[0,T]}\EE\left|\int^t_0 e^{(t-s)A}\mathscr{L}^m_{1}(Y^{m,\vare}_{s})\Phi_m(X_{s}^{m,\vare},Y^{m,\vare}_{s})ds\right|^p\nonumber\\
&&+\sup_{t\in[0,T]}\EE|M^{m,\vare,1}_{t}|^p+\sup_{t\in[0,T]}\EE|M^{m,\vare,2}_{t}|^p\Bigg\}\nonumber\\
:=\!\!\!\!\!\!\!\!&&C_{p,T}\vare^{p}\sum^5_{k=1}\Lambda^{m,\vare,p}_k(T),\label{F5.4}
\end{eqnarray}
where $\mathscr{L}^m_{1}(y)$, $M^{m,\vare,1}_{t}$ and $M^{m,\vare,2}_{t}$ are defined as follows:
\begin{eqnarray}
&&\mathscr{L}^m_{1}(y)\Phi_m(x,y):=D_x\Phi_m(x,y)\cdot \left[Ax+F^m_1(x,y)\right]\nonumber\\
&&\quad\quad\quad\quad\quad\quad\quad\quad\quad+\frac{1}{2}\sum^m_{k=1} \left[D_{xx}\Phi_m(x,y)\cdot (G^m_1(x)e_k,G^m_1(x)e_k)\right];\label{L^m_1}\\
&&M^{m,\vare,1}_{t}:=\int_0^te^{(t-s)A}\left[D_x\Phi_m(X_{s}^{m,\vare},Y^{m,\vare}_{s})\cdot G^m_1(X_{s}^{m,\vare})d\bar{W}^{1,m}_s\right];\label{M_1}\\
&&M^{m,\vare,2}_{t}:=\frac{1}{\sqrt{\vare}}\int_0^te^{(t-s)A} \left[D_y\Phi_m(X_{s}^{m,\vare},Y^{m,\vare}_{s})\cdot G^m_2(X_{s}^{m,\vare},Y^{m,\vare}_{s})d\bar{W}^{2,m}_s\right].\label{M_2}
\end{eqnarray}

\textbf{Step 2.} In this step, we will estimate the terms $\Lambda^{m,\vare,p}_1(T)$, $\Lambda^{m,\vare,p}_2(T)$, $\Lambda^{m,\vare,p}_4(T)$ and $\Lambda^{m,\vare,p}_5(T)$ respectively.

By Minkowski's inequality, \eref{E1}, \eref{Xvare} and \eref{Yvare} in the appendix, we have for any  $m\in \mathbb{N}_{+}$, $p\geq 2$ and $\vare\in(0,1]$,
\begin{eqnarray}
\Lambda^{m,\vare,p}_1(T)\leq\!\!\!\!\!\!\!\!&&C_p\left[1+|x|^p+|y|^p+\sup_{t\in[0,T]}\EE|X^{m,\vare}_{t}|^p+\sup_{t\in[0,T]}\EE|Y^{m,\vare}_{t}|^p\right]\nonumber\\
\leq\!\!\!\!\!\!\!\!&& C_{p,T}(1+|x|^p+|y|^p).\label{F5.5}
\end{eqnarray}

By \eref{E1}, \eref{E2}, \eref{THXvare} and \eref{THYvare} in the appendix, we have for any  $m\in \mathbb{N}_{+}$, $p\geq 2$, $\eta\in (0,1)$ and $\vare\in(0,1]$,
\begin{eqnarray*}
\Lambda^{m,\vare,p}_2(T)\leq\!\!\!\!\!\!\!\!&&C_p\sup_{t\in[0,T]}\EE\left|\int^t_0 (-A)e^{(t-s)A}\left[\Phi_m(X^{m,\vare}_{s},Y^{m,\vare}_{s})-\Phi_m(X^{m,\vare}_{t},Y^{m,\vare}_{t})\right]ds\right|^p\nonumber\\
&&+C_p\sup_{t\in[0,T]}\EE\left|\int^t_0 (-A)e^{(t-s)A}\Phi_m(X^{m,\vare}_{t},Y^{m,\vare}_{t})ds\right|^p\nonumber\\
\leq\!\!\!\!\!\!\!\!&&C_p\sup_{t\in[0,T]}\left|\int^t_0 (t-s)^{-1}\left[\EE|\Phi_m(X^{m,\vare}_{s},Y^{m,\vare}_{s})-\Phi_m(X^{m,\vare}_{t},Y^{m,\vare}_{t})|^p\right]^{1/p}ds\right|^p\nonumber\\
&&+C_p\sup_{t\in[0,T]}\EE\left|\Phi_m(X^{m,\vare}_{t},Y^{m,\vare}_{t})\right|^p\nonumber\\
\leq\!\!\!\!\!\!\!\!&&C_p\sup_{t\in[0,T]}\left|\int^t_0 (t-s)^{-1}\left[\EE\left((1+|X^{m,\vare}_{s}|^p+|X^{m,\vare}_{t}|^p+|Y^{m,\vare}_{s}|^p)|X^{m,\vare}_{s}-X^{m,\vare}_{t}|^p\right)\right]^{1/p}ds\right|^p\nonumber\\
&&+C_p\sup_{t\in[0,T]}\left|\int^t_0 (t-s)^{-1}\left[\EE|Y^{m,\vare}_{s}-Y^{m,\vare}_{t}|^p\right]^{1/p}ds\right|^p\nonumber\\
&&+C_p\sup_{t\in[0,T]}\EE\left|\Phi_m(X^{m,\vare}_{t},Y^{m,\vare}_{t})\right|^p\nonumber\\
\leq\!\!\!\!\!\!\!\!&&C_{p,T}(1+|x|^{2p}+|y|^{2p})\sup_{t\in[0,T]}\left|\int^t_0 (t-s)^{-1}(t-s)^{\frac{\eta}{2}}s^{-\frac{\eta}{2}}ds\right|^p\nonumber\\
&&+C_{p,T}(1+|x|^p+|y|^{p})\sup_{t\in[0,T]}\left|\int^t_0 (t-s)^{-1}\left(\frac{t-s}{\vare}\right)^{\frac{\eta}{2}}s^{-\frac{\eta}{2}}ds\right|^p\nonumber\\
&&+C_{p,T}(1+|x|^p+|y|^{p})\nonumber\\
\leq\!\!\!\!\!\!\!\!&& C_{p,T}(1+|x|^{2p}+|y|^{2p})\vare^{-\frac{\eta p}{2}}.
\end{eqnarray*}

By  \eref{E1}, \eref{E2}, \eref{Xvare}, \eref{Yvare} and  the Burkholder-Davis-Gundy inequality, we get for any $m\in \mathbb{N}_{+}$, $p\geq 2$ and $\vare>0$,
\begin{eqnarray}
\Lambda^{m,\vare,p}_4(T)
\leq\!\!\!\!\!\!\!\!&&C_p\EE\left[\int^T_0(1+|X^{m,\vare}_{s}|^2+|Y^{m,\vare}_{s}|^2) |G^m_1(X_{s}^{m,\vare})|^2_{HS}ds\right]^{p/2}\nonumber\\
\leq\!\!\!\!\!\!\!\!&&C_p\EE\left[\int^T_0(1+|X^{m,\vare}_{s}|^2+|Y^{m,\vare}_{s}|^2)(1+|X^{m,\vare}_{s}|^2)ds\right]^{p/2}\nonumber\\
\leq\!\!\!\!\!\!\!\!&& C_{p,T}(1+|x|^{2p}+|y|^{2p})\label{F5.7}
\end{eqnarray}
and
\begin{eqnarray}
\Lambda^{m,\vare,p}_5(T)
\leq\!\!\!\!\!\!\!\!&&C_p\EE\left[\frac{1}{\vare}\int^T_0 |G^m_2(X_{s}^{m,\vare},Y^{m,\vare}_{s})|^2_{HS}ds\right]^{p/2}\nonumber\\
\leq\!\!\!\!\!\!\!\!&&C_{p,T}\vare^{-p/2}\EE \int^T_0(1+|X_{s}^{m,\vare}|^{p}+|Y^{m,\vare}_{s}|^{p})ds\nonumber\\
\leq\!\!\!\!\!\!\!\!&&C_{p,T}(1+|x|^{p}+|y|^{p})\vare^{-p/2},\label{F5.8}
\end{eqnarray}

If for any $m\in \mathbb{N}_{+}, p\geq 1$, $\eta\in (0,1)$ and $\vare\in (0,1]$, the following estimate holds:
\begin{eqnarray}
\Lambda^{m,\vare,p}_{3}(T)\leq C_{p,T}\vare^{-\frac{p\eta}{2}}(1+\|x\|^{3p}_{\eta}+|y|^{3p}), \label{Gamma3}
\end{eqnarray}
Then by combining \eref{F5.4}-\eref{Gamma3}, we final get
\begin{eqnarray*}
\sup_{t\in[0,T]}\mathbb{E}|X_{t}^{m,\vare}-\bar{X}^m_{t}|^p\leq C_{p,T}(1+\|x\|^{3p}_{\eta}+|y|^{3p})\vare^{p/2}.
\end{eqnarray*}

\textbf{Step 3.}
%In this step, we are devote to proving \eref{Gamma3}. It is easy to see
In this step, we  prove \eref{Gamma3}. It is easy to see
\begin{eqnarray*}
\Lambda^{m,\vare,p}_3(T)\leq\!\!\!\!\!\!\!\!&&C_p\sup_{t\in[0,T]}\EE\left|\int^t_0 e^{(t-s)A}\left[D_x\Phi_m(X_{s}^{m,\vare},Y^{m,\vare}_{s})\cdot \left(AX_{s}^{m,\vare}\right)\right]ds\right|^p\nonumber\\
&&+C_p\sup_{t\in[0,T]}\EE\left|\int^t_0 e^{(t-s)A}\left[D_x\Phi_m(X_{s}^{m,\vare},Y^{m,\vare}_{s})\cdot F^m_1(X_{s}^{m,\vare},Y^{m,\vare}_{s})\right]ds\right|^p\nonumber\\
&&+C_p\sup_{t\in[0,T]}\EE\left|\sum^m_{k=1} \int_0^t e^{(t-s)A}\left[D_{xx}\Phi_m(X_{s}^{m,\vare},Y^{m,\vare}_{s})\cdot (G^m_1(X_{s}^{m,\vare})e_k,G^m_1(X_{s}^{m,\vare})e_k)\right] ds\right|^p\\
:=\!\!\!\!\!\!\!\!&&\sum^3_{i=1}\Lambda^{m,\vare,p}_{3i}(T).
\end{eqnarray*}

By \eref{E2}, \eref{X_theta} \eref{Xvare}, \eref{Yvare}  and \eref{Xvare2} in the Appendix, we have for any $\eta\in (0,1)$,
\begin{eqnarray*}
\Lambda^{m,\vare,p}_{31}(T)
\leq\!\!\!\!\!\!\!\!&&C_p\EE\left|\int^T_0 (1+|X^{m,\vare}_{s}|+|Y^{m,\vare}_{s}|)  \|X_{s}^{m,\vare}\|_2ds\right|^p\nonumber\\
\leq\!\!\!\!\!\!\!\!&&C_{p}\left|\int^T_0\left[\EE\left((1+|X^{m,\vare}_{s}|^p+|Y^{m,\vare}_{s}|^p ) \|X^{m,\varepsilon}_s\|^p_{2}\right)\right]^{1/p}ds\right|^p\nonumber\\
\leq\!\!\!\!\!\!\!\!&&C_{p,T}\vare^{-\frac{p\eta}{2}}(1+\|x\|^{2p}_{\eta}+|y|^{2p})\left|\int^T_0 s^{-1+\eta/2}ds\right|^p\nonumber\\
\leq\!\!\!\!\!\!\!\!&&C_{p,T}(1+\|x\|^{2p}_{\eta}+|y|^{2p})\vare^{-\frac{p\eta}{2}}.
\end{eqnarray*}

By an argument similar as above, we have
\begin{eqnarray*}
\Lambda^{m,\vare,p}_{32}(T)\leq\!\!\!\!\!\!\!\!&&C_{p}\EE\int^T_0 |F^m_1(X_{s}^{m,\vare},Y^{m,\vare}_{s})|^p (1+|X^{m,\vare}_{s}|^p+|Y^{m,\vare}_{s}|^p)  ds\nonumber\\
\leq\!\!\!\!\!\!\!\!&&C_{p}\EE\int^T_0(1+|X_{s}^{m,\vare}|^p+|Y^{m,\vare}_{s}|^p)(1+|X^{m,\vare}_{s}|^p+|Y^{m,\vare}_{s}|^p) ds\\
\leq\!\!\!\!\!\!\!\!&& C_{p,T}(1+|x|^{2p}+|y|^{2p}).
\end{eqnarray*}

By \eref{E3} and condition \eref{BOG1}, we get
\begin{eqnarray*}
\Lambda^{m,\vare,p}_{33}(T)\leq\!\!\!\!\!\!\!\!&&C_{p,T}\EE\int_0^T (1+|X^{m,\vare}_{s}|^p+|Y^{m,\vare}_{s}|^p)|(-A)^{\delta/2}G^m_1(X_{s}^{m,\vare})|^{2p}_{HS}ds\\
\leq\!\!\!\!\!\!\!\!&&C_{p,T}\EE\int^T_0(1+|X^{m,\vare}_{s}|^p+|Y^{m,\vare}_{s}|^p)(1+|X^{m,\vare}_{s}|^{2p}) ds\\
\leq\!\!\!\!\!\!\!\!&&C_{p,T}(1+|x|^{3p}+|y|^{3p}).\label{F5.6}
\end{eqnarray*}

As a result, it is easy to see \eref{Gamma3} holds. The proof is complete.\\

\vspace{0.3cm}

Now we are in a position to prove Theorem \ref{main result 12}.

\noindent
\textbf{Proof of Theorem \ref{main result 12}}.
It is easy to see that for any $T>0$, $p\geq 1$ and $m\in \mathbb{N}_{+}$,
\begin{eqnarray*}
\EE\left(\sup_{t\in [0, T]}|X_{t}^{\vare}-\bar{X}_{t}|^p\right)\leq\!\!\!\!\!\!\!\!&&C_p\EE\left(\sup_{t\in [0, T]}|X_{t}^{m,\vare}-X_{t}^{\vare}|^p\right)+C_p\EE\left(\sup_{t\in [0, T]}|\bar{X}^m_{t}-\bar{X}_{t}|^p\right)\\
&&+C_p\EE\left(\sup_{t\in [0, T]}|X_{t}^{m,\vare}-\bar{X}^m_{t}|^p\right).
\end{eqnarray*}
By Lemmas \ref{GA1} and \ref{GA2},
%we get
%\begin{eqnarray*}
%\lim_{m\rightarrow\infty}\EE\left(\sup_{t\in [0, T]}|X_{t}^{m,\vare}-X_{t}^{\vare}|^p\right)=0,\quad \lim_{m\rightarrow\infty}\EE\left(\sup_{t\in [0, T]}|\bar{X}^m_{t}-\bar{X}_{t}|^p\right)=0.
%\end{eqnarray*}
the proof will be completed if we show that there exists a positive constant $C$ independent of $m$ such that
$$
\EE\left(\sup_{t\in [0, T]}|X_{t}^{m,\vare}-\bar{X}^m_{t}|^p\right)\leq C\vare^{p/2}.
$$

By a similar argument in Step 2 of Subsection 2.3, we have for any $T>0$ and $m\in \mathbb{N}_{+}$,
\begin{eqnarray}
\EE\left(\sup_{t\in [0, T]}|X_{t}^{m,\vare}-\bar{X}^m_{t}|^p\right)\leq\!\!\!\!\!\!\!\!&&C_{p,T}\EE\left[\sup_{t\in[0,T]}\left|\int_{0}^{t}-e^{(t-s)A}\mathscr{L}^m_{2}(X_{s}^{m,\vare})\Phi_m(X_{s}^{m,\vare},Y^{m,\vare}_{s})ds\right|^p\right]\nonumber\\
\leq\!\!\!\!\!\!\!\!&&C_{p,T}\vare^{p}\Bigg\{\EE\left[\sup_{t\in[0,T]}|e^{tA}\Phi_m(x^m,y^m)-\Phi_m(X_{t}^{m,\vare},Y^{m,\vare}_{t})|^p\right]\nonumber\\
&&+\EE\left[\sup_{t\in[0,T]}\left|\int^t_0 (-A)e^{(t-s)A}\Phi_m(X^{m,\vare}_{s},Y^{m,\vare}_{s})ds\right|^p\right]\nonumber\\
&&+\EE\left[\sup_{t\in[0,T]}\left|\int^t_0 e^{(t-s)A}\mathscr{L}^m_{1}(Y^{m,\vare}_{s})\Phi_m(X_{s}^{m,\vare},Y^{m,\vare}_{s})ds\right|^p\right]\nonumber\\
&&+\EE\left[\sup_{t\in[0,T]}|M^{m,\vare,1}_{t}|^p\right]+\EE\left[\sup_{t\in[0,T]}|M^{m,\vare,2}_{t}|^p\right]\Bigg\}\nonumber\\
:=\!\!\!\!\!\!\!\!&&C_{p,T}\vare^{p}\sum^5_{k=1}\Gamma^{m,\vare,p}_k(T),\label{HF5.4}
\end{eqnarray}
where $\mathscr{L}^m_{1}(y)$, $M^{m,\vare,1}_{t}$ and $M^{m,\vare,2}_{t}$ are defined as in \eref{L^m_1}, \eref{M_1} and \eref{M_2} respectively.

%\begin{eqnarray}
%&&\mathscr{L}^m_{1}(y)\Phi_m(x,y):=D_x\Phi_m(x,y)\cdot \left[Ax+F^m_1(x,y)\right]\nonumber\\
%&&\quad\quad\quad\quad\quad\quad\quad\quad\quad+\frac{1}{2}\sum^m_{k=1} \left[D_{xx}\Phi_m(x,y)\cdot (G^m_1(x)e_k,G^m_1(x)e_k)\right];\label{L^m_1}\\
%&&M^{m,\vare,1}_{t}:=\int_0^te^{(t-s)A}\left[D_x\Phi_m(X_{s}^{m,\vare},Y^{m,\vare}_{s})\cdot G^m_1(X_{s}^{m,\vare})d\bar{W}^{1,m}_s\right];\label{M_1}\\
%&&M^{m,\vare,2}_{t}:=\frac{1}{\sqrt{\vare}}\int_0^te^{(t-s)A} \left[D_y\Phi_m(X_{s}^{m,\vare},Y^{m,\vare}_{s})\cdot G^m_2(X_{s}^{m,\vare},Y^{m,\vare}_{s})d\bar{W}^{2,m}_s\right].\label{M_2}
%\end{eqnarray}
By  \eref{E1} and \eref{SYvare}, we have for any  $m\in \mathbb{N}_{+}$, $p\geq 2$ and $\vare\in(0,1]$,
\begin{eqnarray}
\Gamma^{m,\vare,p}_1(T)\leq\!\!\!\!\!\!\!\!&&C_p\left[1+|x|^p+|y|^p+\EE\left(\sup_{t\in[0,T]}|X^{m,\vare}_{t}|^p\right)+\EE\left(\sup_{t\in[0,T]}|Y^{m,\vare}_{t}|^p\right)\right]\nonumber\\
\leq\!\!\!\!\!\!\!\!&&\leq C_{p,T}(1+|x|^p+|y|^p)\vare^{-p/2}.\label{F5.5}
\end{eqnarray}

By \eref{E4} and \eref{SYvare}, we have for any  $m\in \mathbb{N}_{+}$, $p\geq 2$ and $\vare\in(0,1]$,
\begin{eqnarray*}
\Gamma^{m,\vare,p}_2(T)\leq\!\!\!\!\!\!\!\!&&\EE\left[\sup_{t\in[0,T]}\left|\int^t_0 (t-s)^{-1+\alpha/2}\|\Phi_m(X^{m,\vare}_{s},Y^{m,\vare}_{s})\|_{\alpha}ds\right|^p\right]\nonumber\\
\leq\!\!\!\!\!\!\!\!&&C\EE\left[\sup_{t\in[0,T]}\left|\int^t_0 (t-s)^{-1+\alpha/2}(1+|X^{m,\vare}_{s}|^2+|Y^{m,\vare}_{s}|^2)ds\right|^p\right]\nonumber\\
\leq\!\!\!\!\!\!\!\!&&C_{p,T}\EE\left[\sup_{t\in[0,T]}(1+|X^{m,\vare}_{s}|^{2p}+|Y^{m,\varepsilon}_t|^{2p})\right]\nonumber\\
\leq\!\!\!\!\!\!\!\!&& C_{p,T}(1+|x|^{2p}+|y|^{2p})\vare^{-p/2}.
\end{eqnarray*}
%where the last inequality comes from the assumption $\sum_{k}\gamma^{\alpha}_k\Gamma^{\alpha\theta/2}_k<\infty$ by Remark \ref{Re2}.

By a minor revision in estimating $\Lambda^{m,\vare,p}_3(T)$. It is easy to prove that for any $m\in \mathbb{N}_{+}, p\geq 1$ and $\vare\in (0,1]$,
\begin{eqnarray}
\Gamma^{m,\vare,p}_{3}(T)\leq C_{p,T}(1+|x|^{3p}+|y|^{3p})\vare^{-p/2}, \label{Gamma3}
\end{eqnarray}
By  \eref{E2}, \eref{Xvare}, \eref{Yvare} and  Maximal inequality (see \cite[(1.13)]{HS}, we get for any $m\in \mathbb{N}_{+}$, $p\geq 2$ and $\vare>0$,
\begin{eqnarray}
\Gamma^{m,\vare,p}_4(T)
\leq\!\!\!\!\!\!\!\!&&C_p\EE\left[\int^T_0(1+|X^{m,\vare}_{s}|^2+|Y^{m,\vare}_{s}|^2)|G^m_1(X_{s}^{m,\vare})|^2_{HS}ds\right]^{p/2}\nonumber\\
\leq\!\!\!\!\!\!\!\!&&C_p\EE\left[\int^T_0(1+|X^{m,\vare}_{s}|^2+|Y^{m,\vare}_{s}|^2)(1+|X^{m,\vare}_{s}|^2)ds\right]^{p/2}\nonumber\\
\leq\!\!\!\!\!\!\!\!&& C_{p,T}(1+|x|^{2p}+|y|^{2p})\label{F5.7}
\end{eqnarray}
and
\begin{eqnarray}
\Gamma^{m,\vare,p}_5(T)
\leq\!\!\!\!\!\!\!\!&&C_p\EE\left[\frac{1}{\vare}\int^T_0|G^m_2(X_{s}^{m,\vare},Y^{m,\vare}_{s})|^2_{HS}ds\right]^{p/2}\nonumber\\
\leq\!\!\!\!\!\!\!\!&&C_{p,T}\vare^{-p/2}.\label{F5.8}
\end{eqnarray}

Hence, by combining \eref{F5.4}-\eref{F5.8}, we final get
\begin{eqnarray*}
\mathbb{E}\left(\sup_{t\in[0,T]}|X_{t}^{m,\vare}-\bar{X}^m_{t}|^p\right)\leq C_{p,T}(1+\|x\|^{3p}_{\eta}+|y|^{3p})\vare^{p/2}.
\end{eqnarray*}
The proof is complete.

\section{Weak averaging convergence rate}

%In this section, we are devoted to proving Theorem \ref{main result 2}. We still consider the problem in finite dimension firstly, then passing the limit to the infinite dimensional case.
In this section, we present the proof of Theorem \ref{main result 2}. We still consider the problem in finite dimension first, then passing the limit to the infinite dimensional case.
For any test function $\phi \in C_{b}^{3}(H)$, we have for any $t\geq 0$,
\begin{eqnarray} \label{ephix}
\left|\mathbb{E}\phi\left(X^{\vare}_t\right)
-\EE\phi(\bar{X}_t)\right|
\leq\!\!\!\!\!\!\!\!&&
\left|\mathbb{E}\phi\left(X^{\vare}_t\right)
-\mathbb{E}\phi\left(X^{m,\vare}_t\right)\right|+\left|\EE\phi (\bar{X}^m_t)-\EE\phi(\bar{X}_t)\right| \nonumber\\
\!\!\!\!\!\!\!\!&&+\left|\mathbb{E}\phi\left(X^{m,\vare}_t\right)-\EE\phi (\bar{X}^m_t)\right|.
\end{eqnarray}
Lemmas \ref{GA1} and \ref{GA2} in the appendix imply that
\begin{eqnarray*}
&&\lim_{m\rightarrow\infty}\sup_{t\in [0, T]}\left|\mathbb{E}\phi\left(X^{\vare}_t\right)
-\mathbb{E}\phi\left(X^{m,\vare}_t\right)\right|=0,\\
&&\lim_{m\rightarrow\infty}\sup_{t\in [0, T]}\left|\EE\phi (\bar{X}^m_t)-\EE\phi(\bar{X}_t)\right|=0.
\end{eqnarray*}
%Then the proof will be completed if we can show that
The proof will be complete if we can show that
$$
\sup_{t\in [0, T]}\left|\mathbb{E}\phi\left(X^{m,\vare}_t\right)-\EE\phi (\bar{X}^m_t)\right|\leq C\vare,
$$
where $C$ is a positive constant independent of $m$. To do this, we introduce the following finite dimensional Kolmogorov equation:
\begin{equation}\left\{\begin{array}{l}\label{KE}
\displaystyle
\partial_t u_m(t,x)=\bar{\mathscr{L}}^m_1 u_m(t,x),\quad t\in[0, T], \\
u_m(0, x)=\phi(x),\quad x\in H_m,
\end{array}\right.
\end{equation}
where $\phi\in C^{3}_b(H)$ and $\bar{\mathscr{L}}^m_1$ is the infinitesimal generator of the transition semigroup of the averaged equation \eref{Ga 1.3}, which is given by
\begin{eqnarray*}
\bar{\mathscr{L}}^m_{1}\varphi(x):=\!\!\!\!\!\!\!\!&& D\varphi(x)\cdot\left[Ax+\bar{F}^m_1(x)\right]+\frac{1}{2}\sum^m_{k=1}\left[D^2\varphi(x)\cdot(G^m_1e_k,G^m_1e_k)\right], \quad x\in H_m.
\end{eqnarray*}
%where $\varphi\in C^2_b(H)$.

\begin{proposition}\label{Kolmogorov Eq}
%Under the assumptions $F_1,F_2\in C^{3,3}_b(H\times H, H), G_1\in C^{3}_b(H,\mathcal L_{2}(H; H)), G_2\in C^{3,3}_b(H\times H,\mathcal L_{2}(H; H))$.
Suppose that assumptions \ref{A1}-\ref{A3} and \ref{A5} hold. Then equation \eref{KE} admits a unique solution $u_m$. Moreover, for any $h,k,l\in H_m$, $\eta\in (0,2]$,
\begin{eqnarray}
%&&\sup_{t\in[0, T],m\geq 1}\|u_m(t,\cdot)\|_{C^{3}_b}\leq C_T,\label{UE1}\\
&&\sup_{x\in H_m,m\geq 1}|D_{x} u_m(t,x)\cdot h|\leq C_T t^{-1+\eta/2}\|h\|_{\eta-2},\quad t\in(0,T],\label{UE3}\\
&&\sup_{x\in H_m,m\geq 1}|D_{xx} u_m(t,x)\cdot (h,k)|\leq C_T|h||k|,\quad  t\in[0,T],\label{UE4}\\
&&\sup_{x\in H_m,m\geq 1}|D_{xxx} u_m(t,x)\cdot (h,k,l)|\leq C_T |h||k||l|,\quad  t\in[0,T],\label{UE5}\\
&&\!\!\sup_{m\geq 1}|\partial_t(D_x u_m(t,x))\cdot h|\!\leq\!\! C_Tt^{-1+\eta/2}\|h\|_{\eta}\!\!+\!C_T|h|(1\!\!+\!|x|\!+\!\|x\|_{2}),x\in H_m, t\in (0,T].\label{UE2}
\end{eqnarray}
\end{proposition}

\begin{proof}
%Since that
Note that
$$\bar{F}^m_1(x):=\int_{H_m}F^m_1(x,y)\mu^{x,m}(dy)=\lim_{t\rightarrow \infty}\EE F^m_1(x, Y^{x,y,m}_t).$$
Then through a straightforward computation, by assumptions \ref{A3} and \ref{A5} it is easy to check that
\begin{eqnarray*}
&&\!\!\!\!\!\!\!\!\!\!\!\!\!\!\!\!\!|D\bar{F}^m_1(x)\cdot h|\leq C |h| \quad \forall x, h\in H_m ;  \\
&&\!\!\!\!\!\!\!\!\!\!\!\!\!\!\!\!\!|D^2\bar{F}^m_1(x)\cdot(h,k)|\leq C |h|\|k\|_{\delta},\quad \forall x, h, k\in H_m;  \\
&&\!\!\!\!\!\!\!\!\!\!\!\!\!\!\!\!\!|D^3\bar{F}^m_1(x)\cdot(h,k,l)|\leq C|h|\|k\|_{\delta}\|l\|_{\delta}, \quad \forall x,h,k,l\in H_m.
\end{eqnarray*}
Hence \eref{KE} has a unique solution $u_m$ given by
$$
u_m(t,x)=\EE\phi(\bar{X}^m_t(x)),\quad t\in [0,T].
$$

%Note that $\bar{F}^m_1\in C^{3}_b(H_m)$ and  $\phi\in C^{3}_b(H)$,  then it is easy to prove \eref{UE1} holds by a straightforward computation.
For any $h,k\in H_m$, we have
\begin{eqnarray*}
&&D_x u_m(t,x)\cdot h=\EE[D\phi(\bar{X}^m_t)\cdot \eta^{h,m}_t(x)],\\
&&D_{xx} u_m(t,x)\cdot (h,k)=\EE\left[D^2\phi(\bar{X}^m_t)\cdot (\eta^{h,m}_t(x),\eta^{k,m}_t(x))\right]+\EE\left[D\phi(\bar{X}^m_t)\cdot  \zeta^{h,k,m}_t(x)\right]\\
&&D_{xxx} u_m(t,x)\cdot (h,k,l)=\EE\left[D^3\phi(\bar{X}^m_t)\cdot (\eta^{h,m}_t(x),\eta^{k,m}_t(x),\eta^{l,m}_t(x))\right]\\
&&\quad\quad\quad+\EE\left[D^2\phi(\bar{X}^m_t)\cdot (\zeta^{h,l,m}_t(x),\eta^{k,m}_t(x))+D^2\phi(\bar{X}^m_t)\cdot (\eta^{h,m}_t(x),\zeta^{k,l,m}_t(x))\right]\\
&&\quad\quad\quad+\EE\left[D^2\phi(\bar{X}^m_t)\cdot  (\zeta^{h,k,m}_t(x), \eta^{l,m}_t(x))+D\phi(\bar{X}^m_t)\cdot  \chi^{h,k,l,m}_t(x)\right],
\end{eqnarray*}
where $\eta^{h,m}_t(x):=D_x \bar{X}^m_t(x)\cdot h$, $\zeta^{h,k,m}_t(x):=D_{xx} \bar{X}^m_t(x)\cdot (h,k)$ and $\chi^{h,k,l,m}_t(x)$  satisfy
 \begin{equation}\left\{\begin{array}{l}
\displaystyle
d\eta^{h,m}_t(x)\!=\!A\eta^{h,m}_t(x)dt+D\bar{F}^m_1(\bar{X}^m_t)\cdot \eta^{h,m}_t(x)dt,\nonumber\\
\eta^{h,m}_0(x)=h,\nonumber
\end{array}\right.
\end{equation}
 \begin{equation}\left\{\begin{array}{l}
\displaystyle
d\zeta^{h,k,m}_t(x)=\left[A\zeta^{h,k,m}_t(x)+\!\!D\bar{F}^m_1(\bar{X}^m_t)\cdot \zeta^{h,k,m}_t(x)+D^2\bar{F}^m_1(\bar{X}^m_t)\cdot (\eta^{h,m}_t(x), \eta^{k,m}_t(x))\right]\!\!dt\nonumber\\
\zeta^{h,k,m}_0(x)=0.\nonumber
\end{array}\right.
\end{equation}
and
 \begin{equation}\left\{\begin{array}{l}
\displaystyle
d\chi^{h,k,l,m}_t(x)=\left[A\chi^{h,k,m}_t(x)+\!\!D\bar{F}^m_1(\bar{X}^m_t)\cdot \chi^{h,k,l,m}_t(x)+D^2\bar{F}^m_1(\bar{X}^m_t)\cdot (\zeta^{h,k,m}_t(x), \eta^{l,m}_t(x))\right.\nonumber\\
\quad\quad +D^2\bar{F}^m_1(\bar{X}^m_t)\cdot (\zeta^{h,l,m}_t(x),\eta^{k,m}_t(x))+D^2\bar{F}^m_1(\bar{X}^m_t)\cdot (\eta^{h,m}_t(x),\zeta^{k,l,m}_t(x))\nonumber\\
\quad\quad+\left.D^3\bar{F}^m_1(\bar{X}^m_t)\cdot (\eta^{h,m}_t(x), \eta^{k,m}_t(x), \eta^{l,m}_t(x))\right]dt\nonumber\\
\chi^{h,k,l,m}_0(x)=0.\nonumber
\end{array}\right.
\end{equation}
By a straightforward computation, we obtain for any $t\in(0,T]$ and $\eta\in (0,2]$,
\begin{eqnarray}
&&|\eta^{h,m}_t(x)|\leq C_T t^{-1+\eta/2}\|h\|_{\eta-2},\label{D barX}\\
&&|\zeta^{h,k,m}_t(x)|\leq C_T |h||k|,\label{Dxx barX}\\
&&|\chi^{h,k,l,m}_t(x)|\leq C_T|h||k||l|.\label{Dxxx barX}
\end{eqnarray}
Since $\phi\in C^{3}_b(H)$, \eref{UE3}-\eref{UE5} can be easily obtained from \eref{D barX}-\eref{Dxxx barX}.

Next, we intend to prove \eref{UE2}. Considering the auxiliary Markov process defined by $Z^m_t(x,h):=(\bar{X}^m_t, \eta^{h,m}_t(x))$ with generator given by
\begin{eqnarray*}
\tilde{\mathscr{L}}^m\psi(x,h):=\!\!\!\!\!\!\!\!&& D_h\psi(x,h)\cdot \left[Ah+D\bar{F}^m_1(x)\cdot h\right]+\bar{\mathscr{L}}^m_1\psi(x,h).
\end{eqnarray*}
Define $\psi(x,h)=D\phi(x)\cdot h$, then the semigroup $\bar{P}_t \psi(x,h):=\EE[\psi(Z^m_t(x,h))]$ satisfies the following Kolmogorov equation:
\begin{eqnarray*}
\partial_t\bar{P}_t \psi(x,h)=\!\!\!\!\!\!\!\!&&\tilde{\mathscr{L}}^m\bar{P}_t \psi(x,h).
%=\!\!\!\!\!\!\!\!&&\langle Ah+D\bar{F}^m_1(x)\cdot h, D_h\bar{P}_t\psi(x,h)\rangle+\bar{\mathscr{L}}^m_1\bar{P}_t\psi(x,h).
\end{eqnarray*}
%Note that for any $t\in (0,T]$, by the fact that $\sup_{x\in H^{\theta}}\|B(x)\|_{\theta}<\infty$, we have
%%$$\bar{X}^m_t(x)=e^{tA}x+\int^t_0 e^{(t-s)A}\bar{B}^m(\bar{X}^m_s)ds+\bar{L}^m_A.$$
%\begin{eqnarray}
%\EE\|\bar{X}^m_t(x)\|_{2}
%\leq\!\!\!\!\!\!\!\!&&|(-A)e^{tA}x|+C\EE\int^t_0\|(-A)^{(2-\theta)/2}e^{(t-s)A}\|\|\bar{B}^m(\bar{X}^{m}_s)\|_{\theta}ds+\left\|\int^t_0 e^{(t-s)A}d\bar{L}^m_s\right\|_{2}\nonumber\\
%\leq\!\!\!\!\!\!\!\!&&Ct^{-1+\theta/2}\|x\|_{\theta}+C\int^t_0(t-s)^{-1+\theta/2}ds+\EE\left\|\int^t_0 e^{(t-s)A}dL_s\right\|_{2}\nonumber\\
%\leq\!\!\!\!\!\!\!\!&&C\left(t^{-1+\theta/2}\|x\|_{\theta}+t^{\theta/2}\right)+\EE\left\|\int^t_0 e^{(t-s)A}dL_s\right\|_{2}\nonumber\\
%\leq\!\!\!\!\!\!\!\!&&Ct^{-1+\theta/2}(\|x\|_{\theta}+1),\label{X_2}
%\end{eqnarray}
%where we use the fact that $\sup_{t\geq 0}\EE\left\|\int^t_0 e^{(t-s)A}dL_s\right\|_{2}<\infty$ by condition $\sum_{k\in \mathbb{N}_{+}}\beta^{\alpha}_k\lambda^{\alpha-1}_k<\infty$.
Note that $\bar{P}_t \psi(x,h)=D_x u_m(t,x)\cdot h$. By \eref{UE3}-\eref{UE5},  we get for any $t\in (0,T]$ and $\eta\in (0,2]$,
\begin{eqnarray*}
|\partial_t(D_x u_m(t,x))\cdot h|\leq\!\!\!\!\!\!\!\!&&|D_x u_m(t,x)\cdot (Ah+D\bar{F}^m_1(x)\cdot h)+D_{xx} u_m(t,x)\cdot (h, Ax+\bar{F}^m_1(x))|\\
&&+\frac{1}{2}\sum^m_{k=1}\left|D_{xxx}u_m(t,x)\cdot(h,G^m_1e_k,G^m_1e_k)\right|\\
\leq\!\!\!\!\!\!\!\!&& C_Tt^{-1+\eta/2}\|h\|_{\eta}+C_T|h|(1+|x|+\|x\|_{2})\\
\leq\!\!\!\!\!\!\!\!&& C_Tt^{-1+\eta/2}\|h\|_{\eta}+C_T|h|(1+|x|+\|x\|_{2}).
\end{eqnarray*}
Thus \eref{UE2} holds.
\end{proof}

\vspace{0.1cm}
%Now, we give a position to prove our second main result.
Now, we are in a position to prove Theorem \ref{main result 2}.

\noindent
Proof of Theorem \ref{main result 2}: We will divide the proof into two steps.

\textbf{Step 1.} Let $\tilde{u}^t_m(s,x):=u_m(t-s,x)$, $s\in [0,t]$, by It\^{o}'s formula we have
\begin{eqnarray*}
\tilde{u}^t_m(t, X^{m,\vare}_t)=\!\!\!\!\!\!\!\!&&\tilde{u}^t_m(0,x)+\int^t_0 \partial_s \tilde{u}^t_m(s, X^{m,\vare}_s )ds+\int^t_0 \mathscr{L}^m_{1}(Y^{m,\vare}_s)\tilde{u}^t_m(s, X^{m,\vare}_s)ds+\tilde{M}^m_t,
\end{eqnarray*}
where $\tilde{M}^m_t$ is defined by
$$
\tilde{M}^m_t:=\int^t_0\langle D_x\tilde{u}^t_m(s,X_s^{m,\varepsilon}),G^m_{1}d\bar{W}_s^{1,m}\rangle.
$$

Note that
\begin{eqnarray*}
&&\tilde{u}^t_m(t, X^{m,\vare}_t)=\phi(X^{m,\vare}_t);\\
&&\tilde{u}^t_m(0, x)=\EE\phi(\bar{X}^m_t(x));\\
&&\partial_s \tilde{u}^t_m(s, X^{m,\vare}_s )=-\bar{\mathscr{L}}^m_1 \tilde{u}^t_m(s, X^{m,\vare}_s).
\end{eqnarray*}
It follows that
\begin{eqnarray}
\left|\EE\phi(X^{m,\vare}_{t})-\EE\phi(\bar{X}^m_{t})\right|=\!\!\!\!\!\!\!\!&&\left|\EE\int^t_0 -\bar{\mathscr{L}}^m_1 \tilde{u}^t_m(s, X^{m,\vare}_s )ds+\EE\int^t_0 \mathscr{L}^m_{1}(Y^{m,\vare}_s)\tilde{u}^t_m(s, X^{m,\vare}_s)ds\right|\nonumber\\
=\!\!\!\!\!\!\!\!&&\left|\EE\int^t_0 D_x \tilde{u}^t_m(s, X^{m,\vare}_s )\cdot (F^m_1(X^{m,\vare}_s,Y^{m,\vare}_s)-\bar{F}^m_1(X^{m,\vare}_s)) ds \right|\nonumber\\
\leq\!\!\!\!\!\!\!\!&&\left|\EE\int^{t}_0 D_x \tilde{u}^t_m(s, X^{m,\vare}_s )\cdot (F^m_1(X^{m,\vare}_s,Y^{m,\vare}_s)-\bar{F}^m_1(X^{m,\vare}_s)) ds \right|\nonumber\\
&&+\left|\EE\int^{t}_{0} D_x \tilde{u}^t_m(s, X^{m,\vare}_s )\cdot(F^m_1(X^{m,\vare}_s,Y^{m,\vare}_s)-\bar{F}^m_1(X^{m,\vare}_s),) ds \right|\nonumber\\
\leq\!\!\!\!\!\!\!\!&&\left|\EE\int^{t}_0 D_x \tilde{u}^t_m(s, X^{m,\vare}_s )\cdot(F^m_1(X^{m,\vare}_s,Y^{m,\vare}_s)-\bar{F}^m_1(X^{m,\vare}_s)) ds \right|.\label{F5.11}
\end{eqnarray}

In order to estimate \eref{F5.11}, we consider the following Poisson equation:
\begin{eqnarray}
-\mathscr{L}^m_{2}(x)\tilde{\Phi}^t_m(s,x,y)=K^t_m(s,x,y)-\bar{K}^t_m(s,x),\label{WPE}
\end{eqnarray}
where for any $s\in [0,t], x,y\in H_m$,
\begin{eqnarray*}
&&K^t_m(s,x,y):=D_x \tilde{u}^t_m(s, x)\cdot F^m_1(x,y),\\
&&\bar{K}^t_m(s,x):=\int_{H_m} K^t_m(s,x,y)\mu^{x,m}(dy)=D_x \tilde{u}^t_m(s, x)\cdot \bar{F}^m_1(x).
\end{eqnarray*}
By assumptions \ref{A3}, \ref{A5} and $u_m(\cdot,\cdot)$ satisfy \eref{UE3}-\eref{UE2}, then following an argument similar to that used in the proof of Proposition \ref{P3.6}, we can construct
$$
\tilde{\Phi}^t_m(s, x,y):=\int^{\infty}_0 \left[\EE K^t_m(s,x,Y^{x,y,m}_r)-\bar{K}^t_m(s,x) \right]dr, \quad s\in [0,t], x,y\in H_m,
$$
which is a solution of \eref{WPE}. Moreover for any $T>0$, $t\in [0,T]$, there exists $C_T>0$ such that the following estimates hold:
\begin{eqnarray}
&&\sup_{s\in [0, t],m\geq 1}|\tilde{\Phi}^t_m(s,x,y)|\leq C_T(1+|x|+|y|);\label{E121}\\
&&\sup_{s\in [0, t],m\geq 1}|D_x \tilde{\Phi}^t_m(s,x,y)\cdot h|\leq C_{T}(1+|x|+|y|)|h|;\label{E122}\\
&&\sup_{s\in [0, t],m\geq 1}|D_{xx}\tilde{\Phi}^t_m(s,x, y)\cdot(h,k)|\leq \!\!C_{T}(1+|x|+|y|)|h|\|k\|_{\delta}; \label{E221}\\
%&&\sup_{m\geq 1}|D_x \tilde{\Phi}^t_m(s,x,y)\cdot h|\leq C_{T}(t-s)^{-1+\eta/2}(1+|x|+|y|)\|h\|_{\eta-2},\quad s\in [0,t);\label{E123}\\
&&\!\sup_{m\geq 1}|\partial_s \tilde{\Phi}^t_m(s,x,y)|\leq \!\!C_{T}\!\left[(t-s)^{-1+\alpha/2}\!\!+\!1\!+\!|x|\!+\!\|x\|_{2}\!\right]\!(1\!+\!|x|^2\!+\!|y|^2), s\in [0,t),\label{E120}
\end{eqnarray}
where $\delta$ and $\alpha$  the constants in assumptions \ref{A3} and \ref{A5} respectively.

We here only give the proof of \eref{E120}, and the proofs of \eref{E121}-\eref{E221} are omitted
%since it follows almost the same argument as in the proof of Proposition \ref{P3.6}.
since their proofs are similar to those  in the proof of Proposition \ref{P3.6}.

By condition \eref{BH}, \eref{UE2} and Propositions \ref{ergodicity in finite}, we get
\begin{eqnarray*}
|\partial_s \tilde{\Phi}^t_m(s,x,y)|\leq\!\!\!\!\!\!\!\!&&\int^{\infty}_{0}|\langle \EE F^m_1(x,Y^{x,y,m}_r)-\bar{F}^m_1(x), \partial_s(D_x \tilde{u}^t_m(s, x))\rangle |dr\\
\leq\!\!\!\!\!\!\!\!&&C_T(t-s)^{-1+\alpha/2}\int^{\infty}_{0}\|\EE F^m_1(x,Y^{x,y,m}_r)-\bar{F}^m_1(x)\|_{\alpha}dr\\
&&+C_T(1+|x|+\|x\|_{2})\int^{\infty}_{0}|\EE F^m_1(x,Y^{x,y,m}_r)-\bar{F}^m_1(x)|dr\\
\leq\!\!\!\!\!\!\!\!&& C_T(t-s)^{-1+\alpha/2}(1+|x|^2+|y|^2)\\
&&\quad\cdot\int^{\infty}_{0}\left[(r^{-\beta}+r^{-\frac{\beta}{2}})e^{-\frac{\lambda_{1}r}{2}}+(r^{-\frac{\beta}{2}}+1)e^{-\frac{\theta r}{4}}\right]dr\\
&&+C_T(1+|x|+\|x\|_{2})(1+|x|+|y|)\int^{\infty}_0 e^{-\frac{\theta}{4}r}dr\\
\leq\!\!\!\!\!\!\!\!&& C_{T}\left[(t-s)^{-1+\alpha/2}+1+|x|+\|x\|_{2}\right](1+|x|^2+|y|^2).
\end{eqnarray*}

%Note that
%\begin{eqnarray*}
%&&D_y K^t_m(s,x,y)\cdot h=\left[D_y F^m_1(x,y)\cdot D_x \tilde{u}^t_m(s, x)\right]\cdot h;\\
%&&D_{yy} K^t_m(s,x,y)\cdot (h,k)=\left[D_{yy} F^m_1(x,y)\cdot D_x \tilde{u}^t_m(s, x)\right]\cdot(h,k);\\
%&&D_{xy} K^t_m(s,x,y)\cdot (h,k)=\left[D_{xy} F^m_1(x,y)\cdot D_x \tilde{u}^t_m(s, x)\right]\cdot(h,k)\\
%&&\quad\quad\quad\quad\quad\quad\quad\quad\quad\quad\quad+\left\{D_y F^m_1(x,y)\cdot \left[D_{xx} \tilde{u}^t_m(s, x)\cdot h\right]\right\}\cdot k.
%\end{eqnarray*}
%Then by \eref{UE3} and \eref{UE4}, we have for any $t\in [0, T]$, $s\in [0,t]$ and $\eta\in (0,2]$,
%\begin{eqnarray*}
%&&|D_y K^t_m(s,x,y)\cdot h|\leq C_T (t-s)^{-1+\eta/2}\|h\|_{\eta-2};\\
%&&|D_{yy} K^t_m(s,x,y)\cdot (h,k)|\leq C_T (t-s)^{-1+\eta/2}\|h\|_{\eta-2}|k|;\\
%&&|D_{xy} K^t_m(s,x,y)\cdot (h,k)|\leq C_T (t-s)^{-1+\eta/2}\|h\|_{\eta-2}|k|.
%\end{eqnarray*}
%As a consequence, by following a similar argument as in the proof of Step 2 in Proposition \ref{P3.6}, it arrives  \eref{E123}.

\vspace{0.2cm}
\textbf{Step 2.} Applying It\^o's formula and taking expectation, we get for any $t\in [0,T]$,
\begin{eqnarray*}
&&\EE\tilde{\Phi}^t_m(t, X_{t}^{m,\vare},Y^{m,\vare}_{t})=\tilde \Phi^t_m(0, x,y)+\EE\int^{t}_0 \partial_s \tilde{\Phi}^t_m(s, X_{s}^{m,\vare},Y^{m,\vare}_{s})ds\\
&&+\EE\int^{t}_0\mathscr{L}^m_{1}(Y^{m,\vare}_{s})\tilde\Phi^t_m(s, X_{s}^{m,\vare},Y^{m,\vare}_{s})ds+\frac{1}{\vare}\EE\int^{t}_0 \mathscr{L}^m_{2}(X_{s}^{m,\vare})\tilde{\Phi}^t_m(s, X_{s}^{m,\vare},Y^{m,\vare}_{s})ds,
\end{eqnarray*}
which implies
\begin{eqnarray}
&&-\EE\int^{t}_0 \mathscr{L}^m_{2}(X_{s}^{m,\vare})\tilde\Phi^t_m(s, X_{s}^{m,\vare},Y^{m,\vare}_{s})ds\nonumber\\
=\!\!\!\!\!\!\!\!&&\vare\big[\tilde{\Phi}^t_m(0, x,y)-\EE\tilde{\Phi}^t_m(t, X_{t}^{m,\vare},Y^{m,\vare}_{t})+\EE\int^{t}_0 \partial_s \tilde{\Phi}^{t}_m(s, X_{s}^{m,\vare},Y^{m,\vare}_{s})ds\nonumber\\
&&+\EE\int^{t}_0\mathscr{L}^m_{1}(Y^{m,\vare}_{s})\tilde{\Phi}^{t}_m(s, X_{s}^{m,\vare},Y^{m,\vare}_{s})ds\big].\label{F3.39}
\end{eqnarray}

Combining  \eref{F5.11}, \eref{WPE} and \eref{F3.39}, we get
\begin{eqnarray*}
&&\sup_{t\in[0,T]}\left|\EE\phi(X^{m,\vare}_{t})-\EE\phi(\bar{X}^m_{t})\right|
\leq \vare\Bigg[\sup_{t\in [0,T]}|\tilde{\Phi}^t_m(0, x,y)|+\sup_{t\in[0,T]}\left|\EE\tilde{\Phi}^t_m(t, X_{t}^{m,\vare},Y^{m,\vare}_{t})\right|\nonumber\\
&&+\sup_{t\in [0,T]}\EE\int^{t}_0 \left|\partial_s \tilde{\Phi}^t_m(s, X_{s}^{m,\vare},Y^{m,\vare}_{s})\right|ds+\sup_{t\in [0,T]}\EE\int^{t}_0\left|\mathscr{L}^m_{1}(Y^{m,\vare}_{s})\tilde{\Phi}^t_m(s, X_{s}^{m,\vare},Y^{m,\vare}_{s})\right|ds\Bigg]\nonumber\\
:=\!\!\!\!\!\!\!\!&&\vare\sum^4_{k=1}\tilde{\Lambda}^{m,\vare}_k(T).
\end{eqnarray*}

For the terms $\tilde{\Lambda}_1^{m,\varepsilon}(T)$ and  $\tilde{\Lambda}_2^{m,\varepsilon}(T)$. By estimates \eref{E121} and \eref{Yvare}, it is easy to get
\begin{eqnarray}
\tilde{\Lambda}^{m,\vare}_1(T)+\tilde{\Lambda}^{m,\vare}_2(T)\leq C_{T}(1+|x|+|y|+\EE|X^{m,\vare}_{t}|+\EE|Y^{m,\vare}_{t}|)\leq C_T(1+|x|+|y|).\label{Lambda12(T)}
\end{eqnarray}

For the term $\tilde{\Lambda}_3^{m,\varepsilon}(T)$. By estimates \eref{E120}, \eref{Xvare}, \eref{Yvare} and \eref{X_theta}, we have
\begin{eqnarray}
\tilde{\Lambda}^{m,\vare}_3(T)\leq\!\!\!\!\!\!\!\!&& C_{T}\sup_{t\in[0,T]}\int^{t}_0\left[\EE(1+|X^{m,\vare}_{s}|^4+|Y^{m,\vare}_{s}|^4)\right]^{1/2}\\
&&\quad\quad\quad\quad\cdot\left[\EE(1+|X^{m,\vare}_{s}|^2+\|X^{m,\vare}_{s}\|^2_{2})\right]^{1/2}ds\nonumber\\
&&+C_{T}\sup_{t\in[0,T]}\int^{t}_0(t-s)^{-1+\alpha/2}\EE(1+|X^{m,\vare}_{s}|^2+|Y^{m,\vare}_{s}|^2)ds\nonumber\\
\leq\!\!\!\!\!\!\!\!&&C_{T}(1+\|x\|^3_{\eta}+|y|^3)\sup_{t\in[0,T]}\int^{t}_0 \left( s^{-1+\frac{\eta}{2}}+s^{\frac{\alpha-\beta}{2}}\right)ds\nonumber\\
&&+C_{T}(1+|x|^3+|y|^3)\sup_{t\in[0,T]}\int^{t}_0(t-s)^{-1+\alpha/2}ds\nonumber\\
\leq\!\!\!\!\!\!\!\!&&C_{T}(1+\|x\|^3_{\eta}+|y|^3).\label{Lambda3(T)}
\end{eqnarray}

For the term $\tilde{\Lambda}^{m,\vare}_4(T)$. It is easy to see
\begin{eqnarray}
\tilde{\Lambda}^{m,\vare}_4(T)\leq\!\!\!\!\!\!\!\!&&\sup_{t\in[0,T]}\int^{t}_0 \EE\left|D_x\tilde{\Phi}^t_m(X_{s}^{m,\vare},Y^{m,\vare}_{s})\cdot (AX_{s}^{m,\vare}) \right|ds\nonumber\\
&&+\sup_{t\in[0,T]}\int^{t}_0 \EE\left|D_x\tilde{\Phi}^t_m(X_{s}^{m,\vare},Y^{m,\vare}_{s})\cdot F^m_1(X_{s}^{m,\vare},Y^{m,\vare}_{s}) \right|ds\nonumber\\
&&+\sup_{t\in[0,T]}\int^{t}_0\EE\left|\frac{1}{2}\sum_{k=1}^m \left[D_{xx}\tilde{\Phi}_m^t(X_{s}^{m,\vare},Y^{m,\vare}_{s})\cdot\left(G^m_1e_k,G^m_1e_k\right)\right]\right|ds\nonumber\\
:=\!\!\!\!\!\!\!\!&&\sum^3_{i=1}\tilde{\Lambda}_{4i}^{m,\varepsilon}(T).\label{Lambda4(T)}
\end{eqnarray}
By \eref{E122}, \eref{Yvare} and \eref{Xvare2A}, we get for any $\vare>0, m\in \mathbb{N}_{+}$ and $\eta\in (0,1)$,
\begin{eqnarray}
\tilde{\Lambda}^{m,\vare}_{41}(T)\leq\!\!\!\!\!\!\!\!&&C_{T}\sup_{t\in[0,T]}\EE\int^{t}_0 \|X_{s}^{m,\vare}\|_{2}(1+|X^{m,\vare}_{s}|+|Y^{m,\vare}_{s}|)ds\nonumber\\
\leq\!\!\!\!\!\!\!\!&&C_{T}\int^{T}_0 \left[\EE\|X_{s}^{m,\vare}\|^2_{2}\right]^{1/2}\left[\EE(1+|X^{m,\vare}_{s}|^2+|Y^{m,\vare}_{s}|^2 ) \right]^{1/2}ds\nonumber\\
\leq\!\!\!\!\!\!\!\!&&C_{T}(1+\|x\|^2_{\eta}+|y|^2)\int^{T}_0 \left( s^{-1+\frac{\eta}{2}}+s^{\frac{\alpha-\beta}{2}}\right)ds\nonumber\\
\leq\!\!\!\!\!\!\!\!&&C_{T}(1+\|x\|^{2}_{\eta}+|y|^{2}).
\end{eqnarray}
By \eref{E122}, \eref{E221}, \eref{Xvare} and \eref{Yvare}, we have
\begin{eqnarray}
\tilde{\Lambda}^{m,\vare}_{42}(T)\leq\!\!\!\!\!\!\!\!&&C_{T}\int^T_0 \EE(1+|X_{s}^{m,\vare}|^2+|Y^{m,\vare}_{s}|^2)ds\nonumber\\
\leq\!\!\!\!\!\!\!\!&&C_{T}(1+|x|^{2}+|y|^{2})
\end{eqnarray}
and
\begin{eqnarray}
\tilde{\Lambda}^{m,\vare}_{43}(T)\leq\!\!\!\!\!\!\!\!&&C_T  \int^T_0\EE\left[(1+|X^{m,\vare}_s|+|Y^{m,\vare}_s|)|(-A)^{\delta/2}G^m_1|^{2}_{HS}\right]ds\nonumber\\
\leq\!\!\!\!\!\!\!\!&&C_T  \int^T_0\EE\left[(1+|X^{m,\vare}_s|+|Y^{m,\vare}_s|)(1+|X_{s}^{m,\vare}|^{2})\right]ds\nonumber\\
\leq\!\!\!\!\!\!\!\!&& C_T(1+|x|^3+|y|^3). \label{F5.22}
\end{eqnarray}

Finally, combining estimates \eref{Lambda12(T)}-\eref{F5.22}  we have
\begin{eqnarray*}
\sup_{t\in[0,T],m\geq 1}\left|\EE\phi(X^{m,\vare}_{t})-\EE\phi(\bar{X}^m_{t})\right|
\leq C_{T}(1+\|x\|^3_{\eta}+|y|^3)\vare.
\end{eqnarray*}
The proof is complete.

\section{Appendix}

%Firstly, we give some a priori estimates of the solution $(X_{t}^{\varepsilon}, Y_{t}^{\varepsilon})$ (see Lemma \ref{L6.1}). Secondly, we study the Galerkin approximation of the system \eref{main equation} (see Lemma \ref{GA1}). Thirdly, we prove the time H\"{o}lder continuity of the solution $(X_{t}^{\varepsilon}, Y_{t}^{\varepsilon})$ (See Lemma \ref{L6.3}) . Finally, we prove the finite dimensional approximation of the frozen equation \eref{FZE} (see Lemma \ref{GA2}).
We first give some a priori estimates of the solution $(X_{t}^{\varepsilon}, Y_{t}^{\varepsilon})$ (see Lemma \ref{L6.1}). Then we study the Galerkin approximation of the system \eref{main equation} (see Lemma \ref{GA1}). Next, we prove the time H\"{o}lder continuity of the solution $(X_{t}^{\varepsilon}, Y_{t}^{\varepsilon})$ (see Lemma \ref{L6.3}) and give an estimation of $|AX^{\vare,m}_t|$ (see Lemma \ref{L6.4} and Remark \ref{Re6.5}). Finally, we prove the finite dimensional approximation of the frozen equation \eref{FZE} (see Lemma \ref{GA2}). Note that the assumptions \ref{A1} and \ref{A2} hold through this section.

\begin{lemma} \label{L6.1}
For any $x,y\in H$, $p\geq 1$ and $T>0$, there exist constants $C_{p,T},\ C_{T}>0$ such that the solution $(X^{\vare}_t, Y^{\vare}_t)$ of system \eref{main equation} satisfies
\begin{eqnarray}
&&\mathbb{E}\left(\sup_{t\in [0,T]}|X_{t}^{\varepsilon}|^p \right)\leq  C_{p,T}(1+ |x|^p+|y|^p),\quad \forall \vare>0,\label{AXvare}\\
&&\sup_{t\in [0,T]}\mathbb{E} |Y_{t}^{\varepsilon} |^p\leq C_{p,T}(1+|x|^p+|y|^p ),\quad \forall \vare>0;\label{AYvare}.
\end{eqnarray}
Moreover, if condition \eref{CFSUP} holds, the following holds:
\begin{eqnarray}
\mathbb{E}\left(\sup_{t\in[0,T]}|Y_{t}^{\varepsilon} |^p\right)\leq C_{p,T}\vare^{-p/2}+C|y|^p,\quad \forall \vare>0.\label{AYvareH}
\end{eqnarray}
\end{lemma}

\begin{proof}
Applying It\^{o}'s formula (see e.g., \cite[Theorem 6.1.1]{LR}) and taking expectation, we obtain
\begin{eqnarray*}\label{ItoFormu 001}
\mathbb{E}|Y_{t}^{\vare}|^{p}=\!\!\!\!\!\!\!\!&&|y|^{p}+\frac{p}{\varepsilon}\mathbb{E}\left(\int_{0} ^{t}| Y_{s}^{\varepsilon}|^{p-2}\langle AY_{s}^{\varepsilon},Y_{s}^{\varepsilon}\rangle ds\right)\nonumber \\
 \!\!\!\!\!\!\!\!&& +\frac{p}{\varepsilon}\mathbb{E}\left(\int_{0} ^{t}| Y_{s}^{\varepsilon}|^{p-2}\langle F_2(X_{s}^{\varepsilon},Y_{s}^{\varepsilon}),Y_{s}^{\varepsilon}\rangle ds \right) + \frac{p}{2\varepsilon}\mathbb{E}\left(\int_{0} ^{t}|Y_{s}^{\varepsilon}|^{p-2}|G_2(X_{s}^{\varepsilon},Y_{s}^{\varepsilon})|_{HS}^2ds\right)
 \nonumber \\
 \!\!\!\!\!\!\!\!&&+\frac{p(p-2)}{2\varepsilon}\mathbb{E}\left(\int_{0} ^{t}|Y_{s}^{\varepsilon}|^{p-4}|G_2(X_{s}^{\varepsilon},
 Y_{s}^{\varepsilon})^*Y^\varepsilon_s|^2ds\right).\nonumber
\end{eqnarray*}

Note that $\langle Ax, x\rangle\leq -\lambda_{1}|x|^2$ and $\lambda_1-L_{F_2}>0$, by assumptions \ref{A1}, \ref{A2} and Young's inequality, we have
\begin{eqnarray*}
\frac{d}{dt}\mathbb{E}|Y_{t}^{\vare}|^{p}=\!\!\!\!\!\!\!\!&&-\frac{p}{\varepsilon}\mathbb{E}\left(| Y_{t}^{\varepsilon}|^{p-2}\|Y_{t}^{\varepsilon}\|_{1}^2 \right) +\frac{p}{\varepsilon}\mathbb{E}\left[| Y_{t}^{\varepsilon}|^{p-2}\langle F_2(X_{t}^{\varepsilon},Y_{t}^{\varepsilon}),Y_{t}^{\varepsilon}\rangle\right] \nonumber \\
 \!\!\!\!\!\!\!\!&& + \frac{p}{2\varepsilon}\mathbb{E}\left[|Y_{t}^{\varepsilon}|^{p-2}
|G_2(X_{t}^{\varepsilon},Y_{t}^{\varepsilon})|_{HS}^2\right]
 +\frac{p(p-2)}{2\varepsilon}\mathbb{E}\left[|Y_{t}^{\varepsilon}|^{p-4}
 |G_2(X_{t}^{\varepsilon},Y_{t}^{\varepsilon})^*Y^\varepsilon_t|^2\right]\nonumber \\
 \leq\!\!\!\!\!\!\!\!&&-\frac{p\lambda_{1}}{\varepsilon}\mathbb{E}|Y_{t}^
 {\varepsilon}|^{p}+
\frac{p}{\varepsilon}\mathbb{E}\Big[|Y_{t}^{\varepsilon}|^{p-2}\left(C|Y_{t}^{\varepsilon}|+C|X_{t}^{\varepsilon}|\cdot|Y_{t}^{\varepsilon}|+L_{F_2} |Y_{t}^{\varepsilon}|^2)\right) \Big]\nonumber \\
 \!\!\!\!\!\!\!\!&& +
 \frac{C_p}{\varepsilon}\mathbb{E}\left[|Y_{t}^{\varepsilon}|^{p-2}
 (1+|X_{t}^{\varepsilon}|^2+|Y_{t}^{\varepsilon}|^{2\zeta})\right]
\nonumber \\
\leq\!\!\!\!\!\!\!\!&&-\frac{p(\lambda_1-L_{F_2})}{2\varepsilon}\mathbb{E}|Y_{t}^{\varepsilon}
|^{p}+\frac{C_{p}}{\varepsilon}\EE|X_{t}^{\varepsilon}|^{p}
+\frac{C_{p}}{\varepsilon}.\label{4.4.2}
\end{eqnarray*}
By the comparison theorem, it is easy to see that
\begin{eqnarray}
\mathbb{E}|Y_{t}^{\varepsilon}|^{p}\leq\!\!\!\!\!\!\!\!&&|y|^{p}e^{-\frac{p(\lambda_1-L_{F_2})}{2\varepsilon}t}+\frac{C_{p}}{\varepsilon}\int^t_0
e^{-\frac{p(\lambda_1-L_{F_2})(t-s)}{\varepsilon}}\left(1+\EE|X_{s}^{\varepsilon}|^{p}\right)ds.\label{EY}
\end{eqnarray}

For any $p\geq1$, by \eref{EY} and Maximal inequality (see \cite[(1.13)]{HS}, we obtain that
\begin{eqnarray*}
\mathbb{E}\left(\sup_{t\in [0,T]}|X_{t}^{\varepsilon} |^p \right)\leq\!\!\!\!\!\!\!\!&&C_p|x|^p\!+\!C_p\int^T_0 \!\EE|X^{\varepsilon}_s|^p ds\!+\!C_p\int^T_0\!\EE|Y^{\varepsilon}_s|^pds+C_{p,T}\\
&&+C_p\mathbb{E}\left(\sup_{t\in [0,T]}\left|\int^t_0 e^{(t-s)A}G_1(X_{s}^{\varepsilon})dW_s^1\right|^p \right)\\
\leq\!\!\!\!\!\!\!\!&&C_{p,T}(1+|x|^p+|y|^p)+C_p\int^T_0 \EE|X^{\varepsilon}_s|^p ds.
\end{eqnarray*}
By Gronwall's inequality, we get
\begin{eqnarray*}
\mathbb{E}\left(\sup_{t\in [0,T]}|X_{t}^{\varepsilon}|^p \right)\leq\!\!\!\!\!\!\!\!&&C_{p,T}(1+|x|^p+|y|^p),
\end{eqnarray*}
which implies \eref{AYvare} holds by \eref{EY}.

\vspace{0.2cm}
For fixed $\vare>0$. Define $\tilde{Y}^{\vare}_{t}:=Y^{\vare}_{t\vare}$, then it is easy to check that the process $\{\tilde{Y}^{\vare}_{t}\}_{t\geq 0}$ satisfies
\begin{eqnarray*}
\left\{ \begin{aligned}
&d\tilde{Y}^{\vare}_{t}=\left[A\tilde{Y}^{\vare}_{t}+F_2(X^{\vare}_{t\vare},\tilde{Y}^{\vare}_{t})\right]dt+G_2(X^{\vare}_{t\vare},\tilde{Y}^{\vare}_{t})d \tilde W_{t}^2,\\
&\tilde{Y}^{\vare}_{0}=y,
\end{aligned} \right.
\end{eqnarray*}
where $\{\tilde{W}_{t}^2:=\frac{1}{\vare^{1/2}}W_{t\vare}^2\}_{t\geq 0}$ is also a cylindrical-wiener process, which has the same distribution with $\{W_t^2\}_{t\geq 0}$.
Note that under the condition \eref{CFSUP}, by following a similar argument as in the proof of \eref{FEq1}, we get
\begin{eqnarray*}
\EE\left[\sup_{t\in [0,T]}|\tilde{Y}^{\vare}_{t}|^p\right]\leq C_p T^{p/2}+C|y|^p.
\end{eqnarray*}
As a consequence,
\begin{eqnarray*}
\EE\left[\sup_{t\in [0,T]}|Y^{\vare}_{t}|^p\right]=\EE\left[\sup_{t\in [0, T/\vare]}|\tilde{Y}^{\vare}_{t}|^p\right]\leq C_{p,T}\vare^{-p/2}+C|y|^p,\quad \forall \vare,T>0.
\end{eqnarray*}
The proof is complete.
\end{proof}

\vspace{0.3cm}
Recall the Galerkin approximation \eref{Ga mainE} of system \eref{main equation}. We have the following a priori estimates and  approximation.
\begin{lemma} \label{GA1}
For any $\vare>0, (x,y)\in H\times H$ and $m\in \mathbb{N}_{+}$, system \eref{Ga mainE} has a unique mild solution $(X^{m,\vare}_t, Y^{m,\vare}_t)\in H\times H$, i.e., $\PP$-a.s.,
\begin{equation}\left\{\begin{array}{l}\label{A mild solution of F}
\displaystyle
X^{m,\varepsilon}_t=e^{tA}x^m+\int^t_0e^{(t-s)A}F^m_1(X^{m,\varepsilon}_s, Y^{m,\varepsilon}_s)ds+\int^t_0 e^{(t-s)A}G^m_1(X^{m,\varepsilon}_s)d\bar {W}^{1,m}_s,\nonumber\\
Y^{m,\varepsilon}_t=e^{tA/\varepsilon}y^m+\frac{1}{\varepsilon}\int^t_0e^{(t-s)A/\varepsilon}F^m_2(X^{m,\varepsilon}_s,Y^{m,\varepsilon}_s)ds
+\frac{1}{\sqrt{\varepsilon}}\int^t_0 e^{(t-s)A/\varepsilon}G^m_2(X^{m,\varepsilon}_s,Y^{m,\varepsilon}_s)d\bar {W}^{2,m}_s.\nonumber
\end{array}\right.
\end{equation}
Moreover, for any $p\geq 1$ and $T>0$, there exist constants $C_{p,T}, C_{T}>0$ such that any $\vare>0$,
\begin{eqnarray}
&&\sup_{m\geq 1}\mathbb{E}\left(\sup_{t\in [0,T]}|X_{t}^{m,\varepsilon}|^p \right)\leq  C_{p,T}(1+ |x|^p+|y|^p);\label{Xvare}\\
&&\sup_{t\in [0,T],m\geq 1}\mathbb{E} |Y_{t}^{m,\varepsilon} |^p\leq C_{p,T}(1+|x|^p+|y|^p );\label{Yvare}\\
&&\lim_{m\rightarrow \infty}\EE\left(\sup_{t\in[0,T]}|X^{m,\vare}_t-X^{\vare}_t|^p\right)=0.\label{FA1}
\end{eqnarray}
Moreover, if condition \eref{CFSUP} holds, we have
\begin{eqnarray}
\sup_{m\geq 1}\mathbb{E}\left(\sup_{t\in[0,T]}|Y_{t}^{m,\varepsilon} |^p\right)\leq C_{p,T}\vare^{-p/2}+C|y|^p.\label{SYvare}
\end{eqnarray}
\end{lemma}

\begin{proof}
Under the assumptions \ref{A1} and \ref{A2}, it is easy to show that the existence and uniqueness of the mild solution of system \eref{Ga mainE} . The estimates \eref{Xvare}, \eref{Yvare} and \eref{SYvare} can be proved by following the same argument as in the proof of Lemma \ref{L6.1}. Next, we prove the approximation \eref{FA1}.

%On one hand, note that
Note that
\begin{eqnarray*}
Y^{m,\vare}_t-Y^{\vare}_t=\!\!\!\!\!\!\!\!&&e^{tA/\vare}(y^m-y)+\frac{1}{\vare}\int^{t}_{0}e^{(t-s)A/\vare}(\pi_m-I)F_2(X^{\vare}_s, Y^{\vare}_s)ds \nonumber\\ &&+\frac{1}{\vare}\int^{t}_{0}e^{(t-s)A/\vare}[F^{m}_2(X^{m,\vare}_s,Y^{m,\vare}_s)-F^{m}_2(X^{\vare}_s,Y^{\vare}_s)]ds\\
&&+\frac{1}{\sqrt{\varepsilon}}\int^t_0 e^{(t-s)A/\vare}\left[G^m_2(X^{m,\varepsilon}_s,Y^{m,\vare}_s)-G^m_2(X^{\varepsilon}_s,Y^{\vare}_s)\right]d \bar{W}^{2,m}_s\nonumber\\
&&+\frac{1}{\sqrt{\varepsilon}}\int^t_0 e^{(t-s)A/\vare}G^m_2(X^{\varepsilon}_s,Y^{\vare}_s)d \bar{W}^{2,m}_s-\frac{1}{\sqrt{\varepsilon}}\int^t_0 e^{(t-s)A/\vare}G_2(X^{\varepsilon}_s,Y^{\vare}_s)dW^2_s.
\end{eqnarray*}
Then it is clear that for any $t>0$,
\begin{eqnarray*}
\EE|Y^{m,\vare}_t-Y^{\vare}_t|^p\leq\!\!\!\!\!\!\!\!&&C_p|y^m-y|^p+C_{p,\vare}\EE\int^{t}_{0}|(\pi_m-I)F_2(X^{\vare}_s, Y^{\vare}_s)|^pds\nonumber\\
&&+C_{p,\vare}\EE\int^t_0|F^{m}_2(X^{m,\vare}_s,Y^{m,\vare}_s)-F^{m}_2(X^{\vare}_s,Y^{\vare}_s)|^pds\nonumber\\
&&+C_{p,\vare}\int^t_0|G^m_2(X^{m,\varepsilon}_s,Y^{m,\vare}_s)-G^m_2(X^{\varepsilon}_s,Y^{\vare}_s)|^p_{HS}ds\nonumber\\
&&+C_{p,\vare}\EE\left|\int^t_0 e^{(t-s)A/\vare}G^m_2(X^{\varepsilon}_s,Y^{\vare}_s)d \bar{W}^{2,m}_s-\int^t_0 e^{(t-s)A/\vare}G_2(X^{\varepsilon}_s,Y^{\vare}_s)dW^2_s\right|^p\nonumber\\
\leq\!\!\!\!\!\!\!\!&&C_p|y^m-y|^p+C_{p,\vare}\EE\int^{t}_{0}|(\pi_m-I)F_2(X^{\vare}_s, Y^{\vare}_s)|^pds\nonumber\\&&+C_{p,\vare}\int^t_0\EE|X^{m,\vare}_s-X^{\vare}_s|^pds+\int^t_0\EE|Y^{m,\vare}_s-Y^{\vare}_s|^pds\nonumber\\
&&+C_{p,\vare}\EE\left|\int^t_0 e^{(t-s)A/\vare}G^m_2(X^{\varepsilon}_s,Y^{\vare}_s)d \bar{W}^{2,m}_s-\int^t_0 e^{(t-s)A/\vare}G_2(X^{\varepsilon}_s,Y^{\vare}_s)dW^2_s\right|^p.
\end{eqnarray*}
By Gronwall's inequality, we obtain
\begin{eqnarray}
\sup_{t\in[0,T]}\EE|Y^{m,\vare}_t-Y^{\vare}_t|^p\leq\!\!\!\!\!\!\!\!&&C_{p,T,\vare}|y^m-y|^p+C_{p,T,\vare}\EE\int^{T}_{0}|(\pi_m-I)F_2(X^{\vare}_s, Y^{\vare}_s)|^pds\nonumber\\
&&\!\!\!\!\!\!\!\!\!\!\!\!\!\!\!\!\!\!\!\!+C_{p,T,\vare}\int^T_0\EE|X^{m,\vare}_s-X^{\vare}_s|^pds+C_{p,T,\vare}\!\!\sup_{t\in[0,T]}\EE\left|\int^t_0 \!\!e^{(t-s)A/\vare}G^m_2(X^{\varepsilon}_s,Y^{\vare}_s)d \bar{W}^{2,m}_s\right.\nonumber\\
&&\quad\quad\left.-\int^t_0 \!\!e^{(t-s)A/\vare}G_2(X^{\varepsilon}_s,Y^{\vare}_s)dW^2_s\right|^p.\label{Y^m-Y}
\end{eqnarray}

It is easy to see that for any $t\geq 0$,
\begin{eqnarray*}
X^{m,\vare}_t-X^{\vare}_t=\!\!\!\!\!\!\!\!&&e^{tA}(x^m-x)+\int^{t}_{0}e^{(t-s)A}(\pi_m-I)F_1(X^{\vare}_s, Y^{\vare}_s)ds\nonumber\\
&&+\int^{t}_{0}e^{(t-s)A}\left[F^{m}_1(X^{m,\vare}_s,Y^{m,\vare}_s)-F^{m}_1(X^{\vare}_s,Y^{\vare}_s)\right]ds\nonumber\\
&&+\int^t_0 e^{(t-s)A}\left[G^m_1(X^{m,\varepsilon}_s)-G^m_1(X^{\varepsilon}_s)\right]d \bar W^{1,m}_s\\
&&+\int^t_0 e^{(t-s)A}G^m_1(X^{\varepsilon}_s)d \bar W^{1,m}_s-\int^t_0 e^{(t-s)A}G_1(X^{\varepsilon}_s)d  W^{1}_s.
\end{eqnarray*}
Then for any $T>0$, $p\geq 1$, by the Burkholder-Davis-Gundy inequality and \eref{Y^m-Y}, we get
\begin{eqnarray}
&&\EE\left(\sup_{t\in [0,T]}|X^{m,\vare}_t-X^{\vare}_t|^p\right)\nonumber\\
\leq\!\!\!\!\!\!\!\!&&C_p|x^m-x|^p+C_{p,T}\int^{T}_{0}\EE|(\pi_m-I)F_1(X^{\vare}_s, Y^{\vare}_s)|^pds\nonumber\\
&&+C_{p,T}\EE\int^{T}_{0}|X^{m,\vare}_s-X^{\vare}_s|^pds+C_{p,T}\EE\int^{T}_{0}|Y^{m,\vare}_s-Y^{\vare}_s|^pds\nonumber\\
&&+C_p\EE\left[\sup_{t\in [0,T]}\left|\int^t_0 e^{(t-s)A}G^m_1(X^{\varepsilon}_s)d \bar W^{1,m}_s-\int^t_0 e^{(t-s)A}G_1(X^{\varepsilon}_s)dW^1_s\right|^p\right]\nonumber\\
\leq\!\!\!\!\!\!\!\!&&C_p|x^m-x|^p+C_{p,T}|y^m-y|^p+C_{p,T}\EE\int^{T}_{0}|X^{m,\vare}_s-X^{\vare}_s|^p ds\nonumber\\
&&+C_{p,T}\!\int^{T}_{0}\!\!\EE|(\pi_m-I)F_1(X^{\vare}_s, Y^{\vare}_s)|^pds\!+\!C_{p,T}\!\int^{T}_{0}\!\!\EE|(\pi_m-I)F_2(X^{\vare}_s, Y^{\vare}_s)|^p ds\nonumber\\
&&+C_p\EE\left[\sup_{t\in [0,T]}\left|\int^t_0 e^{(t-s)A}G^m_1(X^{\varepsilon}_s)d \bar W^{1,m}_s-\int^t_0 e^{(t-s)A}G_1(X^{\varepsilon}_s)dW^1_s\right|^p\right]\nonumber\\
&&+C_{p,T,\vare}\sup_{t\in[0,T]}\EE\left|\int^t_0 e^{(t-s)A/\vare}G^m_2(X^{\varepsilon}_s,Y^{\vare}_s)d \bar{W}^{2,m}_s-\int^t_0 e^{(t-s)A/\vare}G_2(X^{\varepsilon}_s,Y^{\vare}_s)dW^2_s\right|^p.\label{X^m-X}
\end{eqnarray}
The Gronwall's inequality implies
\begin{eqnarray*}
\EE\left(\sup_{t\in [0,T]}|X^{m,\vare}_t-X^{\vare}_t|^p\right)\leq\!\!\!\!\!\!\!\!&&C_{p,T}(|x^m-x|^p+|y^m-y|^p)\nonumber\\
&&\!\!\!\!\!\!\!\!\!\!\!\!\!\!\!\!\!\!\!\!\!\!\!\!\!\!\!\!\!\!\!\!+C_{p,T}\int^{T}_{0}\EE|(\pi_m-I)F_1(X^{\vare}_s, Y^{\vare}_s)|^pds\!+\!C_{p,T}\int^{T}_{0}\!\!\EE|(\pi_m-I)F_2(X^{\vare}_s, Y^{\vare}_s)|^p ds\!\nonumber\\
&&\!\!\!\!\!\!\!\!\!\!\!\!\!\!\!\!\!\!\!\!\!\!\!\!\!\!\!\!\!\!\!\!+C_p\EE\left[\sup_{t\in [0,T]}\left|\int^t_0 e^{(t-s)A}G^m_1(X^{\varepsilon}_s)d \bar W^{1,m}_s-\int^t_0 e^{(t-s)A}G_1(X^{\varepsilon}_s)dW^1_s\right|^p\right]\nonumber\\
&&\!\!\!\!\!\!\!\!\!\!\!\!\!\!\!\!\!\!\!\!\!\!\!\!\!\!\!\!\!\!\!\!+C_{p,T,\vare}\sup_{t\in[0,T]}\EE\left|\int^t_0 e^{(t-s)A/\vare}G^m_2(X^{\varepsilon}_s,Y^{\vare}_s)d \bar{W}^{2,m}_s-\int^t_0 e^{(t-s)A/\vare}G_2(X^{\varepsilon}_s,Y^{\vare}_s)dW^2_s\right|^p.
\end{eqnarray*}

It is clear that as $m\rightarrow\infty$,
$$
|x^m-x|^p\rightarrow 0,\quad |y^m-y|^p\rightarrow 0.
$$
By the a priori estimate of $(X^{\vare}_s, Y^{\vare}_s)$ in \eref{L6.1} and the dominated convergence theorem, we have
\begin{eqnarray*}
&&\lim_{m\rightarrow \infty}\int^{T}_{0}\EE|(\pi_m-I)F_1(X^{\vare}_s, Y^{\vare}_s)|^pds=0;\\
&&\lim_{m\rightarrow \infty}\int^{T}_{0}\EE|(\pi_m-I)F_2(X^{\vare}_s, Y^{\vare}_s)|^p ds=0;\\
&&\lim_{m\rightarrow \infty}\EE\left[\sup_{t\in [0,T]}\left|\int^t_0 e^{(t-s)A}G^m_1(X^{\varepsilon}_s)d \bar W^{1,m}_s-\int^t_0 e^{(t-s)A}G_1(X^{\varepsilon}_s)dW^1_s\right|^p\right]=0;\\
&&\lim_{m\rightarrow \infty}\sup_{t\in [0,T]}\EE\left|\int^t_0 e^{(t-s)A/\vare}G^m_2(X^{\varepsilon}_s,Y^{\vare}_s)d \bar{W}^{2,m}_s-\int^t_0 e^{(t-s)A/\vare}G_2(X^{\varepsilon}_s,Y^{\vare}_s)dW^2_s\right|^p=0.
\end{eqnarray*}
Summarizing the above, we get for any $\vare>0$,
\begin{eqnarray*}
\lim_{m\rightarrow \infty}\EE\left(\sup_{t\in [0,T]}|X^{m,\vare}_t-X^{\vare}_t|^p\right)=0.
\end{eqnarray*}
The proof is complete.
\end{proof}

\begin{lemma} \label{L6.3}
For any $ (x,y)\in H\times H$, $m\in \mathbb{N}_{+}$ and $T>0$, there exist constants $C_{p,T}, C_{T}>0$ such that any  $\vare\in (0,1], \eta\in (0,1)$ and $0<s\leq t<T$,
\begin{eqnarray}
&&\sup_{m\geq 1}\left(\EE|X^{m,\vare}_t-X^{m,\vare}_s|^p\right)^{1/p}\leq C_T(t-s)^{\frac{\eta}{2}}s^{-\frac{\eta}{2}}(1+|x|+|y|),\quad\label{THXvare}\\
&&\sup_{m\geq 1}\left(\EE|Y^{m,\vare}_t-Y^{m,\vare}_s|^p\right)^{1/p}\leq C_T\left(\frac{t-s}{\vare}\right)^{\frac{\eta}{2}}s^{-\frac{\eta}{2}}(1+|x|+|y|).\label{THYvare}
\end{eqnarray}
\end{lemma}
\begin{proof}
By  \eref{P3} and Minkowski's inequality, we get for any $p>1$, $\eta\in (0,1)$ and $0<t\leq T$,
\begin{eqnarray}
\EE\|X^{m,\varepsilon}_t\|^p_{\eta}
\leq\!\!\!\!\!\!\!\!&&C_p\|e^{tA}x\|^p_{\eta}+C_p\EE\left(\int^t_0\|e^{(t-s)A}F^m_1(X^{m,\vare}_s,Y^{m,\vare}_s)\|_{\eta}ds\right)^p\nonumber\\
&&+C_p\EE\left\|\int^t_0 e^{(t-s)A}G^m_1(X^{m,\vare}_s)d\bar W^{1,m}_s\right\|^p_{\eta}\nonumber\\
\leq\!\!\!\!\!\!\!\!&&Ct^{-\frac{p\eta}{2}}|x|^p+C\left[\int^t_0(t-s)^{-\eta/2}\left[\EE(1+|X^{m,\varepsilon}_s|^p+|Y^{m,\varepsilon}_s|^p)\right]^{1/p}ds\right]^p\nonumber\\
&&+C_p\left[\int^t_0 (t-s)^{-\eta}\EE(1+|X^{m,\varepsilon}_s|^2)ds\right]^{p/2}\nonumber\\
\leq\!\!\!\!\!\!\!\!&&C_T t^{-\frac{p\eta}{2}}(1+|x|^p+|y|^p).\label{X_theta}
\end{eqnarray}

Note that
$$
X^{m,\varepsilon}_t=e^{(t-s)A}X^{m,\varepsilon}_s+\int^t_s e^{(t-r)A}F^m_1(X^{m,\varepsilon}_r, Y^{m,\varepsilon}_r)dr+\int^t_s e^{(t-r)A}G^m_1(X^{m,\varepsilon}_r)d\bar {W}^{1,m}_r
$$
It follows from \eref{P4} that for any $0<s\leq t\leq T$,
\begin{eqnarray*}
\EE|X^{m,\varepsilon}_t-X^{m,\varepsilon}_s|^p\leq\!\!\!\!\!\!\!\!&& C_p(t-s)^{\frac{p\eta}{2}}\EE\|X^{m,\varepsilon}_s\|^p_{\eta}+C_p\left(\int^t_s \left[\EE|F^m_1(X^{m,\varepsilon}_r, Y^{m,\varepsilon}_r)|^p\right]^{1/p}dr\right)^p\\
&&+C_p\left[\int^t_s \EE(1+|X^{m,\varepsilon}_r|^2)dr\right]^{p/2}\\
\leq\!\!\!\!\!\!\!\!&&C_{p,T}(t-s)^{\frac{p\eta}{2}}s^{-\frac{p\eta}{2}}(1+|x|^p+|y|^p)+C(t-s)^p(1+|x|^p+|y|^p)\\
&&+C(t-s)^{p/2}(1+|x|^p+|y|^p)\\
\leq\!\!\!\!\!\!\!\!&&C_{p,T}(t-s)^{\frac{p\eta}{2}}s^{-\frac{p\eta}{2}}(1+|x|^p+|y|^p).
\end{eqnarray*}

Similarly, by  \eref{P3} and Minkowski's inequality, we get for any $\eta\in (0,1)$ and $0<t\leq T$,
\begin{eqnarray}
\EE\|Y^{m,\varepsilon}_t\|^p_{\eta}
\leq\!\!\!\!\!\!\!\!&&C_p\|e^{tA/\vare}y\|^p_{\eta}+C_p\EE\left(\frac{1}{\vare}\int^t_0\|e^{\frac{(t-s)A}{\vare}}F^m_2(X^{m,\vare}_s,Y^{m,\vare}_s)\|_{\eta}ds\right)^p\nonumber\\
&&+C_p\EE\left\|\frac{1}{\sqrt{\vare}}\int^t_0 e^{\frac{(t-s)A}{\vare}}G^m_2(X^{m,\vare}_s,Y^{m,\varepsilon}_s)d\bar W^{2,m}_s\right\|^p_{\eta}\nonumber\\
\leq\!\!\!\!\!\!\!\!&&C_p\left(\frac{t}{\vare}\right)^{-\frac{p\eta}{2}}|y|^p+C\left[\frac{1}{\vare}\int^t_0\left(\frac{t-s}{\vare}\right)^{-\frac{\eta}{2}}e^{-\frac{(t-s)\lambda_1}{2\vare}}\left[\EE(1+|X^{m,\varepsilon}_s|^p+|Y^{m,\varepsilon}_s|^p)\right]^{1/p}ds\right]^p\nonumber\\
&&+C_p\left[\int^t_0 \frac{1}{\vare}\left(\frac{t-s}{\vare}\right)^{-\eta}e^{-\frac{(t-s)\lambda_1}{\vare}}\EE(1+|X^{m,\varepsilon}_s|^2)ds\right]^{p/2}\nonumber\\
\leq\!\!\!\!\!\!\!\!&&C_T t^{-\frac{p\eta}{2}}(1+|x|^p+|y|^p).\label{Y_theta}
\end{eqnarray}

Note that
\begin{eqnarray*}
Y^{m,\varepsilon}_t=\!\!\!\!\!\!\!\!&&e^{\frac{(t-s)A}{\vare}}Y^{m,\varepsilon}_s+\frac{1}{\vare}\int^t_s e^{\frac{(t-r)A}{\vare}}F^m_2(X^{m,\varepsilon}_r, Y^{m,\varepsilon}_r)dr\\
&&+\frac{1}{\sqrt{\vare}}\int^t_s e^{\frac{(t-r)A}{\vare}}G^m_2(X^{m,\varepsilon}_r, Y^{m,\varepsilon}_r)d\bar {W}^{2,m}_r.
\end{eqnarray*}
Thus by \eref{P4}, it follows that for any $0<s<t\leq T$,
\begin{eqnarray*}
\EE|Y^{m,\varepsilon}_t-Y^{m,\varepsilon}_s|^p\leq\!\!\!\!\!\!\!\!&&C_p\EE|e^{\frac{(t-s)A}{\vare}}Y^{m,\varepsilon}_s-Y^{m,\varepsilon}_s|^p +C_p\EE\left[\frac{1}{\vare}\int^t_s \left|e^{\frac{(t-r)A}{\vare}}F^m_2(X^{m,\varepsilon}_r, Y^{m,\varepsilon}_r)\right|dr\right]^p\\
&&+C_p\EE\left|\frac{1}{\sqrt{\vare}}\int^t_s e^{\frac{(t-r)A}{\vare}}G^m_2(X^{m,\varepsilon}_r, Y^{m,\varepsilon}_r)d\bar {W}^{2,m}_r\right|^p\\
\leq\!\!\!\!\!\!\!\!&&C_p\left(\frac{t-s}{\vare}\right)^{\frac{p\eta}{2}}\EE\|Y^{m,\varepsilon}_s\|^p_{\eta}+C_p\left[\frac{1}{\vare}\int^t_s e^{-\frac{(t-r)\lambda_1}{\vare}}
\left[\EE(1+|X^{m,\varepsilon}_r|^p+|Y^{m,\varepsilon}_r|^p)\right]^{1/p}dr\right]^p\nonumber\\
&&+C_p\left[\int^t_s \frac{1}{\vare}e^{-\frac{2(t-r)\lambda_1}{\vare}}\EE(1+|X^{m,\varepsilon}_r|^2)dr\right]^{p/2}\nonumber\\
\leq\!\!\!\!\!\!\!\!&&C_{p,T}\left(\frac{t-s}{\vare}\right)^{\frac{p\eta}{2}} s^{-\frac{p\eta}{2}}(1+|x|^p+|y|^p)+C_{p,T}\left[1-e^{-\frac{(t-s)\lambda_1}{\vare}}\right]^p(1+|x|^p+|y|^p)\\
&&+C_{p,T}\left[1-e^{-\frac{2(t-s)\lambda_1}{\vare}}\right]^{p/2}(1+|x|^p+|y|^p)\\
\leq\!\!\!\!\!\!\!\!&&C_{p,T}\left(\frac{t-s}{\vare}\right)^{\frac{p\eta}{2}}s^{-\frac{p\eta}{2}}(1+|x|^p+|y|^p).
\end{eqnarray*}
The proof is complete.
\end{proof}

\begin{lemma}\label{L6.4}
For any $x\in H^{\eta}, y\in H$ with $\eta\in(0,1)$,  $0<t\leq T$ and $p\geq 2$, there exists a constant $C_{T}$ such that for any $\vare\in (0,1]$,
\begin{align}
\left[\mathbb{E}\|X^{m,\vare}_t\|^p_2\right]^{\frac{1}{p}}
\leq C_T t^{-1+\frac{\eta}{2}}\vare^{-\frac{\eta}{2}}(1+\|x\|_{\eta}+|y|). \label{Xvare2}
\end{align}
\end{lemma}
\begin{proof}
Note that for  any $t>0$, we have
\begin{eqnarray*}
X^{m,\vare}_t=\!\!\!\!\!\!\!\!&&e^{tA}x+\int_{0}^{t}e^{(t-s)A}F^m_1(X^{m,\vare}_{t},Y^{m,\vare}_{t})ds+\int_{0}^{t}e^{(t-s)A}\left[F^m_1(X^{m,\vare}_s,Y^{m,\vare}_{s})-F^m_1(X^{m,\vare}_{t},Y^{m,\vare}_{t})\right]ds  \nonumber\\
\!\!\!\!\!\!\!\!&&+\int_{0}^{t}e^{(t-s)A}G^m_1(X^{m,\vare}_s,Y^{m,\vare}_{s})d\bar{W}^{1,m}_s:=\sum^4_{i=1}I_i.
\end{eqnarray*}

For the term $I_{1}$, using \eqref{P3}, for any $\eta\in (0,1)$ we have
\begin{align} \label{ABarX1}
\|e^{tA}x\|_2
\leq C t^{-1+\frac{\eta}{2}}\|x\|_{\eta}.
\end{align}

For the term $I_{2}$, we have
\begin{eqnarray*} \label{ABarX4}
\left[\mathbb{E}\|I_{2}\|^p_2\right]^{1/p}=\!\!\!\!\!\!\!\!&&\left(\mathbb{E}\left|(e^{tA}-I)F^m_1(X^{m,\vare}_{t},Y^{m,\vare}_{t})\right|^p\right)^{1/p}\\
\leq\!\!\!\!\!\!\!\!&&C\left[1+\left(\mathbb{E}|X^{m,\vare}_{t}|^p\right)^{1/p}+\left(\mathbb{E}|Y^{m,\vare}_{t}|^p\right)^{1/p}\right]\\
\leq\!\!\!\!\!\!\!\!&&C_T(1+|x|+|y| ).
\end{eqnarray*}

For the term $I_{3}$, using Minkowski's inequality and Lemma \ref{L6.3}, we obtain
\begin{eqnarray} \label{ABarX5}
\left[\mathbb{E} \| I_{3}\|^p_2\right]^{1/p}
\leq\!\!\!\!\!\!\!\!&&C\int_{0}^{t}\frac{1}{t-s}\left[\mathbb{E}\left|F^m_1(X^{m,\vare}_s,Y^{m,\vare}_{s})-F^m_1(X^{m,\vare}_{t},Y^{m,\vare}_{t})\right|^p\right]^{\frac{1}{p}}ds \nonumber\\
\leq\!\!\!\!\!\!\!\!&&
C\int_{0}^{t}\frac{1}{t-s}\left[\mathbb{E}\left|X^{m,\vare}_s-X^{m,\vare}_{t}\right|^p+\EE\left|Y^{m,\vare}_{s}-Y^{m,\vare}_{t}\right|^p\right]^{\frac{1}{p}}ds \nonumber\\
\leq\!\!\!\!\!\!\!\!&&C(1+|x|+|y|)\int_{0}^{t}\frac{1}{t-s}(t-s)^{\frac{\eta}{2}}s^{-\frac{\eta}{2}}ds\nonumber\\
&&+C(1+|x|+|y|)\int_{0}^{t}\frac{1}{t-s}\left(\frac{t-s}{\vare}\right)^{\frac{\eta}{2}}s^{-\frac{\eta}{2}}ds\nonumber\\
\leq\!\!\!\!\!\!\!\!&&C\vare^{-\eta/2}(1+|x|+|y|) .
\end{eqnarray}

For the term $I_{4}$, by condition \ref{BOG1}, we easily have
\begin{eqnarray}
\left[\mathbb{E} \| I_{4}\|^p_2\right]^{1/p}\leq\!\!\!\!\!\!\!\!&&C\left[\EE\left(\int^t_0|(-A)e^{(t-s)A}G^m_1(X^{m,\vare}_s)|^2_{HS}ds\right)^{p/2}\right]^{1/p}\nonumber\\
\leq\!\!\!\!\!\!\!\!&&C\left[\int^t_0 (t-s)^{-\gamma}\left(\EE|(-A)^{1-\gamma/2}G^m_1(X^{m,\vare}_s)|^p_{HS}\right)^{2/p}ds\right]^{1/2}\nonumber\\
\leq\!\!\!\!\!\!\!\!&&C\left[\int^t_0 (t-s)^{-\gamma}(1+\EE|X^{m,\vare}_s|^p)^{2/p}\right]^{1/2}\nonumber\\
\leq\!\!\!\!\!\!\!\!&&C(1+|x|+|y|).\label{ABarX6}
\end{eqnarray}
Combining \eref{ABarX1}-\eref{ABarX6} yields the desired result.
\end{proof}
\begin{remark}\label{Re6.5}
If condition \eref{BF} holds, the result \eref{Xvare2} can be replaced by the following:
\begin{align}
\left[\mathbb{E}\|X^{m,\vare}_t\|^p_2\right]^{\frac{1}{p}}
\leq C_T\left( t^{-1+\frac{\eta}{2}}+t^{\frac{\alpha-\beta}{2}}\right)(1+\|x\|_{\eta}+|y|). \label{Xvare2A}
\end{align}
In fact, note that by \eref{X_theta} and \eref{Y_theta}, we have
\begin{eqnarray*}
\left[\EE\left\|\int_{0}^{t}e^{(t-s)A}F(X^{m,\vare}_s,Y^{m,\vare}_{s})ds\right\|^p_2\right]^{1/p}
\leq \!\!\!\!\!\!\!\!&&C\int_{0}^{t}\left[\EE\left\|e^{(t-s)A}F(X^{m,\vare}_s,Y^{m,\vare}_{s})\right\|^p_2\right]^{1/p}ds\nonumber\\
\leq \!\!\!\!\!\!\!\!&&C\int_{0}^{t}(t-s)^{-1+\alpha/2}\left[\EE \left\|F(X^{m,\vare}_s,Y^{m,\vare}_{s})\right\|^p_{\alpha}\right]^{1/p}ds\nonumber\\
\leq \!\!\!\!\!\!\!\!&&C\int_{0}^{t}(t-s)^{-1+\alpha/2}\left[\EE (1+\|X^{m,\vare}_s\|^p_{\beta}+\|Y^{m,\vare}_{s}\|^p_{\beta})\right]^{1/p}ds\nonumber\\
\leq \!\!\!\!\!\!\!\!&&C(1+|x|+|y|)\int_{0}^{t}(t-s)^{-1+\alpha/2} s^{-\beta/2}ds\nonumber\\
\leq \!\!\!\!\!\!\!\!&&C(1+|x|+|y|)t^{\frac{\alpha-\beta}{2}}.
\end{eqnarray*}
Then by a similar argument in Lemma \ref{L6.4}, it is easy to see \eref{Xvare2A} holds.
\end{remark}

\vspace{0.3cm}
Recall the approximate equation \eref{Ga 1.3} to the averaged equation  \eref{1.3}. Note that $\bar{X}^m$ is not the Galerkin approximation of $\bar{X}$, hence we have to check its approximation carefully.
\begin{lemma} \label{GA2}
For any $x\in H$, $T>0$ and $p\geq1$, it holds that
\begin{align}
\lim_{m\rightarrow \infty}\EE\left(\sup_{t\in[0,T]}|\bar{X}^{m}_t-\bar{X}_t|^p\right)=0. \label{FA2}
\end{align}
\end{lemma}
\begin{proof}
It is easy to see that for any $t\geq 0$,
\begin{eqnarray*}
\bar{X}^{m}_t-\bar{X}_t=\!\!\!\!\!\!\!\!&&e^{tA}(x^m-x)+\!\!\int^{t}_{0}\!\!e^{(t-s)A}(\pi_m-I)\bar{F}_1(\bar{X}_s)ds+\!\!\int^{t}_{0}\!\!e^{(t-s)A}\!\!\left[\bar{F}^{m}_1(\bar{X}^{m}_s)-\pi_m \bar{F}_1(\bar{X}_s)\right]ds\nonumber\\
&&+\left[\int^t_0 e^{(t-s)A}G^m_1(\bar{X}^{m}_s)d \bar W^{1,m}_s-\int^t_0 e^{(t-s)A}G^m_1(\bar{X}_s)dW^{1,m}_s\right]\\
&&+\left[\int^t_0 e^{(t-s)A}G^m_1(\bar{X}_s)d \bar W^{1,m}_s-\int^t_0 e^{(t-s)A}G_1(\bar{X}_s)dW^1_s\right].
\end{eqnarray*}
Then for any $T>0$ and $p\geq1$, we have
\begin{eqnarray*}
\EE\left(\sup_{t\in [0,T]}|\bar{X}^{m}_t-\bar{X}_t|^p\right)
\leq\!\!\!\!\!\!\!\!&&C_p|x^m-x|^p+C_{p,T}\int^{T}_{0}\EE|(\pi_m-I)\bar{F}_1(\bar{X}_s)|^pds\\
&&\!\!\!\!\!\!\!\!\!\!\!\!\!\!\!\!\!\!\!\!\!\!\!\!+C_{p,T}\int^T_0 \EE|\bar{X}^{m}_t-\bar{X}_t|^pdt+C_{p,T}\EE\int^{T}_{0}|\bar{F}^{m}_1(\bar{X}_s)-\pi_m \bar{F}_1(\bar{X}_s)|^pds\\
&&\!\!\!\!\!\!\!\!\!\!\!\!\!\!\!\!\!\!\!\!\!\!\!\!+C_p\EE\left(\sup_{t\in [0,T]}\left|\int^t_0 e^{(t-s)A}G^m_1(\bar{X}_s)d \bar W^{1,m}_s-\int^t_0 e^{(t-s)A}G_1(\bar{X}_s)dW^1_s\right|^p\right).
\end{eqnarray*}
By Gronwall's inequality, we get
\begin{eqnarray*}
\EE\left(\sup_{t\in [0,T]}|\bar{X}^{m}_t-\bar{X}_t|^p\right)\leq\!\!\!\!\!\!\!\!&&C_{p,T}|x^m-x|^p+C_{p,T}\int^{T}_{0}\EE|(\pi_m-I)\bar{F}_1(\bar{X}_s)|^pds\\
&&\!\!\!\!\!\!\!\!\!\!\!\!\!\!\!\!\!\!\!\!\!\!\!\!+C_{p,T}\EE\int^{T}_{0}|\bar{F}^{m}_1(\bar{X}_s)-\pi_m \bar{F}_1(\bar{X}_s)|^p ds\nonumber\\
&&\!\!\!\!\!\!\!\!\!\!\!\!\!\!\!\!\!\!\!\!\!\!\!\!+C_{p,T}\EE\left(\sup_{t\in [0,T]}\left|\int^t_0 e^{(t-s)A}G^m_1(\bar{X}_s)d \bar W^{1,m}_s-\int^t_0 e^{(t-s)A}G_1(\bar{X}_s)dW^1_s\right|^p\right).
\end{eqnarray*}
By the a priori estimate of $\bar{X}_t$ and the dominated convergence theorem,
\begin{eqnarray}
&&\lim_{m\rightarrow \infty}\int^{T}_{0}\EE|(\pi_m-I)\bar{F}_1(\bar{X}_s)|^pds=0,\label{AT2.1}\\
&&\lim_{m\rightarrow \infty}\EE\left(\sup_{t\in [0,T]}\left|\int^t_0 e^{(t-s)A}G^m_1(\bar{X}_s)d \bar W^{1,m}_s-\int^t_0 e^{(t-s)A}G_1(\bar{X}_s)dW^1_s\right|^p\right)=0.\label{AT2.2}
\end{eqnarray}
Thus if we can prove for any $x\in H$,
\begin{eqnarray}
\lim_{m\rightarrow \infty}|\bar{F}^{m}_1(x)-\pi_m \bar{F}_1(x)|=0.\label{AT2.3}
\end{eqnarray}
Then by dominated convergence theorem, we get
\begin{eqnarray}
\lim_{m\rightarrow \infty}\EE\int^{T}_{0}|\bar{F}^{m}_1(\bar{X}_s)-\pi_m \bar{F}_1(\bar{X}_s)|^p ds=0.\label{AT2.4}
\end{eqnarray}
Hence, \eref{FA2} holds by combining \eref{AT2.1}, \eref{AT2.2} and \eref{AT2.4}.

Now we prove \eqref{AT2.3}. For any $t>0$, we have
\begin{eqnarray*}
\EE|Y^{x,0,m}_t-Y^{x,0}_t|^2\leq\!\!\!\!\!\!\!\!&&C\int^{t}_{0}\!\!\!\EE|(\pi_m-I)F_2(x, Y^{x,0}_s)|^2ds\!\!+\!C\int^{t}_{0}\!\!\!\EE|F^{m}_2(x,Y^{x,0,m}_s)-F^m_2(x,Y^{x,0}_s)|^2ds \nonumber\\
&&+C\EE\left|\int^t_0 e^{(t-s)A}G^m_2(x,Y^{x,0,m}_s)d \bar W^{2,m}_s-\int^t_0 e^{(t-s)A}G^m_2(x,Y^{x,0}_s)d\bar{W}^{2,m}_s\right|^2\nonumber\\
&&+C\EE\left|\int^t_0 e^{(t-s)A}G^m_2(x,Y^{x,0}_s)d \bar W^{2,m}_s-\int^t_0 e^{(t-s)A}G_2(x,Y^{x,0}_s)dW^2_s\right|^2\nonumber\\
\leq\!\!\!\!\!\!\!\!&&C\int^{t}_{0}\EE|(\pi_m-I)F_2(x, Y^{x,0}_s)|^2ds+C\EE\int^{t}_{0}|Y^{x,0,m}_s-Y^{x,0}_s|^2ds \nonumber\\
&&+C\EE\left|\int^t_0 e^{(t-s)A}G^m_2(x,Y^{x,0}_s)d \bar W^{2,m}_s-\int^t_0 e^{(t-s)A}G_2(x,Y^{x,0}_s)dW^2_s\right|^2.\nonumber\\
\end{eqnarray*}
By Gronwall's inequality, it follows
\begin{eqnarray*}
\EE|Y^{x,0,m}_t-Y^{x,0}_t|^2
\leq\!\!\!\!\!\!\!\!&&Ce^{Ct}\int^{t}_{0}\EE|(\pi_m-I)F_2(x, Y^{x,0}_s)|^2ds\\
&&+Ce^{Ct}\EE\left|\int^t_0 e^{(t-s)A}G^m_2(x,Y^{x,0}_s)d \bar W^{2,m}_s-\int^t_0 e^{(t-s)A}G_2(x,Y^{x,0}_s)dW^2_s\right|^2.
\end{eqnarray*}
As a consequence, it is easy to see for any $t\geq 0$,
\begin{eqnarray*}
\lim_{m\rightarrow \infty}\EE|Y^{x,0,m}_t-Y^{x,0}_t|^2=0.
\end{eqnarray*}
By \eref{ergodicity1}, we have for any $t>0$,
\begin{eqnarray*}
|\bar{F}^{m}_1(x)-\pi_m \bar{F}_1(x)|\leq\!\!\!\!\!\!\!\!&&|\bar{F}^{m}_1(x)-\EE\pi_m F_1(x,Y^{x,0,m}_t)|+|\EE\pi_m F_1(x,Y^{x,0}_t)-\pi_m \bar{F}_1(x)| \nonumber\\
&&+|\EE\pi_m F_1(x,Y^{x,0,m}_t)-\EE\pi_m F_1(x,Y^{x,0}_t)|\nonumber\\
\leq\!\!\!\!\!\!\!\!&&Ce^{-\frac{\theta t}{4}}(1+|x|)+C\EE|Y^{x,0,m}_t-Y^{x,0}_t|.
\end{eqnarray*}
Then by taking $m\rightarrow \infty$ firstly, then $t\rightarrow\infty$, we arrive at \eref{AT2.3}.
The proof is complete.

\end{proof}

\textbf{Acknowledgment}. We would like to thank Professor Renming Song for useful discussion. This work is supported by the National Natural Science Foundation of China (11771187, 11931004, 12090011) and the Priority Academic Program Development of Jiangsu Higher Education Institutions.

\end{document}